\documentclass[12pt,pdftex,a4paper]{amsart}

\selectfont

\usepackage{wrapfig}

\usepackage{qtree}
\usepackage{forest}

\usepackage{caption}

\usepackage{etoolbox}
\usepackage{epic,eepic,ecltree}
\usepackage{graphicx}
\usepackage{tikz}
\usetikzlibrary{decorations.text}
\usepackage{amssymb,color, euscript, enumerate}
\usepackage{amsthm}
\usepackage{amsmath}
\usepackage{braket}
\usepackage{amscd}
\usepackage{txfonts}
\usepackage{comment}
\usepackage{bm}
\usepackage{amscd}
\numberwithin{equation}{section}

\makeatletter
\let\old@tocline\@tocline
\let\section@tocline\@tocline
\newcommand{\subsection@dotsep}{4.5}
\newcommand{\subsubsection@dotsep}{4.5}
\patchcmd{\@tocline}
  {\hfil}
  {\nobreak
     \leaders\hbox{$\m@th
        \mkern \subsection@dotsep mu\hbox{.}\mkern \subsection@dotsep mu$}\hfill
     \nobreak}{}{}
\let\subsection@tocline\@tocline
\let\@tocline\old@tocline

\patchcmd{\@tocline}
  {\hfil}
  {\nobreak
     \leaders\hbox{$\m@th
        \mkern \subsubsection@dotsep mu\hbox{.}\mkern \subsubsection@dotsep mu$}\hfill
     \nobreak}{}{}
\let\subsubsection@tocline\@tocline
\let\@tocline\old@tocline

\let\old@l@subsection\l@subsection
\let\old@l@subsubsection\l@subsubsection

\def\@tocwriteb#1#2#3{%
  \begingroup
    \@xp\def\csname #2@tocline\endcsname##1##2##3##4##5##6{%
      \ifnum##1>\c@tocdepth
      \else \sbox\z@{##5\let\indentlabel\@tochangmeasure##6}\fi}%
    \csname l@#2\endcsname{#1{\csname#2name\endcsname}{\@secnumber}{}}%
  \endgroup
  \addcontentsline{toc}{#2}%
    {\protect#1{\csname#2name\endcsname}{\@secnumber}{#3}}}%

\newlength{\@tocsectionindent}
\newlength{\@tocsubsectionindent}
\newlength{\@tocsubsubsectionindent}
\newlength{\@tocsectionnumwidth}
\newlength{\@tocsubsectionnumwidth}
\newlength{\@tocsubsubsectionnumwidth}
\newcommand{\settocsectionnumwidth}[1]{\setlength{\@tocsectionnumwidth}{#1}}
\newcommand{\settocsubsectionnumwidth}[1]{\setlength{\@tocsubsectionnumwidth}{#1}}
\newcommand{\settocsubsubsectionnumwidth}[1]{\setlength{\@tocsubsubsectionnumwidth}{#1}}
\newcommand{\settocsectionindent}[1]{\setlength{\@tocsectionindent}{#1}}
\newcommand{\settocsubsectionindent}[1]{\setlength{\@tocsubsectionindent}{#1}}
\newcommand{\settocsubsubsectionindent}[1]{\setlength{\@tocsubsubsectionindent}{#1}}

\renewcommand{\l@section}{\section@tocline{1}{\@tocsectionvskip}{\@tocsectionindent}{}{\@tocsectionformat}}%
\renewcommand{\l@subsection}{\subsection@tocline{2}{\@tocsubsectionvskip}{\@tocsubsectionindent}{}{\@tocsubsectionformat}}%
\renewcommand{\l@subsubsection}{\subsubsection@tocline{3}{\@tocsubsubsectionvskip}{\@tocsubsubsectionindent}{}{\@tocsubsubsectionformat}}%
\newcommand{\@tocsectionformat}{}
\newcommand{\@tocsubsectionformat}{}
\newcommand{\@tocsubsubsectionformat}{}
\expandafter\def\csname toc@1format\endcsname{\@tocsectionformat}
\expandafter\def\csname toc@2format\endcsname{\@tocsubsectionformat}
\expandafter\def\csname toc@3format\endcsname{\@tocsubsubsectionformat}
\newcommand{\settocsectionformat}[1]{\renewcommand{\@tocsectionformat}{#1}}
\newcommand{\settocsubsectionformat}[1]{\renewcommand{\@tocsubsectionformat}{#1}}
\newcommand{\settocsubsubsectionformat}[1]{\renewcommand{\@tocsubsubsectionformat}{#1}}
\newlength{\@tocsectionvskip}
\newcommand{\settocsectionvskip}[1]{\setlength{\@tocsectionvskip}{#1}}
\newlength{\@tocsubsectionvskip}
\newcommand{\settocsubsectionvskip}[1]{\setlength{\@tocsubsectionvskip}{#1}}
\newlength{\@tocsubsubsectionvskip}
\newcommand{\settocsubsubsectionvskip}[1]{\setlength{\@tocsubsubsectionvskip}{#1}}

\patchcmd{\tocsection}{\indentlabel}{\makebox[\@tocsectionnumwidth][l]}{}{}
\patchcmd{\tocsubsection}{\indentlabel}{\makebox[\@tocsubsectionnumwidth][l]}{}{}
\patchcmd{\tocsubsubsection}{\indentlabel}{\makebox[\@tocsubsubsectionnumwidth][l]}{}{}

\newcommand{\@sectypepnumformat}{}
\renewcommand{\contentsline}[1]{%
  \expandafter\let\expandafter\@sectypepnumformat\csname @toc#1pnumformat\endcsname%
  \csname l@#1\endcsname}
\newcommand{\@tocsectionpnumformat}{}
\newcommand{\@tocsubsectionpnumformat}{}
\newcommand{\@tocsubsubsectionpnumformat}{}
\newcommand{\setsectionpnumformat}[1]{\renewcommand{\@tocsectionpnumformat}{#1}}
\newcommand{\setsubsectionpnumformat}[1]{\renewcommand{\@tocsubsectionpnumformat}{#1}}
\newcommand{\setsubsubsectionpnumformat}[1]{\renewcommand{\@tocsubsubsectionpnumformat}{#1}}
\renewcommand{\@tocpagenum}[1]{%
  \hfill {\mdseries\@sectypepnumformat #1}}

\let\oldappendix\appendix
\renewcommand{\appendix}{%
  \leavevmode\oldappendix%
  \addtocontents{toc}{%
    \protect\settowidth{\protect\@tocsectionnumwidth}{\protect\@tocsectionformat\sectionname\space}%
    \protect\addtolength{\protect\@tocsectionnumwidth}{2em}}%
}
\makeatother

\makeatletter
\settocsectionnumwidth{2em}
\settocsubsectionnumwidth{2.5em}
\settocsubsubsectionnumwidth{3em}
\settocsectionindent{1pc}%
\settocsubsectionindent{\dimexpr\@tocsectionindent+\@tocsectionnumwidth}%
\settocsubsubsectionindent{\dimexpr\@tocsubsectionindent+\@tocsubsectionnumwidth}%
\makeatother

\settocsectionvskip{10pt}
\settocsubsectionvskip{0pt}
\settocsubsubsectionvskip{0pt}
    
\settocsectionformat{\bfseries}
\settocsubsectionformat{\mdseries}
\settocsubsubsectionformat{\mdseries}
\setsectionpnumformat{\bfseries}
\setsubsectionpnumformat{\mdseries}
\setsubsubsectionpnumformat{\mdseries}

\let\oldtableofcontents\tableofcontents
\renewcommand{\tableofcontents}{%
  \vspace*{-\linespacing}% Default gap to top of CONTENTS is \linespacing.
  \oldtableofcontents}

\catcode`,\active

\catcode`\,12

\newcommand{\dD}{\overline{D}}

\newcommand{\Fr}{F^{\otimes r}}
\newcommand{\ar}{{a_{[r]}}}
\newcommand{\zr}{{z_{[r]}}}

\newcommand{\M}{\overline{M}}

\newcommand{\comp}{{\mathrm{glue}}}

\newcommand{\Hom}{{\mathrm{Hom}}}

\newcommand{\an}{{\mathrm{an}}}

\newcommand{\bH}{\mathbb H}

\newcommand{\modu}{{\text{-mod}}}
\newcommand{\Vmodu}{{\underline{V{\text{-mod}}}}}

\newcommand{\Vmodf}{{\underline{V\text{-mod}}_{f.g.}}}

\newcommand{\Vmodft}{{\underline{V\text{-mod}}_{f.g.}^2}}
\newcommand{\Vmodfo}{{\underline{V\text{-mod}}_{f.g.}^\op}}

\newcommand{\CPaB}{{\underline{\text{PaB}}}}

\newcommand{\PaPB}{{\underline{\text{PaPB}}}}

\newcommand{\Z}{\mathbb{Z}}
\newcommand{\R}{\mathbb{R}}
\newcommand{\C}{\mathbb{C}}
\newcommand{\Q}{\mathbb{Q}}
\newcommand{\HH}{\overline{\mathbb{H}}}
\newcommand{\uH}{\mathbb{H}}

\newcommand{\Y}{\mathcal{Y}}
\newcommand{\mO}{\mathrm{O}}
\newcommand{\Xr}{{X_r(\C)}}

\newcommand{\Or}{\mathrm{O}_{\Xr}}
\newcommand{\Om}{\Omega}
\newcommand{\Ort}{\mathrm{O}_{\Xr/\C}}
\newcommand{\Mr}{{M_{[0;r]}}}

\newcommand{\mr}{{m_{[0;r]}}}

\newcommand{\Dr}{{\mathrm{D}_{\Xr}}}

\newcommand{\Leaf}{\text{Leaf}}

\newcommand{\alg}{\text{alg}}

\newcommand{\conv}{\text{conv}}

\newcommand{\dz}{{\frac{d}{dz}}}

\newcommand{\va}{\bm{1}}
\newcommand{\vac}{\bm{1}}

\newcommand{\id}{{\mathrm{id}}}

\newcommand{\z}{{\bar{z}}}

\newcommand{\h}{{\bar{h}}}
\newcommand{\p}{{\bar{p}}}

\newcommand{\uz}{\underline{z}}
\newcommand{\uzr}{{\underline{z}}_{[r]}}

\newcommand{\pa}{{\partial}}

\newcommand{\Vect}{{\underline{\text{Vect}}_\C}}

\newcommand{\Ms}{{M_{[0;s]}}}

\newcommand{\al}{\alpha}
\newcommand{\ep}{\epsilon}
\newcommand{\be}{\beta}
\newcommand{\ga}{\gamma}
\newcommand{\ze}{\zeta}

\newcommand{\om}{\omega}
\newcommand{\la}{\lambda}

\newcommand{\omb}{{\bar{\omega}}}
\newcommand{\si}{\sigma}

\newcommand{\Log}{\mathrm{Log}}
\newcommand{\Arg}{\mathrm{Arg}}

\newcommand{\lat}{{\mathrm{lat}}}

\newcommand{\g}{{\mathfrak{g}}}
\newcommand{\fg}{{\mathfrak{g}}}

\newcommand{\CB}{{\mathcal{CB}}}

\newcommand{\Ld}{{\overline{L}}}
\newcommand{\D}{{\mathbb{D}}}

\newcommand{\GCor}{{\mathrm{GCor}}}

\newcommand{\Aut}{\mathrm{Aut}\,}
\newcommand{\tw}{{{I\hspace{-.1em}I}_{n,n}}}

\newcommand{\tww}{{{I\hspace{-.1em}I}_{1,1}}}

\newcommand{\End}{\mathrm{End}}
\newcommand{\Func}{\mathrm{Func}}

\newcommand{\Endp}{{\mathcal{PE}\mathrm{nd}}}
\newcommand{\cat}{{\underline{\mathrm{Cat}}_\mathbb{C}}}

\newcommand{\bdy}{\mathrm{bdy}}
\newcommand{\bulk}{\mathrm{bulk}}

\newcommand{\Vir}{{\mathrm{Vir}}}
\newcommand{\Tr}{{\mathcal{T}}}

\newcommand{\cut}{\mathrm{cut}}

\newcommand{\cop}{{\mathrm{cop}}}
\newcommand{\op}{{\mathrm{op}}}
\newcommand{\rev}{{\mathrm{rev}}}

\newtheorem{thm}{Theorem}[section]
\newtheorem{dfn}[thm]{Definition}
\newtheorem{lem}[thm]{Lemma}
\newtheorem{prop}[thm]{Proposition}
\newtheorem{cor}[thm]{Corollary}
\newtheorem{rem}[thm]{Remark}

\newtheorem{mainthm}{Main Theorem}

\begin{document}

\begin{center}
{{\LARGE \bf 
Convergence and operadic compatibility of bulk and boundary OPEs in two-dimensional conformal field theory}
} \par \bigskip

\renewcommand*{\thefootnote}{\fnsymbol{footnote}}
{\normalsize
Yuto Moriwaki \footnote{email: \texttt{moriwaki.yuto (at) gmail.com}}
}
\par \bigskip
{\footnotesize Interdisciplinary Theoretical and Mathematical Science Program (iTHEMS)\\
Wako, Saitama 351-0198, Japan}
\par \bigskip
\end{center}

\noindent

\begin{center}
\textbf{\large Abstract}
\end{center}

We prove convergence and compatibility of iterated bulk and boundary operator
product expansions (OPEs) in two-dimensional conformal field theory with locally
$C_1$-cofinite chiral symmetry.
For each tree, we give an explicit domain of convergence for the corresponding
iterated OPE. These local expansions glue to
single-valued real analytic functions on the configuration spaces,
%of the plane and of the upper half-plane,
which are the correlation functions of the theory.
The proof uses an action of
the parenthesized permutation-braid operad on $C_1$-cofinite module categories
of a vertex operator algebra. 
This operad models the fundamental groupoid of the two-dimensional Swiss-cheese operad, 
and under this action the operadic generators correspond to the genus-zero
bootstrap equations of boundary CFT.

\tableofcontents

\begin{center}
\textbf{\large Introduction}
\end{center}

Conformal field theory is described by an algebraic structure known as the \emph{operator product expansion} (OPE) \cite{Poly1,BPZ,FMS}.  In two dimensions, the OPE is a family of products depending on a point of the punctured complex plane,
\begin{align}
  a \cdot_z b
  = \sum_{\substack{r,s\in\R\\r-s\in\Z}} a(r,s)b\, z^{-r-1}\bar z^{-s-1},
  \qquad z \in \mathbb C^\times.
  \label{intro_OPE}
\end{align}
%It is important to note that an OPE is not merely a single binary operation.  Once the product is iterated, each order and choice of parentheses produces a distinct formal expansion.  
The composition of these products is generally encoded by binary trees.  For instance,
\begin{align}
  (a_1 \cdot_{z_{12}} a_2) \cdot_{z_{23}} a_3,
  \qquad
  a_1 \cdot_{z_{13}}(a_2 \cdot_{z_{23}} a_3)
  \label{intro_OPE_ex}
\end{align}
correspond to the binary trees
\begin{tikzpicture}[baseline=-.55ex,scale=.45]
  \coordinate (r)  at (0,2);
  \coordinate (v)  at (-.8,1);
  \coordinate (l1) at (-1.4,0);
  \coordinate (l2) at (-.2,0);
  \coordinate (l3) at (.8,0);

  \draw (r)--(v);
  \draw (r)--(l3);
  \draw (v)--(l1);
  \draw (v)--(l2);

  \node[below=1pt] at (l1) {$1$};
  \node[below=1pt] at (l2) {$2$};
  \node[below=1pt] at (l3) {$3$};
\end{tikzpicture}
\quad and \quad
\begin{tikzpicture}[baseline=-.55ex,scale=.45]
  \coordinate (r)  at (0,2);
  \coordinate (v)  at (.8,1);
  \coordinate (l1) at (-.8,0);
  \coordinate (l2) at (.2,0);
  \coordinate (l3) at (1.4,0);

  \draw (r)--(l1);
  \draw (r)--(v);
  \draw (v)--(l2);
  \draw (v)--(l3);

  \node[below=1pt] at (l1) {$1$};
  \node[below=1pt] at (l2) {$2$};
  \node[below=1pt] at (l3) {$3$};
\end{tikzpicture}.

These iterated products are formal power series, and the fundamental questions are to determine their domains of convergence in the configuration space
%From a mathematical perspective, these iterated products are formal power series.
%The fundamental questions are to determine their domains of convergence in the
%configuration space
\begin{align*}
X_r(\C) = \{(z_1,\dots,z_r) \in \C^r \mid z_i\neq z_j \}
\end{align*}
and to understand how the expansions associated with different orders and
parenthesizations are compatible with one another.

For a chiral conformal field theory in which the OPE \eqref{intro_OPE} is holomorphic in $z$, the theory can be described by a \emph{vertex operator algebra} (VOA) \cite{B1,G,FLM}.  With the convention $z_{ij}=z_i-z_j$, the two expressions in \eqref{intro_OPE_ex} are known to converge absolutely on
\[
{U}_{(12)3}^c
  = \{(z_1,z_2,z_3)\in X_3(\C) \mid |z_1-z_2|<|z_2-z_3|\},
\]
and
\[
{U}_{1(23)}^c
  = \{(z_1,z_2,z_3)\in X_3(\C) \mid |z_2-z_3|<|z_1-z_3|\},
\]
respectively.
On the intersection ${U}_{(12)3}^c\cap {U}_{1(23)}^c$, the two holomorphic functions agree, and
they analytically continue to a single-valued holomorphic function on
\(X_3(\C)\).
% and admit analytic continuations to a single holomorphic function on $X_3(\C)$.
In fact, just as ordinary commutative algebras are characterized by associativity and commutativity of their product, vertex operator algebras are characterized by the corresponding analytic identities for OPEs \cite{LL,FB}.
When the OPE is real analytic rather than holomorphic, imposing the same associativity and commutativity identities in the sense of analytic continuation leads to the notion of a \emph{full vertex operator algebra} (full VOA) \cite{M1}.
%, which is closely related to the full field algebras of Huang-Kong
%\cite{HK1}.

This leads to the problem of determining the domains of convergence for iterated full
VOA products and proving their compatibility for arbitrary orders and parenthesizations.
%The purpose of this paper is to
%determine the domains of convergence of iterated OPEs and to prove their compatibility for arbitrary orders and parenthesizations.
More precisely, let $\Tr_r$ denote the set of all binary trees with leaves labeled by $\{1,\dots,r\}$, and let $C^\om(X_r(\C),\C)$ denote the space of complex-valued real analytic functions on $X_r(\C)$.
To each $A\in\Tr_r$, we associate an open subset $U_A^c\subset X_r(\C)$.
%  In this paper, we introduce an open subset ${U}_A^c \subset X_r(\C)$ corresponding to a binary tree $A \in \Tr_r$.

The first main result is stated as follows  (see Theorem \ref{thm_bulk}).  Let $F$ be a full vertex
operator algebra which is locally $C_1$-cofinite as a module of its canonical holomorphic
and anti-holomorphic sub-VOAs $\ker L(-1)\otimes\ker \Ld(-1)$. 
Then for every binary tree $A\in\Tr_r$ the corresponding iterated OPE
converges absolutely and locally uniformly on the domain ${U}_A^c\subset X_r(\C)$.
Moreover, there is a sequence of single-valued real analytic functions
\begin{align}
\left\{C_r:F^\vee\otimes F^{\otimes r}\longrightarrow C^\om(X_r(\C),\C)\right\}_{r\geq 0}
 \label{eq_intro_bulk_ope0}
\end{align}
such that, for all $A\in \Tr_r$, $u\in F^\vee = \bigoplus_{h,\h\in\R}F_{h,\h}^*$ and $a_1,\ldots,a_r\in F$,
\begin{align}
  C_r(u;a_1,\ldots,a_r)|_{{U}_A^c}
  =
 \langle u,\, A\text{-shaped OPEs of }a_1,\ldots,a_r\rangle.
 \label{eq_intro_bulk_ope1}
\end{align}
Equivalently, the local analytic functions defined by the tree-wise OPEs glue to a single-valued real analytic function on $X_r(\C)$.
In particular, since a vertex operator algebra is $C_1$-cofinite as a module over itself, this also
yields explicit convergence domains for tree-wise iterated OPEs in the chiral case.

%VOAは自分自身を加群とみて$C_1$-cofiniteなため、上記の結果はVOAの tree wise な iterated OPE の収束域も与えることに注意する。

%
% Then for every binary tree $A\in\Tr_n$ the
%corresponding iterated full OPE converges absolutely and locally uniformly on
%the domain $U_A\subset X_n(\C)$, and these local expansions are the restrictions
%of a single-valued real analytic function
%\[
%  C_n:F^\vee\otimes F^{\otimes n}\longrightarrow C^\om(X_n(\C),\C).
%\]
%Equivalently, the correlation functions are independent of the order and
%parenthesization of the OPEs.
%\begin{mainthm}\label{thm_A}
%Let $F$ be a full vertex operator algebra and let $\ker L(-1), \ker \Ld(-1) \subset F$ be the canonical holomorphic and anti-holomorphic sub-VOAs.  Assume that $F$ is locally $C_1$-cofinite as a $\ker L(-1)\otimes \ker \Ld(-1)$-module.  Set $F^\vee = \bigoplus_{h,\h \in \R}F_{h,\h}^*$, the restricted dual vector space.  Then there exists a sequence of maps
%$C_n: F^\vee \otimes F^{\otimes n} \rightarrow C^\om(X_n(\C),\C)$ such that
%\begin{align}
%C_n(u;a_1,\dots,a_n) |_{\overline{U}_A} =
%\left\langle u,\,
%\text{the iterated OPE of shape $A$ applied to } a_1,\dots,a_n
%\right\rangle
%% \langle u, \text{$A$-shape OPEs of $a_1,\dots,a_n$}\rangle
%
%\end{align}
%for any tree $A\in \Tr_n$ and $u\in F^\vee$ and $a_i \in F$, where the right-hand side of \eqref{eq_intro_bulk_ope1} is a formal power series absolutely convergent on $\overline{U}_A$ and converges locally uniformly to $C_n|_{\overline{U}_A}$.
%\end{mainthm}

The proof uses the action of the \emph{parenthesized braid
operad} $\CPaB$ on $\Vmodf$, the category of $C_1$-cofinite modules of a vertex
operator algebra $V$ \cite{M6}.
The operad \(\CPaB\) is an operad object in
the category of categories and is a combinatorial model for the fundamental
groupoid of the little 2-disks operad \cite{Bar,Ta}.
%$\CPaB$は圏の圏における operad 対象であり、? によって導入された little 2-disk operad の 基本亜群の組み合わせ的なモデルである\cite{}。
For each $r$, the objects of $\CPaB(r)$ are binary trees
with $r$ labeled leaves.  If $A,B\in \Tr_r$, then the morphisms from $A$ to $B$
are given by braids whose underlying permutation is compatible with the
permutations determined by $A$ and $B$.
By \cite{M6}, compositions of intertwining operators 
among $V$-modules of shape $A\in \Tr_r$ 
converge on \(U_A^c \subset X_r(\C)\) 
 and, after choosing branches over this domain,
define sections of the conformal block. 
Furthermore, analytic continuations of these conformal blocks
along paths in $X_r(\C)$ give rise to the $\CPaB$-action on $\Vmodf$.
Then, the associativity and commutativity axioms of
a full VOA say that the corresponding monodromies in $X_3(\C)$ and $X_2(\C)$ are trivial. 
Since the
associator and the braiding generate $\CPaB$ as an operad, all tree-wise OPE
expansions have the same analytic continuation,
which gives the real analytic functions in \eqref{eq_intro_bulk_ope0}; in physics, they are called the \emph{
\(r\)-point correlation functions}.
This result may therefore be regarded as the conformal-field-theoretic
analogue of the elementary fact that an iterated product in a commutative
associative algebra is independent of the order and parentheses.

In this paper we extend the above result to two-dimensional conformal field
theory with boundary.  A boundary CFT has two kinds of states, bulk states and
boundary states.  Accordingly, its OPE algebra has three basic operations:

\[
  \begin{array}{rcll}
  \cdot_z^\bulk &:& F_\bulk \otimes F_\bulk \longrightarrow F_\bulk((z,\bar z,|z|^{\mathbb R})),
      &\text{bulk OPE},\\[2mm]
 \cdot_x^\bdy &:& F_\bdy \otimes F_\bdy \longrightarrow F_\bdy((x^{\mathbb R})),
      &\text{boundary OPE},\\[2mm]
 \tau_y &:& F_\bulk \longrightarrow F_\bdy((y^{\mathbb R})),
      &\text{bulk-boundary OPE}.
  \end{array}
\]

This structure is the conformal-field-theoretic analogue of an algebra over the
homology of the \emph{Swiss-cheese operad}. By Voronov's description, such an algebra
is a triple $  (A_{\bulk},A_{\bdy},\iota)$,
where $A_{\bulk}$ is a Gerstenhaber algebra, $A_{\bdy}$ is an associative algebra, and
$\iota:A_{\bulk}\longrightarrow A_{\bdy}$ is an algebra homomorphism whose image is contained in the center of
$A_{\bdy}$ \cite{Voronov} (see also \cite{KS1,KS2}).
The boundary CFT data $(F_{\bulk},F_{\bdy},\tau_y)$ should be viewed as its
OPE-theoretic counterpart.  The ordinary products are replaced by the bulk OPE,
the boundary OPE, and the bulk-boundary OPE.  Their parameters reflect the
geometry of the upper half-plane: bulk insertions are placed at points
$z\in\bH$, boundary insertions at points $x\in\R=\partial\bH$, and the
bulk-boundary OPE depends on $y=\mathrm{Im}\,z$.

Although $x$ is real for an actual boundary-boundary OPE, we regard
$\cdot_x^{\bdy}$ as a formal operation in $x$; in iterated OPEs one may
substitute complex differences such as $z_i-x_j$.  With this convention,
iterated OPEs in boundary CFT are naturally indexed by two-colored binary trees;
in the figures below, boldface labels denote bulk leaves and ordinary labels
denote boundary leaves.
%The boundary CFT data 
%$(F_{\bulk},F_{\bdy},\tau_y)$
%should be viewed as its OPE-theoretic counterpart.  The ordinary products are
%replaced by the bulk OPE, the boundary OPE, and the bulk-boundary OPE, whose
%parameters have the geometry of the upper half-plane: bulk insertions are placed
%at points $z\in\bH$, boundary insertions at points $x\in\R=\partial\bH$, and the
%bulk-boundary OPE depends on the distance $y=\mathrm{Im}\,z$ from the boundary.
%Consequently, iterated OPEs in boundary CFT are not described by ordinary binary
%trees, but by two-colored binary trees; in the figures below, boldface labels
%denote bulk leaves and ordinary labels denote boundary leaves.
For example, an expression such as
\[
  \bigl( \tau_{y_2}(\mathbf{a}_2) \cdot^{\bdy}_{x_{24}} b_4 \bigr)
  \cdot^{\bdy}_{x_{45}}
  \bigl( \tau_{y_1}(\mathbf{a}_3 \cdot^{\bulk}_{z_{31}} \mathbf{a}_1)
  \cdot^{\bdy}_{x_{15}} b_5 \bigr).
\]
is represented by a two-colored tree of the form
\[
  \begin{tikzpicture}[baseline=-.5ex,scale=.65]
    \coordinate (r)   at (0,3);
    \coordinate (vL)  at (-1.5,2);
    \coordinate (vR)  at (1.5,2);
    \coordinate (b4)  at (-.8,1);
    \coordinate (b5)  at (2.2,1);
    \coordinate (c31) at (.8,0);
    \coordinate (a2)  at (-2.2,0);
    \coordinate (a3)  at (.2,-1);
    \coordinate (a1)  at (1.4,-1);

    \draw (r)--(vL);
    \draw (r)--(vR);

    \draw (vL)-- node[pos=.55,left=1pt,fill=white,inner sep=.5pt] {$\tau$} (a2);
    \draw (vL)--(b4);

    \draw (vR)-- node[pos=.55,left=1pt,fill=white,inner sep=.5pt] {$\tau$} (c31);
    \draw (vR)--(b5);

    \draw (c31)--(a3);
    \draw (c31)--(a1);

    % boundary leaves: thin
    \node[below=1pt] at (b4) {$4$};
    \node[below=1pt] at (b5) {$5$};

    % bulk leaves: bold
    \node[below=1pt] at (a2) {$\mathbf{2}$};
    \node[below=1pt] at (a3) {$\mathbf{3}$};
    \node[below=1pt] at (a1) {$\mathbf{1}$};
  \end{tikzpicture}.
\]

For each two-colored tree $E\in \Tr_{r,s}$, we define a domain 
$U_E^o\subset X_{r,s}(\HH)$, where
\[
  X_{r,s}(\HH)
  = \{(z_1,\ldots,z_r;x_1,\ldots,x_s)
       \in \mathbb H^r\times \mathbb R^s
       \mid \text{all insertions are distinct}\}.
\]
Under the locally $C_1$-cofinite assumptions, we prove convergence
and consistency of boundary OPEs on these domains; see Theorem~\ref{thm_bulk_boundary}.

For the bulk OPE algebra, namely a full VOA, the basic identities ensuring the
compatibility of different tree-wise OPE expansions are associativity and commutativity.  As explained above, these identities correspond to the operadic generators of the parenthesized braid operad $\CPaB$.
In the
boundary case the corresponding role is played by the \emph{parenthesized
permutation-braid operad} $\PaPB$, a two-colored operad introduced by Idrissi \cite{Id}.
Its objects in arity $(r,s)$ are the two-colored trees in $\Tr_{r,s}$,
and its morphisms are the corresponding permutation-braids, 
equivalently morphisms in the fundamental groupoid of $X_{r,s}(\HH)$ between the corresponding base configurations.
Using the embedding
\[
X_{r,s}(\HH)\hookrightarrow X_{2r+s}(\C),\quad
  (z_1,\ldots,z_r;x_1,\ldots,x_s)
  \longmapsto
  (z_1,\bar z_1,\ldots,z_r,\bar z_r,x_1,\ldots,x_s),
\]
together with the compatible doubling map on trees
$\Tr_{r,s}\rightarrow \Tr_{2r+s}$, we prove that $\PaPB$ acts on the pair of the categories
$(\Vmodf,\Vmodf\times\Vmodf)$; see Theorem~\ref{thm_SC_action}.

The \(\PaPB\)-action reduces the compatibility of all boundary OPE expansions to
the identities associated with operadic generators.  In this paper we
use generators corresponding to the following five elementary identities:
associativity and commutativity of the bulk OPE, associativity of the boundary
OPE, and the commutativity and compatibility between the bulk and
bulk-boundary OPEs.  These identities are precisely the genus-zero
\emph{bootstrap equations} of boundary CFT in the physics literature \cite{Le}.
Thus, in the present formulation, the bootstrap equations are the
 identities associated with the homotopy-theoretic generators of the
operadic structure of configuration spaces.
It follows that the tree-wise local expansions glue to single-valued real analytic 
functions
\[
C_{r,s}:F_\bdy^\vee \otimes F_{\bulk}^{\otimes r}
\otimes F_{\bdy}^{\otimes s}
\longrightarrow C^\om(X_{r,s}(\HH)),
\]
such that, for all $E \in \Tr_{r,s}$, $u\in F_\bdy^\vee$, $a_1,\ldots,a_r\in F_\bulk$ and $b_1,\dots,b_s\in F_\bdy$,
\begin{align}
C_{r,s}(u;a_1,\dots,a_r,b_1,\dots,b_s)\big|_{U_E^o}
 &=
 \langle u,\,
 \text{\(E\)-shaped OPEs of }
 a_1,\dots,a_r,b_1,\dots,b_s\rangle.
 \label{eq_intro_bulk_ope2}
\end{align}
%for every \(E\in\Tr_{r,s}\), \(u\in F_\bdy^\vee\),
%\(a_i\in F_\bulk\), and \(b_j\in F_\bdy\).

%
%Once this $\PaPB$-action on full conformal blocks is constructed, the compatibility of
%all boundary OPE expansions is reduced to the identities corresponding to the
%generators of $\PaPB$.  More precisely, we obtain single-valued correlation
%functions
%\begin{align*}
%C_{r,s}:F_\bdy^\vee \otimes F_{\mathrm{bulk}}^{\otimes r} \otimes F_{\bdy}^{\otimes s} &\rightarrow C^\om(X_{r,s}(\HH)),
%\end{align*}
%such that, for all $E \in \Tr_{r,s}$, $u\in F_\bdy^\vee$, $a_1,\ldots,a_r\in F_\bulk$ and $b_1,\dots,b_s\in F_\bdy$,
%\begin{align}
%C_{r,s}(u;a_1,\dots,a_r,b_1,\dots,b_s) |_{U_E^o}
% &= \langle u, \text{E-shaped OPEs of }
%      a_1,\dots,a_r,b_1,\dots,b_s\rangle.
% \label{eq_intro_bulk_ope2}
%\end{align}
%Their existence follows from the five
%elementary identities generating $\PaPB$: associativity and commutativity of the
%bulk OPE, associativity of the boundary OPE, and the commutativity and
%compatibility between the bulk and bulk-boundary OPEs.  These
%are precisely the genus-zero \emph{bootstrap equations} of boundary CFT 
%in the physics literature \cite{Le}.

%One of the main points of this paper is to identify the bootstrap equations with
%the identities associated with the homotopy-theoretic generators of the operadic
%structure of configuration spaces.

Let us also indicate how our results are related to existing constructions of
full and boundary CFT.  For rational \(C_2\)-cofinite vertex operator algebras,
full and boundary CFTs have been constructed and studied by means of modular
tensor categories, in particular in the works of Fuchs, Runkel, Schweigert and
their collaborators \cite{FRS,FRS2,FFFS}. 
Huang and Kong formulated full field algebras and constructed genus-zero full
CFTs from modules and intertwining operators of vertex operator algebras
\cite{HK1}. Kong introduced open-closed field algebras and developed an algebraic formulation of boundary CFT \cite{Ko1,Ko3}. 
These works are based on
the vertex tensor category theory of Huang and Lepowsky \cite{HL}; see also \cite{HLZ6,HLZ7}.

%The present paper is based instead on conformal blocks, using the action of the
%parenthesized braid operad on the conformal blocks for \(C_1\)-cofinite modules \cite{M6}.  
In the present paper, the analytic input is the action of the parenthesized braid
operad on conformal blocks for \(C_1\)-cofinite modules \cite{M6}.
We use this action to describe explicit convergence domains for
arbitrary tree-wise iterated OPEs and to prove the compatibility of the
resulting local analytic functions through the operadic structure of
configuration spaces.
These domains are also relevant in
comparisons with other formulations of quantum field theory.  In the bulk case,
they are used in the verification of the Osterwalder--Schrader axioms
\cite{OS1,OS2} for unitary full vertex operator algebras in joint work with
Adamo and Tanimoto \cite{AMT}.

%
%
%
%The present paper addresses a related analytic problem from a different point of
%view:
% Our emphasis is on the analytic domains of the OPE expansions themselves:
%for arbitrary tree-wise iterated OPEs, we give explicit convergence domains and
%prove the compatibility of the resulting local analytic functions through the
%operadic structure of configuration spaces.  

The use of local \(C_1\)-cofiniteness is motivated by examples beyond
the rational setting.  In sigma models associated with Calabi--Yau manifolds,
CFTs arise in families reflecting deformations of the underlying geometry, and
such families generally fail to remain rational \cite{AGM,Hori}. Local
\(C_1\)-cofiniteness is weaker than rationality. 
In Appendix~A, we verify it for the
current-current deformations of regular full VOAs constructed in \cite{M1}.
This class includes the deformation families arising from sigma models
associated with abelian varieties \cite{Mmirror}.

The paper is organized as follows.  In Section~1, we recall binary trees, configuration
spaces, the parenthesized braid operad, and the tree-wise convergence regions for chiral
conformal blocks.  We also review the action of $\mathrm{PaB}$ on conformal blocks associated
with locally $C_1$-cofinite $V$-modules.  In Section~2, we pass to the two-colored setting:
we recall the parenthesized permutation-braid operad $\mathrm{PaPB}$, identify it with the
combinatorics of the Swiss-cheese operad, and construct its action on boundary conformal
blocks by doubling configurations in the upper half-plane.  In Section~3, we formulate
bulk and boundary OPE algebras and prove the main consistency theorem: every
bulk or boundary tree-wise OPE expansion converges on its explicit domain and all such
expansions glue to single-valued real analytic correlation functions.

\vspace{3mm}
\begin{center}
\textbf{\large Notations}
\end{center}

\vspace{3mm}
We will use the following notations:
\begin{itemize}
\item[$X_r(\C)$:] $=\{(z_1,\dots,z_r)\in \C^r\mid z_i\neq z_j \text{ for any }i\neq j\}$
\item[$X_{r,s}(\HH)$:] $= \{(z_1,\dots,z_{r+s}) \in \C^{r+s} \mid \mathrm{Im}\, z_i>0, \mathrm{Im}\,z_j=0\text{ for }i\leq r, j>r, \text{all coordinates are distinct}\}.$
\item[$\Phi$:] $X_{r,s}(\HH) \hookrightarrow X_{2r+s}(\C),\quad (z_1,\dots,z_r,z_{r+1},\dots,z_{r+s}) \mapsto (z_1,\z_1,\dots,z_r,\z_r,z_{r+1},\dots,z_{r+s})$
\item[$\lbrack r\rbrack$:] $=\{1,2,\dots,r\}$
\item[$\Tr_r$:] the set of all binary trees with $r$ leaves labeled by $[r]$, \S \ref{sec_magma}
\item[$A$:] a binary tree in $\Tr_r$
\item[$E(A)$:] a set of all edges of $A$, \S \ref{sec_model_C}
\item[$\{z_A, x_A, \zeta_e\}_{e\in E(A)}$:] a local coordinate of $X_r(\C)$, \S \ref{sec_model_C}
\item[$T_A$:] a space of formal power series in $z_A, x_A, \zeta_e$, \S \ref{sec_model_C}
\item[$U_A$:] a simply-connected open subset of $\Xr$ associated with $A$, \S \ref{sec_model_C}
\item[$\overline{U}_A$:] an open subset of $\Xr$ without branch cut, \S \ref{sec_model_C}
\item[$\CPaB$:] the parenthesized braid operad, \S \ref{sec_magma}
\item[$\Om\Om$:] the free 2-colored operad generated by the three elements, \S \ref{sec_2magma}
\item[$\Tr^{c}(r),\Tr^o(r,s)$:] 2-colored operad of 2-colored trees, \S \ref{sec_2magma}
\item[$\tilde{\bullet}$:] $\Tr^o(r,s) \rightarrow \Tr_{2r+s},\quad E\mapsto \tilde{E}$, embedding of trees \S \ref{sec_2magma}
\item[$\PaPB^\bullet$:] the parenthesized permutation and braid operad, \S \ref{sec_2magma}
\item[$U_A^c$:] an open subset of $X_r(\C)$ associated with $A \in \Tr^c(r)$ \S \ref{sec_SC}
\item[$U_E^o$:] an open subset of $X_{r,s}(\HH)$ associated with $E \in \Tr^o(r,s)$, \S \ref{sec_SC}
\item[$\Endp_{C,D}^\bullet$:] a $2$-operad defined by coends associated with categories, \S \ref{sec_SC}
\item[$V$:] a (positive graded) vertex operator algebra, \S \ref{sec_VOA}
\item[$I_{\log}\binom{M_0}{M_1M_2}$:] a vector space of logarithmic intertwining operators, \S \ref{sec_VOA}
\item[$I\binom{M_0}{M_1M_2}$:] a vector space of intertwining operators, \S \ref{sec_VOA}
\item[$\Vmodf$:] a category of all $C_1$-cofinite $V$-modules, \S \ref{sec_VOA}
\item[$\CB$:] a chiral conformal block on $X_r(\C)$ \S \ref{sec_chiral_CB}
\item[$\CB^c$:] a full conformal block on $X_r(\C) \times X_r(\C)$ \S \ref{sec_SC}
\item[$\CB^o$:] a boundary conformal block on $X_{r,s}(\HH)$, \S \ref{sec_SC}
\end{itemize}

\vspace{5mm}

\section{Vertex operator algebra and homotopy little 2-disk operad}
\label{sec_pre}
%The bulk OPE algebra of 2d conformal field theory always has canonical subalgebras consisting of holomorphic and anti-holomorphic fields (see Section \ref{sec_full_VOA}). 
%This subalgebra has a much better mathematical structure called a vertex operator algebra.
% 
%This subalgebra is called the vertex operator algebra and has much better mathematical properties.
In this section, we review the result of \cite{M6} that the fundamental groupoid of the little 2-disk operad (parenthesized braid operad) acts on  conformal blocks of a vertex operator algebra. This result will be extended to conformal blocks defined on the upper-half plane in Section \ref{sec_2magma}.
In Section \ref{sec_magma}, we recall the definition of the parenthesized braid operad $\CPaB$.
In Section \ref{sec_model_C}, we recall the open regions $U_A \subset X_r(\C)$ associated with trees $A\in \Tr_r$.
In Section \ref{sec_VOA}, the consistency of operator product expansions of 
a vertex operator algebra is mathematically formulated.
Section \ref{sec_chiral_CB} and \ref{sec_CPaB} recall the definition of conformal blocks, their glueings and the action of $\CPaB$ on them.

% in the case of chiral conformal field theory (). The consistency in the non-chiral case and on the upper-half plane is formulated in Section \ref{sec_consistency} and proved under the assumption that Theorem \ref{thm_SC_action} in Section \ref{sec_2magma} can be used.
%共形場理論のOPEのなす代数はいつでも 正則-反正則なcanonical 部分代数を持つ。この部分代数は頂点作用素代数と呼ばれ数学的に遥かに良い性質をもつ。
%この章では頂点作用素代数の共形ブロックに fundamental groupoid of little 2-disk operad が作用するという\cite{}の結果を振り返る。この結果は2章で non-chiral な共形場理論や上反平面上で定義された共形ブロックへと拡張される。
%またSection \ref{} では、共形場理論の operator product expansion の consistency をchiral 共形場理論(頂点作用素代数)の場合に数学的に定式化し述べる。non-chiral な場合や上反平面上の consistency は3章で定式化され、2章の結果を用いることができる仮定の下で証明される。

%Let $r\in \Z_{>0}$ and set 
%\begin{align*}
%\Xr&=\{(z_1,\dots,z_r) \in \C^r \mid z_i\neq z_j \text{ for any }i\neq j \},
%\end{align*}
%which is called {\it an $r$-point configuration space}. Correlation functions of two-dimensional compact conformal field theory are real analytic functions on $\Xr$,
%which have nice expansions along singularities $\{z_i=z_j\}$.
%%\begin{align}
%%f(z_1,\dots,z_r) \sim (z_i-z_j)^r (\z_i-\z_j)^s \quad (r,s\in \C\text{ with }r-s\in \Z).
%%\label{eq_singularity}
%%\end{align}
%
%$r$点相関関数は多変数関数なので特異点周りで展開をするときは、どの変数をどういう順番で展開するかを指定する必要がある。
%そのため展開はr個の葉を持つ樹によって指定される。各樹は、配位空間の開領域とその領域の上の自然な変数を定める。
%共形場理論の側では、樹は粒子の近づけ方に対応しており、頂点作用素の括弧付きの積で表現される。
%こうした配位空間の組み合わせ的構造、樹と開領域、そして配位空間の局所座標は\cite{}で導入されて、頂点作用素代数の表現圏の組紐テンソル圏の構造を証明するために用いられた。この章ではこれらの記号を簡潔に振り返る.

\subsection{Magma, trees and braids}
\label{sec_magma}
%積(二項演算)があるとそれらを繰り返すことで、n項演算が作られる。一般にn項演算は2項演算を取る順序と括弧付けによっている。
%可能な積の順序と括弧付けはbinary trees と一対一対応がある。この章では樹と積の関係、magma operad の定義を振り返る.
%積には可換性を課すことができる.
%ここでは 2次元的/ 2-categorical な可換化である parenthesized braid operad の定義を振り返る。

We first recall the operadic description of iterated products.  The different
ways of forming an \(n\)-fold product from a binary operation are indexed by
binary trees.  These trees form the magma operad, the free operad generated by
one binary operation.  The parenthesized braid operad is obtained by enriching
this picture with braids.

%Iterating a binary operation gives \(n\)-ary operations.  In general, these
%operations depend on the order and parenthesization of the binary operations,
%and all possible iterated operations are encoded by binary trees, or equivalently
%by the magma operad.
%In general, $n$-ary operations depend on orders and parentheses of binary operations,
%and all possible $n$-ary operations have a one-to-one correspondence with binary trees (also known as magma operad).
%In this section, we review the definition of magma operad.
%In this section, we first recall the magma operad, then impose a two-dimensional, or braid, form of commutativity.

Let $\Tr_{r}$ be the set of all binary trees whose leaves are labeled by $[r]=\{1,2,\dots,r\}$.
%and $\PW_r$ denote the subset of $\Tr_r$ consisting of all binary trees whose leaves are ordered from $1$ to $r$.
Each element in $\Tr_r$ can be regarded as a parenthesized word of $\{1,2,\dots,r\}$,
that is, non-associative, non-commutative monomials on this set in which every letter appears exactly once.
For example, $(5(23))((17)(64))$ corresponds to the tree in Fig. \ref{fig_tree_example0}.
Note that $\Tr_0$ consists of the empty word
and $\Tr_3$, for example, is a set of 12 elements
\begin{align*}
\Tr_0&=\{\emptyset\},\\
\Tr_3&=\{1(23), (12)3, 1(32), (13)2, 2(13),(21)3,2(31),(23)1,3(12),(31)2,3(21),(32)1\}
\end{align*}
and
$\Tr_4$ consists of all permutations of $5$ elements
\begin{align*}
\{(12)(34),1(2(34)),1((23)4),(1(23))4,((12)3)4\}.
\end{align*}

\begin{minipage}[c]{.35\textwidth}
\centering
\begin{forest}
for tree={
  l sep=20pt,
  parent anchor=south,
  align=center
}
[
[[5][[2][3]]]
[[[1][7]]
[[6][4]]]
]
\end{forest}
\captionof{figure}{}
\label{fig_tree_example0}
\end{minipage}
\begin{minipage}[l]{.7\textwidth}
\centering
For $A\in \Tr_r$, we will use the following notations:
\begin{align*}
\Leaf(A)&=\{\text{the set of all leaves of }A\}\\
V(A)&=\{\text{the set of all vertexes of }A  \text{ which are}\\
&\quad\quad\text{ not leaves}\}\\
E(A)&=\{\text{the set of all edges of }A \text{ which are}\\
&\quad\quad\text{ not connected to leaves}\}.
\end{align*}
\end{minipage}

Since binary trees describe the freest possible \(n\)-ary operations, the
collection \(\{\Tr_r\}_{r\ge0}\) forms the free operad generated by a single
binary operation, namely the magma operad.
%Since binary trees describe the freest possible $n$-ary operations, it becomes a free operad generated from 
%a single binary operation (the magma operad).
More explicitly, its operad structure is described as follows:
%樹は可能な二項演算を最も自由に記述しているので, 最も自由な二項演算のoperad (magma operad) になる.
%持つベクトル空間$\cdot:A\otimes A\rightarrow A$である。
%代数の積は
%$(a_1\cdot a_2)\cdot (a_3 \cdot a_4)$ or $((a_2 \cdot a_4)\cdot a_3) \cdot a_1$ ($a_1,a_2,a_3,a_4\in A$) など様々な順序と括弧のつけ方が考えられる。
%可能な積全体は樹を用いて記述できる.

%We will see that an operad structure can be introduced on $\{ \Tr_r\}_{r\geq 0}.$
Let $A \in \Tr_n$ and $B\in \Tr_m$ with $n> 0$ and $p \in [n]$.
The partial composition of the operad is then defined as shown in Fig. \ref{fig_partial_comp}.
The figure shows the composition of $3((12)4)\circ_2 2(13)$.
\begin{figure}[b]
    \centering
    \includegraphics[scale=1]{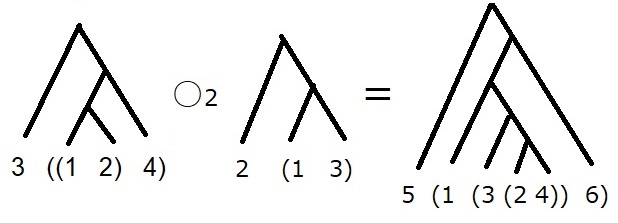}
    \caption{$3((12)4)\circ_2 2(13)$}\label{fig_partial_comp}
\end{figure}
In general, $A\circ_p B$ is defined by inserting the tree $B$ into the leaf labeled with $p$ in $A$,
adding $p-1$ to labels of leaves in $B$, and adding $m-1$ to the labels of leaves after $p+1$ in $A$.
If $B=\emptyset$, the $p$-th leaf in $A$ is simply erased, and the numbers are shifted forward.
For example,
\begin{align*}
3((12)4) \circ_2 \emptyset = 2(13).
\end{align*}
The symmetric group $S_r$ acts on $\Tr_r$ by the permutation of labels,
which satisfies the definition of a symmetric operad.

%This operad is called {\it a magma operad} and is denoted by $\magmas$.
%Since $\{\Tr_r\}_{r\geq 1}$ is closed under the operadic compositions, this is also an operad, which is denoted by
%$\magma$ and called {\it a non-unitary magma operad}.

%\begin{rem}
%\label{rem_magma_free}
%The magma operad can be defined as the free
%symmetric operad on a single generator $\mu =(12) \in \magma(2)$ (as in [Fre16]).
%\end{rem}

%次に我々は\cite{Bar,Ta1}で導入されたparenthesized braid operad (CPaB operad) の定義を振り返る。
%$\cat$を圏の圏、すなわち対象が圏で射が関手とする。圏の直積により$\cat$は対称モノイダル圏の構造を持つ。
%operadの概念は任意の対称モノイダル圏の中で考えることができ、CPaB operad は$\cat$におけるoperad object である。

We next recall the definition of the parenthesized braid operad $\CPaB$ introduced in \cite{Bar,Ta}.
Let $\cat$ be the category of categories, i.e., objects are categories and morphisms are functors. 
By the direct product of categories, $\cat$ has a structure of a symmetric monoidal category.
The notion of operad can be considered in any symmetric monoidal category, and $\CPaB$ is an operad object in $\cat$.

%figure_on_off
\begin{figure}[t]
  \begin{minipage}[b]{0.45\linewidth}
    \centering
    \includegraphics[width=2.5cm]{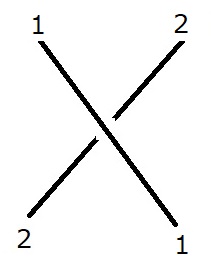}
    \caption{morphism $\sigma$}\label{fig_sigma}
  \end{minipage}
  \begin{minipage}[b]{0.45\linewidth}
    \centering
    \includegraphics[width=3cm]{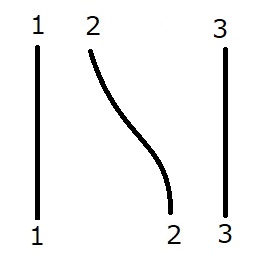}
    \caption{morphism $\alpha$}\label{fig_alpha}
  \end{minipage}
      \begin{minipage}[b]{0.45\linewidth}
    \centering
    \includegraphics[width=3.2cm]{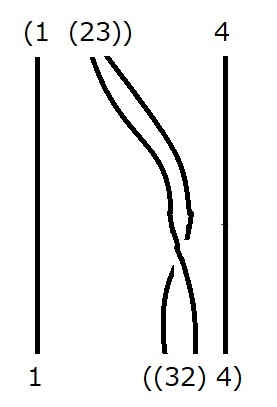}
    \caption{morphism $\alpha\circ_2\sigma$}\label{fig_alphasigma}
  \end{minipage}
    \begin{minipage}[b]{0.45\linewidth}
    \centering
    \includegraphics[width=3.2cm]{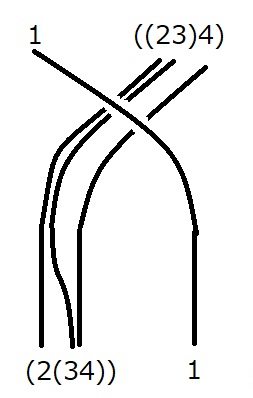}
    \caption{morphism $\si \circ_2\al$}\label{fig_sigmaalpha}
  \end{minipage}
\end{figure}

For each $r\geq 1$, $\CPaB(r)$ is the category defined as follows:
The set of all objects in $\CPaB(r)$ is the set of binary trees $\Tr_r$,
\begin{align*}
\mathrm{Ob}(\CPaB(r))=\Tr_r.
\end{align*}
%Objects are parenthesized words of $\{1,\dots,r\}$, that is non-associative, non-commutative monomials on this set in which every letter appears exactly once.
%For example, the set of all objects of $\CPaB(3)$ consists of $12$ elements,
%\begin{align*}
%\{1(23), (12)3, 1(32), (13)2, 2(13),(21)3,2(31),(23)1,3(12),(31)2,3(21),(32)1\}.
%\end{align*}
%Each object in $\CPaB(r)$ can be represented by a binary tree with $r$ leaves.

Let $p:B_r \rightarrow S_r$ be the canonical projection from the braid group to the symmetric group whose kernel is the pure braid group $PB_r$.
Let denote by $g:\Tr_r\rightarrow S_r$ the map given by forgetting the parenthesization and viewing trees as permutations.
Then, for  $A,B \in \Tr_{r}$, the space of homomorphisms is defined by
\begin{align}
\Hom_{\CPaB(r)}(A,B)=\C p^{-1}({g_A^{-1}g_B}),
\label{eq_perm_braid}
\end{align}
%where  is the map given in \eqref{eq_tree_word}
where $\C p^{-1}({g_A^{-1}g_B})$ is a $\C$-linear space with a basis $p^{-1}({g_A^{-1}g_B})$.
The composition law is induced from the one on $B_r$.
The symmetric group $S_r$ acts on $\CPaB(r)$ via renumbering the objects $\Tr_r$ and acts identically on morphisms.
The composition
\begin{align*}
\circ_p: \CPaB(n)\times \CPaB(m)\rightarrow \CPaB(n+m-1)
\end{align*}
is given by replacing the $p$-th strand of the first braid, by the second braid made very thin (see Fig. \ref{fig_sigma}, \ref{fig_alpha}, \ref{fig_alphasigma}).
This composition is consistent with the magma operad when restricted to objects.

For $r=0$, $\CPaB(0)$ is a category whose object is the only empty parenthesized word $\emptyset$ and whose 
morphism consists only of the identity map $\Hom(\emptyset,\emptyset)=\C\{\id\}$.
The composition
\begin{align*}
\circ_p: \CPaB(n)\times \CPaB(0)\rightarrow \CPaB(n-1)
\end{align*}
is given by just erasing the $p$-th strand.

\subsection{Configuration space of $\C$ and trees}
\label{sec_model_C}
%この章では各 tree $A$に対応して、$\C$上の単連結開領域$U_A$ を定義する。二次元共形場理論の共形ブロックは配位空間上のlocally constant sheaf である。
%共形ブロックの$U_A$上のsection を考えることで、$U_A$は parenthesized braid operad と 共形場理論を直接つなげる役割を果たす。
For $r \geq 1$, set 
\begin{align*}
\Xr&=\{(z_1,\dots,z_r) \in \C^r \mid z_i\neq z_j \text{ for any }i\neq j \},
\end{align*}
which is called {\it an $r$-point configuration space}. 
A chiral conformal block is a multi-valued holomorphic function on $\Xr$,
which may have branch singularities along $\{z_i=z_j\}$.
In this section, we recall the definition of simply-connected open domains $U_A \subset X_r(\C)$ for each tree $A\in \Tr_r$.
By considering conformal blocks on $U_A$, these open domains serve as a link between the parenthesized braid operad and operator product expansions of 2d conformal field theory (see \cite[Section 3]{M6} for more detail).

Let $(z_1,\dots,z_r)$ be the standard coordinate of $\C^r$.
%When examining global properties, it is convenient to use $z$-coordinate.
In this section, we will define local coordinates associated with trees $\Tr_r$.
Let $A\in \Tr_r$.
For each edge $e\in E(A)$, let $u(e)$ denote the upper vertex and $d(e)$ denote the lower vertex.
Define maps $L,R: V(A)\rightarrow \Leaf(A)$ as follows:
For each vertex $v\in V(A)$, $R(v)$ is defined by the rightmost leaf that is the descendant of $v$
and $L(v)$ by the rightmost leaf among the leaves that are descendants of the child to the left of $v$. 
Let $t_A$ be the uppermost vertex and $r_A$ be the rightmost leaf among all leaves.
Then, $r_A=R(t_A)$.

\begin{minipage}[c]{.5\textwidth}
\centering
\begin{forest}
for tree={
  l sep=20pt,
  parent anchor=south,
  align=center
}
[$t_{A}$
[$v_1$[5][[2][3]]]
[$v_2$,edge label={node[midway,right]{$e_0$}}[[1][7]]
[[6][4]]]
]
\end{forest}
%\captionof{figure}{}
%\label{fig_tree_example2}
\end{minipage}
\begin{minipage}[c]{.5\textwidth}
\centering
In the case of the left figure,
\begin{align*}
A&=(5(23))((17)(64))\\
d(e_0)&=v_2\quad u(e_0)=t_{A}\\
L(v_1)&=5 \quad R(v_1)=3\\
L(v_2)&=7 \quad R(v_2)=4\\
r_A&=4.
\end{align*}
\end{minipage}

%各edge $e\in E(A)$に対して、上側の頂点を$u(e)$、下側の頂点を$d(e)$
%とおく。$u,d$は写像$u,e:E(A)\rightarrow V(A)$を定める。
%Define maps $L,R: V(A)\rightarrow [r]$,
%i.e., maps from vertexes to leaves,
% as follows:
%各頂点$v\in V(A)$に対して、
%$R(v)$は$v$のdescendantである葉の中で最も右側にある葉によって定義する。
%また、$L(v)$は$v$の左側のchildのdescendantである葉の中で最も右側にある葉によって定義する。$R(t_A)$は全ての葉の中で最も右側にある葉であり、これを$l_A$とおく。

The functions $\{z_v:\Xr\rightarrow \C\}_{v\in V(A)}$ and $\{\zeta_e:\Xr\rightarrow \C\}_{e\in E(A)}$ are
defined by
\begin{align}
z_v &= z_{L(v)}-z_{R(v)},\\
\zeta_e &= \frac{z_{d(e)}}{z_{u(e)}}
\end{align}
This gives the family of $r-1$ functions $\{z_v:\Xr\rightarrow \C\}_{v\in V(A)}$
and the family of $r-2$ functions $\{\zeta_e:\Xr\rightarrow \C\}_{e\in E(A)}$.
Let $z_A:\Xr \rightarrow \C,\quad(z_1,\dots,z_r)\mapsto z_{r_A}$ be the projection onto the $r_A$-th component.
%これにより$r-1$個の関数の族$\{x_v:X_r(\C)\rightarrow \C\}_{v\in V(A)}$
%と$r-2$個の関数の族$\{\zeta_e:X_r(\C)\rightarrow \C\}_{v\in E(A)}$を得る。
Then,
%{\it The $x$-coordinate system} is the system of functions
\begin{align*}
(z_v)_{v\in V(A)}\times z_A: \Xr\rightarrow \C^{r-1}\times \C
\end{align*}
forms a local coordinate on $X_r(\C)$.

\begin{minipage}[c]{.5\textwidth}
\centering
To see this coordinate, it is easier to draw a tree with the function $z_v=z_{L(v)}-z_{R(v)}$ filled in at each vertex $v \in V(A)$.
%$x$-coordinateを見るには、各$v \in V(A)$に関数$x_v$を記入したtreeを書くと見やすい。
The right figure is an example for $(23)((15)4)  \in P_5$.
\end{minipage}
\begin{minipage}[c]{.5\textwidth}
\centering
\begin{forest}
for tree={
  l sep=20pt,
  parent anchor=south,
  align=center
}
[$z_3-z_4$
[$z_2-z_3$,edge label={node[midway,left]{a}}[2][3]]
[$z_5-z_4$,edge label={node[midway,right]{c}}[$z_1-z_5$,edge label={node[midway,left]{b}}[1][5]]
[4]]
]
\end{forest}
\captionof{figure}{}
\label{fig_tree_example2}
\end{minipage}

{\it The $A$-coordinate system} is the system of functions
\begin{align}
\Psi_A = 
z_A\times x_A\times (\zeta_e)_{e\in E(A)}: \Xr\rightarrow \C\times \C\times  \C^{r-2},
\label{eq_zeta_coordinate}
\end{align}
where $x_A=z_{t_A}:\Xr \rightarrow \C$.

%Define a map
%\begin{align*}
%\Psi_A:X_r^{\si A}(\R)\rightarrow \R \times \R_{>0} \times (\R_{> 0})^{E(A)}
%\end{align*}
%by
%\begin{align*}
%(z_1,\dots,z_r) \mapsto (z_{*_A},x_{t_A},(\zeta_e)_{e\in E(A)}).
%\end{align*}
%$\Psi_A$を見るにはTreeの各頂点$v\in V(A)$に数字の組$L(v)R(v)$を書くとよい。
For $A=(23)((15)4) \in \Tr_5$, $A$-coordinate is given as:
\begin{align}
\Psi_{(23)((15)4)}=
(z_4,z_3-z_4,\ze_a=\frac{z_2-z_3}{z_3-z_4},\ze_b=\frac{z_1-z_5}{z_5-z_4},\ze_c=\frac{z_5-z_4}{z_3-z_4}),
\label{eq_example_psi}
\end{align}
where the labels $\{a,b,c\}$ of edges are given as in Fig. \ref{fig_tree_example2}.

It is easy to see that the inverse function $\Psi_A^{-1}:\C\times \C^{E(A)}\rightarrow \C^r$ is a polynomial
of $\{\ze_e\}_{e\in E(A)}$ and $x_A,z_A$.
For example,
\begin{align*}
\Psi_{(23)((15)4)}^{-1}=(z_1,z_2,z_3,z_4,z_5)
=(x_A\ze_c(1+\ze_b)+z_A, (1+\ze_a)x_A+z_A,x_A+z_A,z_A,\ze_cx_A+z_A).
\end{align*}
Thus, we have:
\begin{prop}
\label{prop_psiA}
For any $A\in \Tr_r$, 
$\Psi_A$ is a bi-holomorphic function from $X_r(\C)$ onto the image in $\C^r$.
Furthermore, $\Psi_A^{-1}$ is a polynomial of $\{\ze_e\}_{e\in E(A)}$ and $x_A,z_A$,
and thus can be extended to a holomorphic function $\Psi_A^{-1}:\C^2\times \C^{E(A)}\rightarrow \C^r$.
%The map $\Psi_A:X_r^{\si A}(\R)\rightarrow \R \times \R_{>0} \times (\R_{> 0})^{E(A)}$ defines a diffeomorphism
%and is extended to a biholomorphic map from 
%$\C^r\setminus \cup_{v \in V(A)}\{z_{L(v)}=z_{R(v)}\}$ to $\C \times \C^\times \times (\C^{\times})^{E(A)}$.
\end{prop}

Set 
\begin{align*}
\Or^\alg = \C[z_1,\dots,z_r,(z_i-z_j)^\pm],
\end{align*}
a ring of regular functions on $X_r(\C)$, and
\begin{align*}
T_A &= \C[[\zeta_e\mid e \in E(A)]][z_A,\log x_A, x_A^\C,\log\zeta_e,\zeta_e^\C \mid e \in E(A)],
\end{align*}
which is a space of formal power series spanned by
the finite sum of formal power series of the form:
\begin{align*}
z_A^n x_A^r (\log x_A)^k \Pi_{e\in E(A)}\zeta_e^{r_e}(\log \zeta_e)^{k_e}F
\end{align*}
with $F\in\C[[\zeta_e\mid e \in E(A)]]$,
$n,k,k_e \in \Z_{\geq 0}$ and $r_e,r \in \C$ ($e\in E(A)$).

Any function of $\Or^\alg$ can be expanded as a formal power series in $T_A$.
For example, in the case of $A=(23)((15)4) \in \Tr_5$, we have:
%For $(z_2-z_1)^{-1} \in \mO_{X_5}^\alg$, we have:
%この章では$\zeta$座標系を用いた形式的級数やその収束半径を考察する。
%$\Or$の任意の元は、$\zeta$座標で展開することができる。まず例を見ることにして再び$A=(23)((15)4) \in \Tr_5$を考える。
%$(z_2-z_1)^{-1} \in \mO(X_3)$を考えると、
\begin{align}
\begin{split}
(z_2-z_1)^{-1} &= \left((z_2-z_3)+(z_3-z_4)-(z_5-z_4)-(z_1-z_5)
\right)^{-1}\\
&= (z_3-z_4)^{-1}\left(1 + \frac{(z_2-z_3)}{(z_3-z_4)}-\frac{(z_5-z_4)}{(z_3-z_4)}-\frac{(z_1-z_5)}{(z_3-z_4)}
\right)^{-1}\\
&=x_{(23)((15)4)}^{-1}(1+\zeta_a-\zeta_c-\zeta_b\zeta_c)^{-1}\\
&=x_{(23)((15)4)}^{-1}\sum_{l=0}^\infty (-\zeta_a+\zeta_c+\zeta_b\zeta_c)^l
\in \C[[\zeta_a,\zeta_b,\zeta_c]][x_{(23)((15)4)}^{-1}].\label{eq_example_conv}
\end{split}
\end{align}

The series in $T_A$ is called {\it a parenthesized formal power series}.
It is noteworthy that $T_A$ naturally inherits a $\Or^\alg$-algebra structure, by the $\C$-algebra homomorphism:
\begin{align}
e_A:\Or^\alg \rightarrow T_A \label{lem_module_ort}
\end{align}
In particular, $T_A$ is an $\Ort^\alg$-module.

Next, consider the radius of convergence of parenthesized formal power series.
For $p>0$, set
\begin{align*}
\D_p &= \{\ze\in \C\mid |\ze|<p\},\\
\D_p^\times &=\{\ze \in \C\mid 0<|\ze|<p\}.
\end{align*}

%For $\mathfrak{k}=(k_e)_{e\in E(A)}\in \Z_{\geq 0}^{E(A)}$,
%set $\zeta^{\mathfrak{k}}=\Pi_{e\in E(A)}\zeta_{e}^{k_e}$.
Let $\mathfrak{p}=(p_e)_{e\in E(A)}\in \R_{>0}^{E(A)}$.
%A series $F=\sum_{\mathfrak{k} \in \Z_{\geq 0}^{E(A)}}a_{\mathfrak{k}} \zeta^{\mathfrak{k}} \in \C[[\zeta_e\mid e\in E(A)]]$ is called absolutely convergent for $\mathfrak{p}$ if
%\begin{align*}
%\sum_{\mathfrak{k} \in \Z_{\geq 0}^{E(A)}} |a_{\mathfrak{k}}| \Pi p_e^{k_e} < \infty.
%\end{align*}
%For $\mathfrak{p}=(p_e)_{e\in E(A)}\in \R_{>0}^{E(A)}$,
Let $\C[[\zeta_e\mid e\in E(A)]]_{\mathfrak{p}}^\conv$
be a subspace of $\C[[\zeta_e\mid e\in E(A)]]$
consisting of formal power series which is absolutely convergent in $\Pi_{e\in E(A)}\D_{p_e}$
and set
\begin{align*}
T_A^{\mathfrak{p}}=\C[[\zeta_e\mid e\in E(A)]]_{\mathfrak{p}}^\conv[z_A,z_v^\C,\log z_v\mid v\in V(A)],
\end{align*}
a subspace of $T_A$.
% spanned by
%the finite sum of formal power series of the form:
%\begin{align*}
%\Pi_{v\in V(A)}x_v^{r_v}(\log x_v)^{k_v}F
%\end{align*}
%with $x_v\in \C$, $k_v\in\Z_{\geq 0}$
%and $F\in \C\{\zeta_e\mid e \in E(A)\}_{\mathfrak{p}}$.
It is important to note that the region of absolute convergence of 
$e_{(23)((15)4)}((z_2-z_1)^{-1})$ in the example  \eqref{eq_example_conv}
is not $|\zeta_a|<1,|\zeta_b|<1,|\zeta_c|<1$.
Since 
$e_{(23)((15)4)}((z_2-z_1)^{-1})=x_{{(23)((15)4)}}^{-1}\sum_{l=0}^\infty (-\zeta_a+\zeta_c+\zeta_b\ze_c)^l$,
if $p_a+p_bp_c+p_c <1$, then $e_{(23)((15)4)}\left((z_2-z_1)^{-1}\right) \in T_{(23)((15)4)}^{\mathfrak{p}}$.

%実際には$\zeta_a,\zeta_b,\zeta_c$についてあるpolydisk 上で収束することになる。
% by $\Psi_A$ is not $X_r(\C)$.
%しかし、$\C \times \C^\times \times (\C^{\times})^{E(A)}$において十分に小さい0の近傍を考えることで、$\Psi_A$の逆像は$X_r(\C)$に含まれる。

%For $p>0$ and $\epsilon > 0$, set
%\begin{align*}
%\D_p &= \{z\in \C\mid |z|<p\},\\
%\D_p^\times &=\{z \in \C\mid 0<|z|<p\},\\
%\D_p^\epsilon &=\{z \in \C\mid 0<|z|<p, -\epsilon<\mathrm{Arg}(z) < \epsilon\}.
%\end{align*}
\begin{dfn}
A sequence of positive real numbers $(p_e)_{e\in E(A)} \in \R_{>0}^{E(A)}$ is called {\it $A$-admissible} if $\Psi_A^{-1}(\C\times \C^\times \times \Pi_{e\in E(A)}\D_{p_e}^\times) \subset X_r(\C)$,
where $\Psi_A^{-1}$ is the polynomials in Proposition \ref{prop_psiA}.
\end{dfn}

A convergent series $f\in T_A^{\mathfrak{p}}$ is a multi-valued holomorphic function on $\Psi_A^{-1}(\C \times \C^\times \times \Pi_{e\in E(A)}\D_{p_e}^\times)$ 
because it contains $\log(z_v)$ and $z_v^r$.
Below, we will fix the branch.
%さて$f\in T_A^{\mathfrak{p}}$は、$\log(x_v)$や$x_v^r$を含むため$\Psi_A^{-1}(\C \times \C^\times \times \Pi_{e\in E(A)}\D_{p_e}^\times)$上のmulti-valued holomorphic function である。
%しばしば、branchを固定した方が便利なため、
For $A$-admissible numbers $\mathfrak{p}$,
set
\begin{align*}
U_{A}^{\mathfrak{p}}&=\Psi_A^{-1}(\C \times \C^\cut \times \Pi_{e\in E(A)}\D_{p_e}^\cut),\\
\overline{U}_{A}^{\mathfrak{p}}&=\Psi_A^{-1}(\C \times \C^\times \times \Pi_{e\in E(A)}\D_{p_e}^\times),
\end{align*}
where 
\begin{align*}
\R_-&= \{r\in \R\mid r\leq 0\},\\
\C^\cut &= \C \setminus \R_-,\\
\D_p^\cut &=\{\ze \in \C^\cut \mid |\ze|<p \}.
\end{align*}

Define the branch of $\Log:\C^\cut \rightarrow \C$ by
\begin{align}
\Log(\exp(\pi i t))= \pi i t
\label{eq_log_def}
\end{align}
for $t\in (-1,1)$. In particular, $\Arg= \mathrm{Im}\,\Log$ takes the values in $(-\pi,\pi)$.

Then, each formal power series in $T_A^{\mathfrak{p}}$
can be regarded as a single-valued holomorphic function on
$U_{A}^\mathfrak{p}$.
Set
\begin{align}
U_A  &= \cup_{\mathfrak{p}:A\text{-admissible}}U_{A}^{\mathfrak{p}} \subset \Xr,\\
\overline{U}_{A}  &= \cup_{\mathfrak{p}:A\text{-admissible}}\overline{U}_{A}^{\mathfrak{p}} \subset \Xr.\label{eq_no_cut}
\end{align}
Note that $U_A$ are connected simply-connected open subsets of $\Xr$.
Set
\begin{align*}
T_A^\conv = \cap_{\mathfrak{p}:A\text{-admissible}} T_A^\mathfrak{p} \subset T_A,
\end{align*}
which is a linear space of convergent parenthesized formal power series.
Then, by \eqref{lem_module_ort}, we have a $\C$-algebra homomorphism
\begin{align}
e_A: \Or^\alg \rightarrow T_A^\conv.
\label{eq_eA_conv}
\end{align}

Set
\begin{align*}
\pa_i=\frac{d}{dz_i},
\end{align*}
the partial differential operator  on $\Xr$ with respect to the standard coordinate $(z_1,\dots,z_r)$,
and
\begin{align*}
\Dr = \C[\pa_1,\dots,\pa_r, z_1,\dots,z_r,(z_i-z_j)^\pm\mid 1\leq i <j\leq r],
\end{align*}
a ring of differential operators on $\Xr$. Then, it is clear that $T_A^\conv$ and $\Or^\alg$ are $\Dr$-modules
and $e_A$ \eqref{eq_eA_conv} is a $\Dr$-module homomorphism.

\subsection{Vertex operator algebra and trees}
\label{sec_VOA}
A vertex operator algebra (VOA) is roughly an algebra with infinitely many products depending on the complex parameter $z\in \C^\times$:
\begin{align*}
\cdot_z: V\otimes V\rightarrow V((z)),
\end{align*}
which is called an {\it operator product expansion} in physics.
By repeating the products, we can get $n$-ary operations, which depend on $(z_1,\dots,z_n) \in X_r(\C)$.
This physically corresponds to the probability amplitude of the state in which $a_1,\dots,a_n \in V$ are inserted at $n$ points on the Riemann sphere (Fig \ref{fig_state}).
\begin{figure}[h]
  \begin{minipage}[b]{0.45\linewidth}
    \centering
    \includegraphics[width=4.5cm]{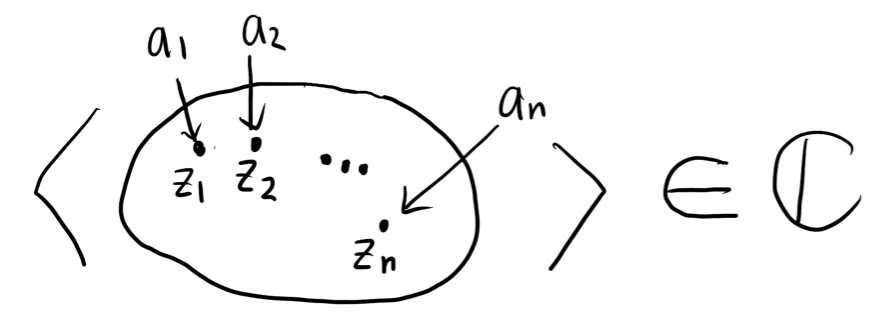}
    \caption{state}\label{fig_state}
  \end{minipage}
\end{figure}

No matter what order and parentheses we take (which correspond to binary trees in Section \ref{sec_magma}), they should coincide as analytic functions on $X_r(\C)$, since they correspond to the probability amplitudes of the same physical state in Fig \ref{fig_state}. This will be mathematically stated in Definition \ref{dfn_VOA_consistent}, which we call a consistency of operator product expansions.
In this section, we review the basic definition of a VOA and describe its consistency.
%
%
%頂点作用素代数とは大雑把にいって複素パラメータ$z\in \C^\times$に依存した無限個の積を持つ代数
%
%である. この場合もやはり積を繰り返すことでn項演算が作られる。
%
%このn項演算は
%このとき、n項演算は$(z_1,\dots,z_n)$に依存した積になり、これは物理的にはリーマン球面のn点に$V$の元(状態)を挿入した状態の確率振幅を表している (Fig \ref{fig_}).
%2項演算をどのような順序と括弧付けに取ろうが(前の章における樹の形)、同じ物理状態の確率振幅に対応しているため、これらは$X_r(\C)$上の解析関数として一致するはずである。
%これを数学的に述べたのが、operator product expansion の consistency である (Definition \ref{dfn_VOA_consistent})である。
%この章では、VOAの基本的な定義や notion を振り返り、consistency を述べる。
%

We first recall the definition of a $\Z$-graded vertex algebra based on \cite{FB,LL,Li2}:
\begin{dfn}\label{def_ZVA}
A {\it $\Z$-graded vertex algebra} is a $\Z$-graded $\C$-vector space $V=\bigoplus_{n\in \Z} V_n$ equipped with a linear map
$$Y(-,z):V \rightarrow \End (V)[[z^\pm]],\; a\mapsto Y(a,z)=\sum_{n \in \Z}a(n)z^{-n-1}$$
and an element $\va \in V_0$ satisfying the following conditions:
\begin{enumerate}
\item[V1)]
For any $a,b \in V$, $Y(a,z)b \in V((z))$;
\item[V2)]
For any $a \in V$, $Y(a,z)\va \in V[[z]]$ and $\lim_{z \to 0}Y(a,z)\va = a(-1)\va=a$;
\item[V3)]
$Y(\va,z)=\mathrm{id} \in \End V$;
\item[V4)]
%convergence
For any $a,b,c \in V$ and $u \in V^\vee$, there exists $\mu(z_1,z_2) \in \C[z_1^\pm,z_2^\pm,(z_1-z_2)^\pm]$ such that
\begin{align*}
u(Y(a,z_1)Y(b,z_2)c) &= \mu|_{|z_1|>|z_2|}, \\
u(Y(Y(a,z_0)b,z_2)c) &= \mu|_{|z_2|>|z_1-z_2|},\\
u(Y(b,z_2)Y(a,z_1)c)&=\mu|_{|z_2|>|z_1|},
\end{align*}
where $z_0=z_1-z_2$;
\item[V5)]
For any $a \in V$, $z\frac{d}{dz}Y(a,z) = [L(0),Y(a,z)]-Y(L(0)a,z)$.
%$V_n(r)V_m \subset V_{n+m-r-1}$ for any $n,m,r \in \Z$.
\end{enumerate}
\end{dfn}

\begin{rem}
\label{rem_vertex_as_product}
The vertex operator $Y(-,z):V\rightarrow \End V[[z^\pm]]$ defines ``a product'' depending on $z$ on $V$.
For psychological reasons, in this section, we write $Y(a,z)b$ as follows:
\begin{align*}
a \cdot_z b= Y(a,z)b.
\end{align*}
Then the axiom (V4) of vertex algebra is nothing but the following equality
\begin{align*}
a \cdot_{z_1} \Bigl(b \cdot_{z_2} c \Bigr) 
 &=_{a.c.} \Bigl( a \cdot_{z_1-z_2} b \Bigr)\cdot_{z_2} c\\
 &=_{a.c.}  b \cdot_{z_2} \Bigl(a \cdot_{z_1}  c\Bigr).
\end{align*}
Note that $=_{a.c.}$ means ``equal up to an analytic continuation''.
This means that the product $\cdot_z :V \otimes V \rightarrow V((z))$ is associative and commutative.
%
%
%頂点作用素$Y(-,z):V\rightarrow \End V[[z^\pm]]$は$V$上の$z$に依存した積を定めていると思える.
%心理的な理由により、$Y(a,z)b$を次のように書く:
%\begin{align*}
%a \cdot_z b= Y(a,z)b.
%\end{align*}
%すると頂点代数の公理 (V4)は次の等号が成立することに他ならない:
%
%ただし$=_{a.c.}$は「解析接続をすると等しい」という意味である.
%これは積 $\cdot_z :V \otimes V \rightarrow V((z))$ が, 可換結合的であることに他ならない.
\end{rem}

\begin{dfn}\label{def_conformal_vector}
A positive graded vertex operator algebra is a $\Z$-graded vertex algebra $V$ with a distinguished element $\om \in V_2$, called a {\it conformal vector}, such that:
\begin{enumerate}
\item
There exists a scalar $c\in \C$ such that
\begin{align*}
[L(n),L(m)]=(n-m)L(n+m) + c\delta_{n+m,0} \frac{n^3-n}{12},
\end{align*}
where $L(n)=\om(n+1)$;
\item
$\om(0)a=a(-2)\va$ and $\om(1)a=na$ for any $a\in V_n$;
\item
$V_n=0$ for $n<0$, $\dim V_n < \infty$ and $V_0=\C \va$.
\end{enumerate}
\end{dfn}

%Recall that we can think $A$ as a parenthesized product of the words $\{1,\dots,r\}$,
%%e.g., $(23)((15)4)  \in \Tr_5$, Fig. \ref{fig_tree_vertex}.
%Remark \label{rem_vertex_as_product} で述べられていたように、頂点作用素は$z$に依存した実解析的な積であるから、対応する parenthesized product を考えることができる.
%たとえば $1(23), (12)3 \in \Tr_3$ に対応する積は
%\begin{align*}
%a_1 \cdot_{z_{13}} \Bigl(a_2 \cdot_{z_{23}} a_3 \Bigr)&= Y(a_1,z_{13})Y(a_2,z_{23})a_3\\
%\Bigl( a_1 \cdot_{z_{12}} a_2 \Bigr) \cdot_{z_{23}} a_3&= Y(Y(a_1,z_{12})a_2,z_{23})a_3.
%\end{align*}
%で与えられる。ここで重要なのは頂点作用素の変数が括弧のつけ方によって変化する点である (なぜ変数変換が必要かは、introduction を参照.)
%樹$A$に対応する変数は、変数付きの積がちょうど樹の頂点に当たる場所に現れることに注意すると、Section \ref{} で与えた$x$-coordinate を用いると正しい変数を決めることができる.

Let $A \in \Tr_r$, which corresponds to the parenthesized product as in Section \ref{sec_magma}.
As was mentioned in Remark \ref{rem_vertex_as_product}, the vertex operator is a product depending on $z$, so we can consider the corresponding parenthesized product.
For example, the product corresponding to $1(23), (12)3 \in \Tr_3$ is given by
\begin{align*}
a_1 \cdot_{z_{13}} \Bigl(a_2 \cdot_{z_{23}} a_3 \Bigr)&= Y(a_1,z_{13})Y(a_2,z_{23})a_3\\
\Bigl( a_1 \cdot_{z_{12}} a_2 \Bigr) \cdot_{z_{23}} a_3&= Y(Y(a_1,z_{12})a_2,z_{23})a_3.
\end{align*}
It is important to note that the variables of the vertex operators depend on the shape of the tree.
Note that each vertex operator corresponds a vertex of the tree, $v\in V(A)$,
and the variables $\{z_v\}_{v\in V(A)}$, given in Section \ref{sec_model_C}, give the correct ones (see Fig \ref{fig_tree_vertex}).

\begin{minipage}[c]{.5\textwidth}
\centering
\begin{forest}
for tree={
  l sep=20pt,
  parent anchor=south,
  align=center
}
[$z_3-z_4$
[$z_2-z_3$,edge label={node[midway,left]{}}[2][3]]
[$z_5-z_4$,edge label={node[midway,right]{}}[$z_1-z_5$,edge label={node[midway,left]{}}[1][5]]
[4]]
]
\end{forest}
\captionof{figure}{}
\label{fig_tree_vertex}
\end{minipage}
\begin{minipage}[c]{.5\textwidth}
\centering
\begin{align}
&\Bigl( a_2 \cdot_{z_{23}} a_3 \Bigr) \cdot_{z_{34}} 
\Biggl(\Bigl( a_1 \cdot_{z_{15}} a_5 \Bigr) \cdot_{z_{54}}a_4 \Biggr)\nonumber \\
&=
Y(Y(a_2,z_{23})a_3,z_{34})Y(Y(a_1,z_{15})a_5,z_{54})a_4
\label{eq_vertex_comp}
\end{align}
\end{minipage}

In general, for trees $A \in \Tr_r$ and $a_1,\dots,a_r \in V$, denote by $Y_A(a_1,\dots,a_r,z_1,\dots,z_r)$ the composition of vertex operators defined by \eqref{eq_vertex_comp}, which we call a {\it parenthesized vertex operator}.
For $u\in V^\vee$, $\langle u, \exp(L(-1)z_A) Y_A(a_1,\dots,a_r,z_1,\dots,z_r)\rangle$ is a formal power series in $T_A$.

%, but it is not at all clear whether it converges absolutely. The condition Asymptotically $C_1$-cofinite, which will be introduced in the next chapter, is a sufficient condition for this formal series to converge.
%
%
%一般に樹$A \in \Tr_r$ and $a_1,\dots,a_r \in F$に対して, \eqref{eq_vertex_comp} で定義される頂点作用素の合成を、
%$Y_A(a_1,\dots,a_r,z_1,\dots,z_r)$ を parenthesized vertex operator と呼ぶ。
%For $u\in F^\vee$, $\langle u, Y_A(a_1,\dots,a_r,z_1,\dots,z_r)\rangle$は formal power series であるが、それが絶対収束するかは全く明らかではない。次の章で導入する Asymptotically $C_1$-cofiniteという条件は, この形式的級数が収束するための十分条件である。

\begin{dfn}\label{dfn_VOA_consistent}
A positive graded vertex operator algebra $V$ is called consistent if the following properties hold for any $u \in V^\vee$ and $a_{[r]}\in V^{\otimes r}$ for $r\geq 2$:
\begin{enumerate}
\item
For any tree $A\in \Tr_r$, the formal power series $\langle u, \exp(L(-1)z_A)Y_A(a_1,\dots,a_r,z_{[r]})\rangle$ is in $T_A^\conv$, i.e., is absolutely convergent in $\overline{U}_A$.
Denote this holomorphic function on $\overline{U}_A$ by $S_A(u,a_{[r]},z_{[r]})$.
\item
There exists a sequence of linear maps
\begin{align*}
S_r:V^\vee \otimes V^{\otimes r} \rightarrow \C[z_i,(z_i-z_j)^\pm\mid i\neq j],\quad\quad r \geq 2.
\end{align*}
such that $S_A(u,a_{[r]},z_{[r]}) = S_r |_{U_A}$ for any tree $A\in \Tr_r$ as holomorphic functions.
\end{enumerate}
\end{dfn}
This definition says that the product of a vertex operator algebra is consistent, i.e., it is uniquely determined regardless of the order and parentheses of the products.
The following proposition can be proved  by elementary and direct ways, but we omit the proof because we will give the proof in more non-trivial setting later (Theorem \ref{thm_bulk}):
\begin{prop}\label{prop_VOA_consistent}
A positive graded vertex operator algebra is consistent.
\end{prop}

The following proposition gives the actual convergence region when taking special trees as an example. The second tree in this proposition is not needed in this paper, but will be used to prove the cluster decomposition property of  correlation functions of bulk CFT in \cite{AMT}:
\begin{prop}
\label{rem_tree_region}
For any $r>0$,
\begin{align*}
\overline{U}_{1(2(3(\dots (r-1,r)\dots))} &=\{(z_1,\dots,z_r)\in \C^r \mid |z_1-z_r|>|z_2-z_r|>\cdots >|z_{r-1}-z_r|>0\}
\end{align*}
and for any $m,n>0$,
\begin{align*}
&\overline{U}_{\Bigl(1(2(3\cdots (m-1, m))\cdots)\Bigr)\Bigl(m+1(m+2(m+3\cdots (m+n-1,m+n )\cdots)\Bigr)}\\
&=\{|z_{i+1m}|<|z_{im}|, |z_{j+1,n+m}|<|z_{j,n+m}|, |z_{1m}|+|z_{m+1,m+n}|<|z_{m,n+m}|\\
&\text{ for }1 \leq i \leq m-2, m+1 \leq j\leq m+n-2\}.
\end{align*}
\end{prop}

%\begin{minipage}[l]{.5\textwidth}
%For example, $n=3, m=3$ corresponds to the tree shown on the right and the corresponding composition of vertex operator is
%\begin{align}
%e^{L(-1)z_6}Y(Y(a_1,z_{13})Y(a_2,z_{23})a_3,z_{36})Y(a_4,z_{46})Y(a_5,z_{56})a_6.
%\label{eq_comp_vertex}
%\end{align}
%\end{minipage}
%\begin{minipage}[l]{.5\textwidth}
%\centering
%\begin{forest}
%for tree={
%  l sep=20pt,
%  parent anchor=south,
%  align=center
%}
%[
%[[1][[2][3]]]
%[[4][[5][6]]]
%]
%\end{forest}
%\label{fig_intro_tree1}
%\end{minipage}
\begin{proof}
In the case of $A= 1(2(3(\dots (r-1,r)\dots))\in \Tr_r$, the $A$-coordinate is as follows:
\begin{align*}
\ze_i= \frac{z_{i+1r}}{z_{ir}}\quad\quad\text{ for }i=1,\dots,r-2.
\end{align*}
Let $1 \leq i<j \leq r$. If $j \neq r$, then
\begin{align*}
z_{ij} = z_{ir} -z_{jr} = z_{ir}(1- \ze_{i}\ze_{i+1}\dots\ze_{j-1}).
\end{align*}
Since $(1- \ze_{i}\ze_{i+1}\dots\ze_{j-1})^{-1}$ absolutely converges in $|\ze_{i}\ze_{i+1}\dots\ze_{j-1}|<1$. Hence, the assertion holds.

In the case of $A= \Bigl(1(2(3\cdots (m-1, m))\cdots)\Bigr)\Bigl(m+1(m+2(m+3\cdots (m+n-1,m+n )\cdots)\Bigr) \in \Tr_{n+m}$, the $A$-coordinate is as follows:
\begin{align*}
\ze_l = \frac{z_{1m}}{z_{m,m+1}},\quad
\ze_r = \frac{z_{m+1,m+r}}{z_{m,m+n}},\quad \ze_i= \frac{z_{i+1}}{z_{i}}\quad\quad (i=1,\dots,m-2, m+1,\dots,m+n-2).
\end{align*}
Then, for $1 \leq i < j \leq n+m$, we have:
\begin{align*}
z_{ij} =
\begin{cases}
z_{im} -z_{jm} = z_{im}(1- \ze_{i}\ze_{i+1}\dots\ze_{j-1}) & \text{ if } j \leq m \\
z_{i,m+n} -z_{j,m+n} = z_{i,m+n}(1- \ze_{i}\ze_{i+1}\dots\ze_{j-1}) & \text{ if } i \geq m+1 \\
z_{ir} -z_{jr} = -z_{m,m+n}(1- \ze_{l}\ze_{1}\dots\ze_{i-1}+\ze_r \ze_{m+1}\cdots \ze_{j-1}) & \text{ if } i \leq m \text{ and } j \geq m,
\end{cases}
\end{align*}
which converge if $|\ze_i|<1$ for $i=1,\dots,m-2, m+1,\dots,m+n-2$ and $|\ze_l+|\ze_r|<1$.
\end{proof}

Let $V$ be a positive graded vertex operator algebra. Throughout of this paper, we assume that $V$-module $M$ satisfies the following conditions:
\begin{enumerate}
\item
The action of $L(0)$ on $M$ is locally finite;\\
Denote the generalized eigenspace by $M_h$ for $h\in \C$.
\item
$\dim M_h < \infty$;
\item
There are finitely many $\Delta_i \in \C$ ($i=1,\dots,N$) such that:
\begin{align*}
M = \bigoplus_{i=1}^N \bigoplus_{n\geq 0} M_{\Delta_i+n}.
\end{align*}
\end{enumerate}

Let $M$ be a $V$-module.
For any $n \in \Z_{>0}$, set 
\begin{align*}
C_n(M)= \{a(-n)m\mid m\in M \text{ and }a \in\bigoplus_{k\geq 1}V_k \}.
\end{align*}
%and
%\begin{align}
%H(M)=\{m \in M\mid a(\Delta -1 +k)m=0 \text{ for any }\Delta \in \Z_{\geq 0}, a\in V_{\Delta} \text{ and }k\geq 1 \},
%\label{eq_lowest_weight}
%\end{align}
%the space of lowest weight vectors.

\begin{dfn}
A $V$-module $M$ is called {\it $C_n$-cofinite} if $M/C_n(M)$ is a finite-dimensional vector space.
\end{dfn}

Since $(L(-1)a)(-n)=na(-n-1)$ for any $a \in V$ and $n\in \Z_{>0}$,
$C_{n+1}(M) \subset C_{n}(M)$. Hence, if $M$ is $C_{n+1}$-cofinite,
then $M$ is $C_{n}$-cofinite.
Note that any vertex algebra is $C_1$-cofinite by (V2).
Denote by $\Vmodu$ the category of all $V$-modules
and $\Vmodf$ the full subcategory of $\Vmodu$ consisting of all $C_1$-cofinite V-modules $M$ such that the dual module $M^\vee$ is finitely generated.

We will recall the definition of a logarithmic intertwining operator of a vertex operator algebra from \cite{Mi1,Mi2}.
Let $M_0,M_1,M_2$ be $V$-modules.
\begin{dfn}\label{def_int}
%We assume that for all $i=1,2,3$,
% the action of $L(0)=\om(1)$ on $M_i$ is semisimple with real eigenvalues and the eigenspaces $(M_i)_r=\{m\in M_i\;|\; 
%L(0)m=r m \}$ satisfy $(M_i)_r=0$ for any sufficiently small 
%$r\in \R$ for $i=1,2,3$.
{\it A logarithmic intertwining operator} of type $\binom{M_0}{M_1 M_2}$
is a linear map
$$\Y_1(\bullet,z):M_1 \rightarrow \text{Hom} (M_2,M_0)[[z^\C]][\log z],
\; m \mapsto \Y_1(m,z)=\sum_{k \geq 0} \sum_{r \in \C} m(r;k) z^{-r-1}(\log z)^k$$
such that:
\begin{enumerate}
\item[I1)]
For any $m \in M_1$ and $m' \in M_2$,
$\Y_1(m,z)m' \in M_0[[z]][z^\C,\log z]$;
\item[I2)]
$[L(-1),\Y_1(m,z)]=\frac{d}{dz}\Y_1(m,z)$ for any $m \in M_1$;
\item[I3)]
For any $m \in M_1$, $a \in V$ and $n\in \Z$,
\begin{align*}
[a(n), \Y_1(m,z)]&=\sum_{k\geq 0} \binom{n}{k} \Y_1(a(k)m,z)z^{n-k}\\
\Y_1(a(n)m,z)&=\sum_{k\geq 0}\binom{n}{k}
\left( 
a(n-k)\Y_1(m,z)(-z)^k
- \Y_1(m,z)a(k)(-z)^{n-k}
\right).
\end{align*}
\end{enumerate}
\end{dfn}

The space of all logarithmic intertwining operators of type $\binom{M_0}{M_1M_2}$ forms a vector space,
which is denoted by $I_{\log} \binom{M_0}{M_1M_2}$.
If $\Y_1(\bullet,z) \in I_{\log} \binom{M_0}{M_1M_2}$ does not contain any logarithmic term,
i.e., $\Y_1(m,z) \in \mathrm{Hom}(M_2,M_0)[[z]][z^\C]$ for any $m\in M_1$,
then
$\Y_1(\bullet,z)$ is called {\it an intertwining operator} of type $\binom{M_0}{M_1M_2}$ \cite{FHL}.
Denote by $I\binom{M_0}{M_1M_2}$
the space of all intertwining operators of type $\binom{M_0}{M_1M_2}$.

%\begin{rem}
%A logarithmic intertwining operator was introduced in \cite{Mi1,Mi2}. 
%There are subtle variations in the definition of logarithmic intertwining operators, see \cite{M6} for the precise definition used in this paper.
%\end{rem}

\subsection{Conformal blocks on $\C$}
\label{sec_chiral_CB}
Let $V$ be a vertex operator algebra, $r\in \Z_{>0}$,
and $\{M_i\}_{i =0,1,\dots,r}$ $V$-modules.
Set 
\begin{align*}
\Mr=M_0^\vee \otimes M_1 \otimes \cdots \otimes M_r,
\end{align*}
where $M_0^\vee=\bigoplus_{h\in \C} (M_0)_h^*$ is the dual module of $M_0$.
A conformal block is a sheaf of holomorphic solutions of a $\Dr$-module,
defined for a sequence of $V$-modules $\{M_i\}_{i =0,1,\dots,r}$ (see \cite{FB,NT}).
In this section, we will review conformal blocks and their operadic structures  based on \cite{M6}.
%共形ブロックとは加群$\{M_i\}_{i =0,1,\dots,r}$に対して定義される$\Dr$-module の正則関数解のなす層のことである (see  \cite{FB,NT})。
%この章では\cite{M6}に基づき、共形ブロックとその貼り合わせを復習する。

For each $a \in V$, $i =1,\dots,r $ and $n\in \Z$,
define a linear map $a(n)_i:\Mr\rightarrow M_{[r]}$ by
the action of $a(n):M_i \rightarrow M_i$ on the $i$-th component.

On the $0$-th component, 
%we consider two actions.
%Define $a(n)_0:\Mr\rightarrow \Mr$ by
%\eqref{eq_dual}, the dual module structure on $M_0^\vee$,
%and
define $a(n)_0^*:\Mr\rightarrow \Mr$ by $a(n)^*:M_0^\vee \rightarrow M_0^\vee$ on the $0$-th component,
where 
\begin{align*}
(a(n)^*u)(\bullet)=u(a(n)\bullet) \text{ for } u\in M_0^\vee.
\end{align*}
%\begin{rem}
%\label{rem_dual}
%Note that for $a,b\in V$, $n,m\in \Z$, $a(n)_0^*,b(m)_0^* \in \End M_0^\vee$ are actions on the dual vector space, thus $\left(b(m)a(n)\right)_0^*=a(n)_0^*b(m)_0^*$, that is,
%\begin{align}
%(a(n)_0^*b(m)_0^*f)(\bullet)=f(b(m)_0a(n)_0\bullet)\quad\text{ for }f\in M_0^\vee.
%\label{eq_rem_dual}
%\end{align}
%\end{rem}

%We will introduce a D-module $D_{M_S,M_\outm}$ such that
%the solution sheaf $\sol(D_{M_S,M_\outm}) = \Hom_{D_{X_n}}(D_{M_S,M_\outm},\mO_{X_n}^\an)$
%is isomorphic to the conformal block $\CB_{M_S,M_\outm}$.

Let $\Or^\alg \otimes \Mr$ be an $\Or^\alg$-module,
where the $\Or^\alg$-module structure is defined by the multiplication on the left component.
%Let $C \in \CB_{M_S}^{M_\outm}(V)$.
Define a $\Dr$-module structure on $\Or^\alg \otimes \Mr$
by
\begin{align}
\pa_i \cdot (f\otimes \mr) = (\pa_i f) \otimes \mr + f\otimes L(-1)_i \mr
\label{eq_D_def}
\end{align}
for $i=1,\dots,r$, $f \in \Or^\alg$, $\mr \in \Mr$.

%We can extend $C$ linearly over $\mO_{X_n}(V) \otimes M_\outm^\vee\otimes M_S$
%as an $\mO_{X_n}(V)$-module, or as an $\mO_{X_n}$-module.

Let $N_{\Mr}$ be the $\Dr$-submodule of 
$\Or^\alg \otimes \Mr$ generated by the following elements:
\begin{align}
1 \otimes a(n)_i \mr - &\sum_{k \geq 0}\binom{n}{k} (-z_i)^k \otimes a(-k+n)_0^* \mr
+ \sum_{1\leq s\leq r, s \neq i} \sum_{k \geq 0} \binom{n}{k}(z_s-z_i)^{n-k}\otimes a(k)_s \mr,
\label{eq_ker1}
%&1\otimes a(l)^*u \otimes m_S
%- \sum_{s \in S}\sum_{k \geq 0} \binom{l}{k} z_s^{l-k}\otimes u \otimes a(k)_s\cdot m_S
%\label{eq_ker2}
\end{align}
for all $\mr \in \Mr$, $a \in V$ and $i \in \{1,\dots,r\}$ and $n \in \Z$.
%We note that the all above sums are finite
%and $z_s^l, (z_p-z_q)^{k}$ is in $\Or^\alg$ for any $l \geq 0$ and $k\in \Z$.
%Thus, the above definition is well-defined.
Set 
$$
D_{\Mr} = (\Or^\alg \otimes \Mr) / N_{\Mr},
$$
which is a $\Dr$-module.
%本来であれば、$D_{\Mr}$ではなく$D_{M_0,M_1,\dots,M_r}$と書くべきであることを注意する。
%\begin{rem}
%\label{remark_D_KZ}
%Let $\g$ be a finite-dimensional simple Lie algebra,
%and $V_{\g,k}$ be the affine vertex operator algebra at level $k \in \C \setminus \Q$,
%$M_0,M_1,\dots,M_r$ be Weyl modules which are the induced modules from finite-dimensional $\g$-modules.
%Then, $D_\Mr$ is the Knizhnik--Zamolodchikov equations.
%For another example, if $V$ is a Virasoro minimal model, then the corresponding D-module is the BPZ equation \cite{BPZ}.
%\end{rem}
The following lemma is clear:
\begin{lem}
\label{lem_D_functor}
The assignment of $\Mr$ to $D_{\Mr}$ determines the following $\C$-linear functor:
\begin{align*}
D_{\bullet}:{\Vmodu}^\op \times {\Vmodu}^r &\rightarrow \underline{\Dr \text{-mod}},\\
(M_0,M_1,\dots,M_r) \quad &\mapsto \quad D_{\Mr}.
\end{align*}
\end{lem}

Let $\Or^\an$ be the sheaf of holomorphic functions on $\Xr$.
Set
\begin{align*}
\CB_{\Mr} = \mathrm{Hom}_{\Dr}(D_{\Mr},\Or^\an),
\end{align*}
the holomorphic solution sheaf, which is called {\it a chiral conformal block}.

\begin{rem}
\label{rem_assign_hol}
Let $U\subset \Xr$ be an open subset and $C\in \CB_\Mr(U)$.
Let $\Mr\rightarrow \Or^\alg \otimes \Mr$ be the embedding defined by
$\mr\mapsto 1\otimes \mr$ for $\mr\in \Mr$.
Then, by $\Mr\rightarrow \Or^\alg \otimes \Mr \rightarrow D_{\Mr}$,
$C$ can be regarded as a linear map $\Mr \rightarrow \Or^\an(U)$,
which assigns each vector in $\Mr$ to a holomorphic function on $U$.
\end{rem}

The following result is obtained in \cite{M6} by refining the idea in \cite{H2}:
\begin{prop}
\label{prop_coherence}
Let $M_0,M_1,\dots,M_r\in \Vmodu$.
Assume that $M_1,\dots,M_r$ are $C_1$-cofinite and $M_0^\vee$ is finitely generated.
Then, $D_{\Mr}$ is a finitely generated $\Or^\alg$-module. 
In particular, $D_{\Mr}$ is holonomic on $\Xr$ and the holomorphic solution sheaf $\mathrm{Hom}(D_{\Mr},\Or^\an)$ is a locally constant sheaf of finite rank on $\Xr$.
\end{prop}

For a locally constant sheaf, we can define the monodromy representation of the fundamental groupoid $\Pi_1(\Xr)$,
which plays an essential role in our construction of an action of $\CPaB$,
which is natural with respect to $M_0,M_1,\dots,M_r \in \Vmodf$.
%Let $U,V \subset \Xr$ be connected simply-connected open subsets, $q_0 \in U$, $q_1\in V$
%and $\ga:[0,1]\rightarrow \Xr$ be a continuous map such that $\ga(0)=q_0$ and $\ga(1)=q_1$.
%For any $C\in \CB_\Mr(U)$, 
%let $A_\ga C$ be the analytic continuation of $C$ along the path $\ga$.
%Then, $A_\ga$ depends on the homotopy classes of $\ga$ and thus gives a map
%\begin{align*}
%A:\mathrm{Hom}_{\Pi_1(\Xr)}(q_0,q_1) \rightarrow \mathrm{Hom}_{\Vect}(\CB_\Mr(U),\CB_\Mr(V)),
%\quad \ga\mapsto A_\ga.
%\end{align*}
%Then, we have:
%\begin{prop}
%\label{prop_functor_monodromy}
%Let $N_0,N_1,\dots,N_r \in \Vmodf$, and $f_i: N_i \rightarrow M_i$, $g:M_0\rightarrow N_0$ be $V$-module homomorphisms.
%Then, the following diagram commutes:
%\begin{align*}
%\begin{array}{ccc}
%\CB_\Mr(U) & \overset{A_\ga}{\rightarrow}&\CB_\Mr(V)\\
%{}_{(g_*,f_1^*,\dots,f_r^*)}\downarrow &&\downarrow_{(g_*,f_1^*,\dots,f_r^*)}\\
%\CB_{N_{[0;r]}}(U) & \overset{A_\ga}{\rightarrow}&\CB_{N_{[0;r]}}(V),
%\end{array}
%\end{align*}
%where $(g_*,f_1^*,\dots,f_r^*)$ is defined by
% functoriality in Lemma \ref{lem_D_functor}.
%%That is, the linear isomorphism $A$ is a natural transformation between the functors in $\Endp_\Vmodf(r)$.
%\end{prop}

Let $A\in \Tr_r$.
In Section \ref{sec_model_C},
we introduced $\Dr$-modules 
\begin{align*}
T_A^\conv = \C[[\zeta_e\mid e\in E(A)]]^\conv[z_{A}, x_A^\C,\log x_A,\zeta_e^\C,\log\zeta_e \mid e\in E(A)].
\end{align*}
Since we impose the convergence in $U_A$ on the formal power series and $U_A$ is simply-connected,
any formal solutions
\begin{align*}
\mathrm{Hom}_{\Dr}(D_{\Mr}, T_A^\conv)
\end{align*}
define a well-defined section of $\CB_\Mr(U_A)$.
This gives a linear map
\begin{align*}
s_A:\mathrm{Hom}_{\Dr}(D_{\Mr}, T_A^\conv) \rightarrow \mathrm{Hom}_{\Dr}(D_{\Mr}, \Or^\an(U_A)).
\end{align*}

Conversely, we showed that 
any conformal block has an expansion of the form in $T_A^\conv$.
\begin{thm}\cite[Theorem 4.23]{M6}
\label{thm_expansion}
Let $M_0,M_1,\dots,M_r \in \Vmodf$.
For $A \in \Tr_r$, $s_A$ is isomorphism of vector spaces. The inverse map, which is defined by the series expansion, is denoted by
$$
e_A:\CB_{\Mr}(U_A) \rightarrow \mathrm{Hom}_{\Dr}(D_{\Mr},T_A^\conv).
$$
\end{thm}

Let us consider the simplest tree $(12) \in \Tr_2$. Note that, in this case,
\begin{align*}
T_{12} = \C[z_2, z_{12}^\C, \log(z_{12})]=T_{12}^\conv
\end{align*}
is just a polynomial and does not contain any infinite series, since the set of all edges $E(12)$ is empty.
Let $M_i \in \Vmodu$ ($i=0,1,2$).
For $I(\bullet,z) \in I_{\log}\binom{M_0}{M_1M_2}$ and $u \in M_0^\vee$, $m_i \in M_i$,
\begin{align*}
\langle u, \exp(L(-1)z_2) I(m_1,z_{12})m_2 \rangle \in \C[z_2, z_{12}^\C, \log(z_{12})] = T_{12}^\conv
\end{align*}
and it is easy to show that this is an element of $\mathrm{Hom}_{D_{X_2(\C)}}(D_{M_{[0;2]}},T_{(12)}^\conv)$.
Then, we have (see \cite[Proposition 5.7]{M6}):
\begin{prop}\label{prop_int}
The above map $I_{\log}\binom{M_0}{M_1M_2} \rightarrow \CB_{M_{[0;2]}}(U_{12})$
is isomorphism.
\end{prop}

Let $A\in \Tr_r$, $B\in \Tr_{s}$ and $p\in [r]$. Then, $A\circ_p B\in \Tr_{r+s-1}$.
Let $M_0, M_1^A,\dots,M_r^A$ and $M_1^B,\dots,M_{s}^B$ be $V$-modules in $\Vmodf$.
Set
\begin{align*}
M_{r\circ_p s}^{A,B}= M_0^\vee \otimes  M_1^A\otimes M_2^A \otimes \cdots \otimes 
M_{p-1}^A \otimes M_1^B\otimes \cdots \otimes M_s^B \otimes M_{p+1}^A \otimes \cdots \otimes M_r^A
\end{align*}
and let
\begin{align*}
C_A \in \CB_{M_0;M_1^A,\dots,M_r^A}(U_A) \text{ and } C_B \in \CB_{M_p^A;M_1^B,\dots,M_s^B}(U_B).
\end{align*}
We finally recall the operadic composition (the glueing of solutions) of $C_A$ and $C_B$,
which defines a new conformal block
\begin{align*}
C_A \circ_p C_B \in \CB_{M_{r\circ_p s}^{A,B}}(U_{A \circ_pB})
\end{align*}
(see \cite[Section 5.3]{M6}).

By Remark \ref{rem_assign_hol} and Theorem \ref{thm_expansion}, we regard the conformal blocks as
the (convergent) formal power series valued linear map on the tensor product of modules.
For $m_{r\circ_p s}^{A,B}= u \otimes  m_1^A\otimes m_2^A \otimes \cdots \otimes 
m_{p-1}^A \otimes m_1^B\otimes \cdots \otimes m_s^B \otimes m_{p+1}^A \otimes \cdots \otimes m_r^A \in M_{r\circ_p s}^{A,B}$,
define the operadic composition by
\begin{align}
&(C_A \circ_p C_B)(m_{[r \circ_p r']})\label{eq_comp}\\
&= \sum_{h \in \C} \Bigl(\sum_{i \in I_h} 
e_A(C_A)(u,m_1^A,\dots,m_{p-1}^A,e_i^h, m_{p+1}^A,\dots,m_r^A)
e_B(C_B)(\exp(-L(-1)^*z_{r_B})e_h^i,m_1^B,\dots,m_s^B)\Bigr) \nonumber\\
& \in \C[[\zeta_e\mid e\in E(A\circ_p B)]] [z_{A\circ_p B}, x_{A\circ_p B}^\C,\log x_{A \circ_p B},\zeta_e^\C,\log\zeta_e \mid e\in E(A\circ_p B)].\nonumber
\end{align}
Here, $\{e_i^h\}_{i \in I_h}$ is a basis of $(M_p^A)_h$ and $\{e_h^i\}_{i \in I_h}$ is the dual basis of $(M_p^A)_h^*$.
This infinite sum is well-defined as formal power series, i.e., each coefficient of formal variables is a finite sum.
It is non-trivial that the right-hand-side of \eqref{eq_comp} is in $T_{A\circ_p B}^\conv$,
that is, absolutely convergent in $U_{A\circ_p B}$.
This result is obtained in \cite[Corollary 5.12]{M6}.
Hence, \eqref{eq_comp} gives a section of 
$\mathrm{Hom}_{D_{X_{r+s}}}(D_{M_{r\circ_p s}^{A,B}}, T_{A\circ_pB}^\conv)$,
and thus, $\CB_{M_{r\circ_p s}^{A,B}}(U_{A\circ_p B})$.

To summarize, we have the following result:
\begin{thm}
\label{thm_glue}
The following sum is locally uniformly convergent in $U_{A\circ_p B}$
\begin{align*}
\sum_{h \in \C} \Bigl|\sum_{i \in I_h} 
e_A(C_A)(u,m_1^A,\dots,m_{p-1}^A,e_i^h, m_{p+1}^A,\dots,m_r^A)
e_B(C_B)(\exp(-L(-1)^*z_{r_B})e_h^i,m_1^B,\dots,m_s^B)\Bigr|,
\end{align*}
where the absolute values are taken for the sum over the conformal weights $h\in\C$.
In particular, \eqref{eq_comp} defines the linear map:
%\begin{align}
%\CB_{\bullet}(U_A): \Vmodf \times (\Vmodfr)^\op \rightarrow \Vect, \Mr \mapsto \CB_{\Mr}(U_A),
%\end{align}
%we have a natural transformation
\begin{align}
\comp_{p}:\CB_{M_0;M_1^A,\dots,M_r^A}(U_A)\otimes\CB_{M_p^A;M_1^B,\dots,M_s^B}(U_B)
\rightarrow \CB_{M_{r\circ_p s}^{A,B}}(U_{A\circ_p B}).\label{eq_glue}
\end{align}
\end{thm}

\subsection{Homotopy action of little 2-disk operad}
\label{sec_CPaB}

Let $r\geq 1$ and set
\begin{align*}
X_r(\R)=\{(z_1,\dots,z_r)\in \R^r\mid z_i \neq z_j \text{ for any }i\neq j\}.
\end{align*}
\begin{dfn}
\label{def_order_Q}
Let $A\in \Tr_r$ and $Q \in X_r(\R)$.
The order of $Q$ and $A$ is said to be equal if the following conditions are satisfied:
For any $i,j\in\{1,\dots,r\}$ with $i\neq j$,
$z_i <z_j$ if and only if the leaf $i$ is to the right of the leaf $j$ in $A$.
\end{dfn}

Let $Q:\Tr_r\rightarrow X_r(\R) \subset \Xr$ satisfy
the following conditions:
\begin{enumerate}
\item[Q1)]
The order of $Q(A)$ and $A$ is equal,
\item[Q2)]
$Q(A) \in U_A$ for $A\in \Tr_r$.
\end{enumerate}

%$\CPaB(r)$の射はBraid group $B_r$を用いて定義されていることを思い出そう。
%Braid group は実数直線$\R$上の異なる$n$点から異なる$n$点への複素平面における道のhomotopy classes とみなせるのであった。よって順序が正しいことから、$g \in \mathrm{Hom}_\CPaB(r)(A,A')$に対して
%path $\ga(g):[0,1]\rightarrow X_r$ がup to homotopy で定義できる。
%この path の homotopy class を$[g]_Q$と書くことにする。

Recall that the maps of $\CPaB(r)$ are defined using the pure braid group $PB_r$.
The pure braid group can be regarded as the homotopy classes of paths in the complex plane from different $n$ points to different $n$ points on the real line $\R$. Therefore, since the order is correct, for $g \in \mathrm{Hom}_\CPaB(r)(A,A')$, a corresponding path $\ga(g):[0,1]\rightarrow \Xr$ with $\ga(g)(0)=Q(A)$ and $\ga(g)(1)=Q(A')$ can be defined up to homotopy. The homotopy class of this path is written as $[g]_Q$.

\begin{rem}
For example, for $((12)3)(45)$, we can consider $Q=(z_1,z_2,z_3,z_4,z_5)\in X_5(\R)$ as in the figure below.
%zeta-coordinateの幾何学的な意味を考えると、上記の条件を満たす$Q$の存在は明らかである。
%たとえば$((12)3)(45)$に対しては図のような$Q=(z_1,z_2,z_3,z_4,z_5)\in X_5(\R)$を考えればよい。
\end{rem}

\begin{tikzpicture}[baseline=(current bounding box.center)] 
\tikzstyle point=[circle, fill=black, inner sep=0.05cm]
 \node[point, label=below:$(5$] at (0.5,0.5) {};
 \node[point, label=below:$4)$] at (1,0.5) {};
 \node[point, label=below:$(3$] at (3,0.5) {};
 \node[point, label=below:$(2$] at (4,0.5) {};
 \node[point, label=below:$1))$] at (4.5,0.5) {};
\end{tikzpicture}

Define a functor $\CB_Q:\CPaB \rightarrow \Endp_\Vmodf$ as follows:
For an object $A\in \Tr_r$ with $r\geq 1$,
\begin{align*}
\CB_Q(A):\Vmodf \times (\Vmodfo)^r \rightarrow \Vect,\quad \Mr \mapsto \CB_\Mr(U_A).
\end{align*}
For $r=0$, recall $\Tr_0$ consists of the empty word $\emptyset$.
Define $\CB(\emptyset) \in \Endp_\Vmodf(0)=\Func(\Vmodf,\Vect)$ by
\begin{align*}
\CB_Q(\emptyset)=\mathrm{Hom}(V,\bullet): \Vmodf \rightarrow \Vect,\quad
M\mapsto \mathrm{Hom}(V,M).
\end{align*}

%Note that by Theorem \ref{thm_expansion} and Lemma \ref{lem_sym_block}, Lemma \ref{lem_translation_isomorphism},
%%$\PI_A$ is naturally isomorphic to the functor $\CB(U_A)$.
%%More explicitly,
%we have linear isomorphisms
%\begin{align*}
%d_A:(g_A\PI_{w_A})\binom{M_0}{M_{[r]}} \rightarrow \PI_A\binom{M_0}{M_{[r]}}\\
%s_A: \PI_A\binom{M_0}{\Mrr} \rightarrow \CB_{M_{[0;r]}^{r_A}}(U_A),\\
%Q_{r_A}: \CB_{\Mr}(U_A) \rightarrow \CB_{M_{[0;r]}^{r_A}}(U_A),\\
%e_A:\CB_{M_{[r]}^{r_A}}(U_A)\rightarrow \PI_A\binom{M_0}{\Mrr}
%\end{align*}
%which are natural for $M_0,M_1,\dots,M_r\in \Vmodf$.
For a morphism $g:A\rightarrow A'$ in $\CPaB(r)$,
let $\ga \in [g]_Q$
and
define a linear map $A_\ga(Q):\CB_{\Mr}(U_A)\rightarrow \CB_{\Mr}(U_{A'})$
by the analytic continuation along the path $\ga$,
which is natural with respect to $V$-module homomorphisms $f_i: N_i \rightarrow M_i$ and $g:M_0\rightarrow N_0$  (see \cite{M6}).
Since $\CB$ is a locally constant sheaf, $A_\ga$ is well-defined and
independent of the choice of $\ga\in [g]_Q$.
Moreover, if $Q_0,Q_1:\Tr_r\rightarrow \Xr$ satisfy (Q1) and (Q2), then $\CB_{Q_0}=\CB_{Q_1}$ as functors.
Thus, we simply denote $\CB_Q$ by $\CB$, which we call a chiral conformal block.
%Define a natural transformation $\rho_Q(g):\rho_Q(A) \rightarrow \rho_Q(A')$ by
%the composition
%\begin{align}
%\begin{split}
%\rho(A)&\stackrel{d_A}{\rightarrow}
%\PI_A\binom{M_0}{\Mrr}
%\stackrel{s_A}{\rightarrow}
%\CB_{M_{[r]}^{r_A}}(U_A)
%\stackrel{Q_{r_A}^{-1}}{\rightarrow}
%\CB_{\Mr}(U_A)\\
%&\stackrel{A_{[g]_Q}}{\rightarrow} \CB_{\Mr}(U_{A'})
%\stackrel{Q_{r_{A'}}}{\rightarrow}
%\CB_{M_{[r]}^{r_{A'}}}(U_{A'})
%\stackrel{e_{A'}}{\rightarrow}
%\PI_{A'}\binom{M_0}{\Mrr}
%\stackrel{d_{A'}^{-1}}{\rightarrow}
%\rho(A').
%\end{split}
%\label{eq_def_rho_comp}
%\end{align}

Now we have: 
%obtained a family of functors $\{\CB:\CPaB(r)\rightarrow \Endp_\Vmodf(r)\}_{r\geq 0}$ and 
%the maps 
%$\comp:\CB(A)\circ(\CB(B_1),\dots,\CB(B_n))\rightarrow \CB\left( A\circ (B_1,\dots,B_n)\right)$ defined in \eqref{eq_comp_def}.
\begin{thm}\cite[Proposition 6.17 and Proposition 6.18]{M6}
\label{thm_glue}
Let \(r\ge 1\), \(s\ge 0\), $p \in \{1,\dots,r\}$,
\(A,A'\in \Tr_r\), \(B,B'\in \Tr_s\),
and let \(g_A:A\to A'\), \(g_B:B\to B'\) be morphisms in \(\CPaB\).
%Let $n\geq 1$ and $m\geq 0$, 
%$A,A'\in \Tr_n$, $B\in \Tr_m$
%and $\CPaB$-morphisms $g:A\rightarrow A'$, $g_B:B\rightarrow B'$ and $p \in \{1,\dots,n\}$
Then, the following diagram commutes:
\begin{align}
\begin{split}
\begin{array}{ccc}
\CB_{\Mr}(U_A) \otimes \CB_{\Ms}(U_B)
      &\overset{\comp_{p}}{\longrightarrow}&
\CB_{M_{[r+s-1]}^{A,B}}(U_{A\circ_p B})
    \\
    {}^{{\rho(g_A)\otimes \rho(g_B)}}\downarrow 
      && 
    \downarrow^{{\rho(g_A\circ_p g_B)}}
    \\
\CB_{\Mr}(U_{A'}) \otimes \CB_{\Ms}(U_{B'})
      &\overset{\comp_{p}}{\longrightarrow}&
\CB_{M_{[r+s-1]}^{A',B'}}(U_{A'\circ_p B'})
\label{eq_lax_diagram_lem}
\end{array}
\end{split}
\end{align}
\end{thm}

To formulate the above theorem, in terms of operads, we recall the notion of proendomorphism operad from \cite{M6}.
Let $C$ be a small $\C$-linear category. Then, the proendomorphism operad is the sequence
\begin{align}
\Endp_C (n)= \Func (C \times (C^\op)^r, \Vect)\label{eq_proend}
\end{align}
with the composition is defined by the coend of functors, $\int_C$ (for the definition and properties of coends see \cite{Mac}),
\begin{align}
F \circ_p G = \int_{N\in C} F(\bullet_0;\bullet_1,\dots,\bullet_{p-1},N, \bullet_{P+1},\dots,\bullet_n)\otimes G(N;\bullet_{n},\dots,\bullet_{n+m-1}) \in \Endp(n+m-1).
\label{eq_proend_comp}
\end{align}
for $F \in \Endp_C(n)$ and $G \in \Endp_C(m)$ with $p\in [n]$.
Then, $\{\Endp_C (n)\}_{n\geq 0}$ is a 2-operad with the associative isomorphism is defined by the universality of the coend (for more detail see \cite[Section 2.3]{M6}).

A conformal block is a functor
\begin{align*}
\CB_{\bullet_0;\bullet_1,\dots,\bullet_r}(U_A):\Vmodf \times (\Vmodfo)^r \rightarrow \Vect
\end{align*}
for $A\in \Tr_{[r]}$.
Hence, we can define a functor
\begin{align}
\CB:\CPaB(n) &\rightarrow \Endp(n)\\
A \mapsto \CB_{\bullet_0;\bullet_{[r]}}(U_A)\quad&\text{ and }\quad \ga:A \rightarrow A'  \mapsto A(\ga)
\end{align}
and, by Theorem \ref{thm_glue} and the universality of the coend \cite{Mac},
we have a linear map:
\begin{align}
\comp_{p}:\int_{N \in \Vmodf}
\CB_{\Mr}(U_A) \otimes \CB_{\Ms}(U_B)
      {\longrightarrow}
\CB_{M_{[r+s-1]}^{A,B}}(U_{A\circ_p B}).
\label{eq_glue}
\end{align}
Thus, we have:
\begin{thm}\cite[Theorem 6.19]{M6}
\label{thm_action}
Let $V$ be a positive graded vertex operator algebra.
Then, the pair of the functors $\CB:\CPaB \rightarrow \Endp_\Vmodf$
and the natural transformations $\comp_{p}: \CB_{A}\circ_p \CB_{B}\rightarrow \CB_{A \circ_p B}$ ($A,B\in \Tr_{*}$) is a lax 2-morphism of 2-operads.
%Moreover, the lax 2-morphism satisfies the assumption (1) in Proposition \ref{prop_representative}.
\end{thm}

\section{Vertex operator algebra and homotopy 2-Swiss-Cheese operad}
\label{sec_homSC}
In this section, we consider 2-colored operads colored by $\{c,o\}$, where $c$ (resp. $o$) stands for ``closed'' (resp. ``open''). 
In Section \ref{sec_2magma}, we will recall the definition of the parenthesized permutation and braid operad $\PaPB$ which is introduced by Idrissi \cite{Id}.
In Section \ref{sec_SC}, we will construct an action of the 2-colored operad $\PaPB$ on the representation category of a vertex operator algebra $\Vmodf$.
%They are deeply related with open-closed string in string theory \cite{}.

%\begin{itemize}
%\item
%notation を決めよう！
%\end{itemize}
\subsection{2-colored magma, trees and braids}
\label{sec_2magma}
Here, we review the definition of 2-colored operads $\Om\Om$ and $\PaPB$ from \cite{Id}
and give an explicit description of $\Om\Om$ by trees as in Section \ref{sec_magma}  (see \cite{Id} for more details).

A colored operad or a symmetric multicategory was introduced in \cite{Lam69}.
An element of a colored operad $\mathcal{O}$ with the colors $\{c,o\}$ of $(r,s)$-array operation has inputs labeled with r $c$'s and s $o$'s and the output labeled with $c$ or $o$. We denote the set of such operations with output labeled by $x \in \{c,o\}$ by:
\begin{align*}
\mathcal{O}^x(r,s) = \mathcal{O}(\underbrace{c,\dots,c}_{r},\underbrace{o,\dots,o}_{s};x).
\end{align*}
Let $V_c,V_o$ be vector spaces. Then, the endomorphism 2-colored operad is given by
\begin{align}
\End^c(r,s) &= \mathrm{Hom}(V_c^{\otimes r}\otimes V_o^{\otimes s}, V_c)\\
\End^o(r,s) &= \mathrm{Hom}(V_c^{\otimes r}\otimes V_o^{\otimes s}, V_o)
\end{align}
with obvious compositions.

In \cite{Id}, $\Om\Om$ is introduced as the free colored operad $O(\mu_c,\iota,\mu_o)$ on the three generators
$\cdot_c \in \Om\Om(c,c;c)$, $\iota \in \Om(c;o)$ and $\cdot_o \in \Om\Om(o,o;o)$.
An algebra over $\Om\Om$ in
$\underline{\text{Set}}$ is the data of
\begin{itemize}
\item
Sets $S_c$ and $S_o$ with maps $\cdot_c:S_c\times S_c \rightarrow S_c$ and $\cdot_o:S_o \times S_o \rightarrow S_o$,
i.e., a pair of magmas;
\item
A map $\iota:M_c \rightarrow M_o$ (not necessarily compatible with the products).
\end{itemize}
An element of $\Om\Om$ has, for example, the following form:
\begin{align*}
(\tau({\bf 2})\cdot_o 4 )\cdot_o (\tau({\bf 3}\cdot_c {\bf 1}) \cdot_o 5),
\end{align*}
where the bold numbers correspond to the color ``c" and the normal numbers correspond to the color ``o".
We can naturally associate a tree with leaves labeled by bold and normal numbers as in Fig \ref{fig_2tree}.

\begin{minipage}[l]{.4\textwidth}
\centering
\begin{forest}
for tree={
  l sep=20pt,
  parent anchor=south,
  align=center
}
[o,
[[{\bf2},edge label={node[midway,left]{$\tau$}}][4]]
[[{},edge label={node[midway,left]{$\tau$}}[{\bf3}][{\bf1}]]
[5]]
]
\end{forest}
\captionof{figure}{}
\label{fig_2tree}
\end{minipage}
\begin{minipage}[l]{.5\textwidth}
%\begin{figure}[t]
%    \centering
    \includegraphics[width=7cm]{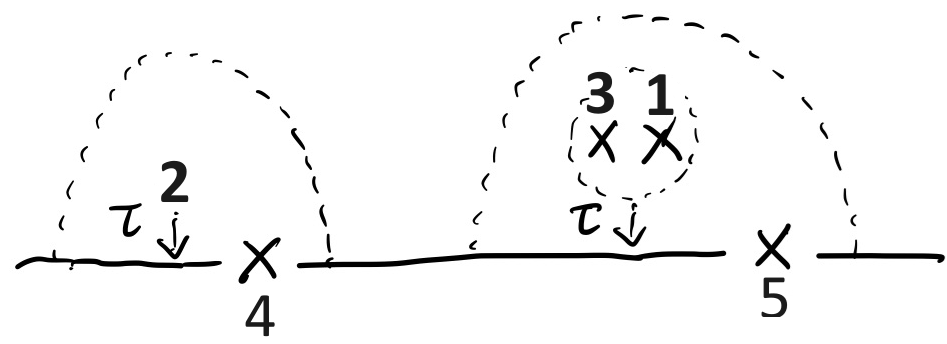}
%    \caption{$\alpha_c$}\label{fig_al_c}
%\end{figure}
Configuration of points in $X_{3,2}(\HH)$ corresponding to Fig \ref{fig_2tree}.
\end{minipage}

\begin{itemize}
\item
$\Tr^c(r)$ is a set of all binary trees whose leaves are labeled by bold numbers $\{\bf 1,2,\dots,r\}$.
\item
$\Tr^o(r,s)$ is a set of all binary trees whose leaves are labeled by bold numbers $\{\bf 1,2,\dots,r\}$
and normal numbers $\{r+1,\dots,r+s\}$ such that 
the labels on normal (open) leaves are arranged so that the numbers increase from left to right.
%leaves labeled by normal numbers are ordered by 
%of leaves on the right side are larger in number than those on its left side.
%normal numbers 
%open leaf については order を保ったもののみを考える.
\item
We also consider the map corresponding to $\tau$ by the natural embedding:
\begin{align}
\Tr^c(r) \rightarrow \Tr^o(r,0). \label{eq_tr_co}
\end{align}
\end{itemize}

Then, $\{\Tr^c(r),\Tr^o(r,s)\}_{r,s\geq 0}$ form a two-colored operad with obvious composition of trees,
which is symmetric in the $c$-labels and not symmetric in the $o$-labels.
The following proposition is clear:
\begin{prop}
The colored operad $\{\Tr^c(r),\Tr^o(r,s)\}_{r,s\geq 0}$ is isomorphic to $\Om\Om$ as a 2-colored operad.
\end{prop}

%$\Om\Om$ is defined as a free colored operad from the generators, but in \cite{Id} 
%An explicit way to write the elements of $\Om\Om$ are given in \cite{Id}.
%Here, we consider another explicit description, which is convenient for the purpose of this paper.

Set
\begin{align}
X_{r,s}(\HH) &=
\{(z_1,\ldots,z_{r+s})\in \mathbb C^{r+s}
\mid \mathrm{Im}\, z_i>0\ (1\le i\le r),\mathrm{Im}\, z_j=0\ (r<j\le r+s),\ 
z_a\ne z_b\ (a\ne b)\}.
% \bigl\{(z_1,\dots,z_r,z_{r+1},\dots,z_{r+s}) \in \C^{r+s} \mid \\
%&\mathrm{Im}\, z_i>0 \text{ and } \mathrm{Im}\,z_j=0,
%z_i\neq z_j \quad \text{for all } i\neq j
%\text{ for }i\leq r, j>r\bigr\}.
\end{align}
The operad of {\it parenthesized permutations and braids} $\PaPB$ introduced in \cite{Id} is a 2-colored operad
in the category of categories, which is defined as follows:
\begin{figure}[h]
  \begin{minipage}[c]{0.7\linewidth}
    \centering
    \includegraphics[width=7cm]{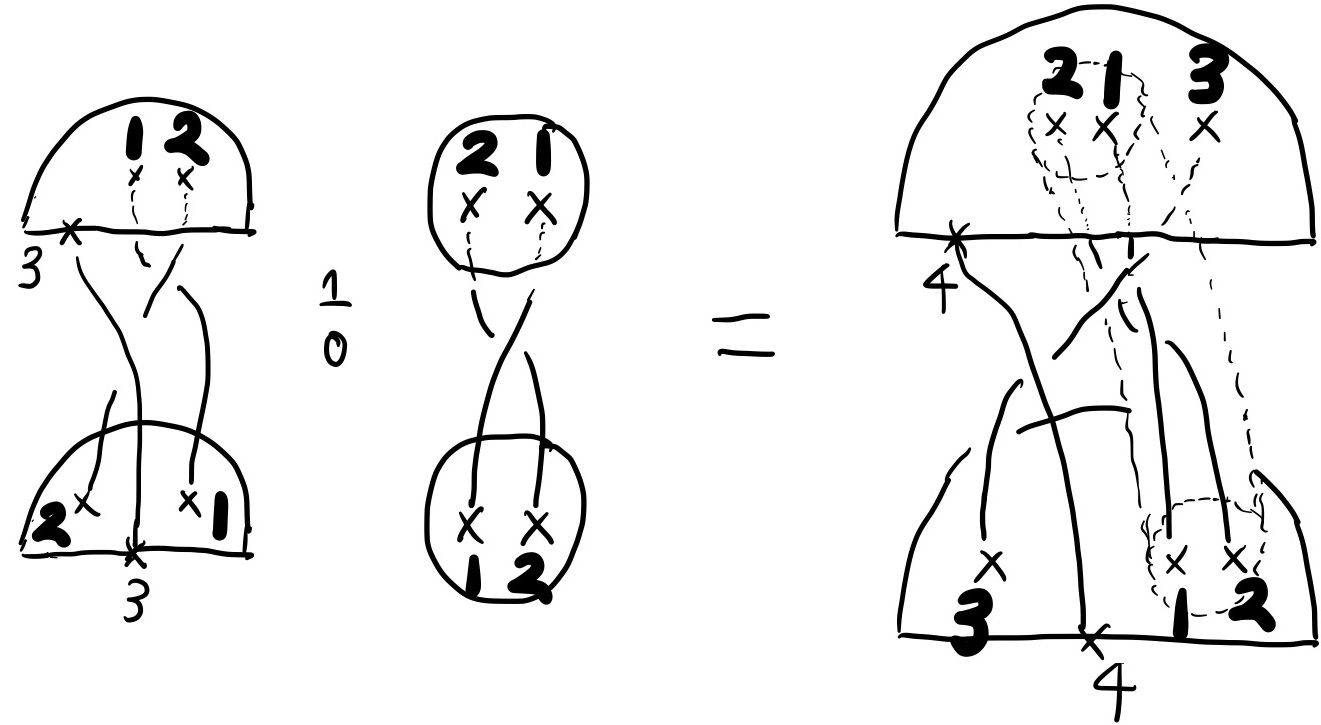}
    \caption{operad structure}\label{fig_comp}
  \end{minipage}
\end{figure}
As in Section \ref{sec_CPaB}, consider maps 
\begin{align*}
Q^c:\Tr^c(r)\rightarrow X_r(\R)\quad\text{and }\quad Q^o:\Tr^o(r,s) \rightarrow X_{r,s}(\HH)
\end{align*}
such that:
\begin{itemize}
\item[QB1)]
The order of $Q^c(A)$ and $A$ is equal for all $A \in \Tr^c$;
\item[QB2)]
The order of real parts of $Q^o(E)$ and $E$ is equal.
\end{itemize}
Then, $\PaPB^c(r)$ is a full subcategory of $\Pi_1(\Xr)$ whose objects are $\{Q^c(A)\}_{A \in \Tr^c(r)}$
and $\PaPB^o(r,s)$ is a full subcategory of $\Pi_1(X_{r,s}(\HH))$ whose objects are $\{Q^o(E)\}_{E \in \Tr^o(r,s)}$
which equipped with functors
\begin{align}
\tau(r):\PaPB^c(r) \rightarrow \PaPB^o(r,0). \label{eq_tau_CPaB}
\end{align}
Here, the functor $\tau(r)$ sends objects by \eqref{eq_tr_co}.
By (QB1) and (QB2), morphisms of $\CPaB^c(r) =\CPaB(r)$ canonically correspond to morphisms of 
$\PaPB^o(r,0)$.

An example of a composition of $\PaPB^o$ is given in Fig \ref{fig_comp},
which is compatible with \eqref{eq_tau_CPaB} and the operad structure on $\PaPB^c=\CPaB$.
Thus, $\PaPB$ is a 2-colored operad (see \cite{Id} for more details).
It is noteworthy that $\PaPB$ is equivalent to the fundamental groupoid of the Swiss-Cheese operad
$\Pi_1(\mathrm{SC}_2)$ \cite[Theorem 3.10]{Id} (see also \cite{Fr}).

%合成の例を一個くらい書いておけばよい
%\begin{figure}[t]
%  \begin{minipage}[b]{0.45\linewidth}
%    \centering
%    \includegraphics[width=2.5cm]{sigma.jpg}
%    \caption{morphism $\sigma$}\label{fig_sigma}
%  \end{minipage}
%  \begin{minipage}[b]{0.45\linewidth}
%    \centering
%    \includegraphics[width=3cm]{alpha.jpg}
%    \caption{morphism $\alpha$}\label{fig_alpha}
%  \end{minipage}
%      \begin{minipage}[b]{0.45\linewidth}
%    \centering
%    \includegraphics[width=3.2cm]{alpha_2_sigma.jpg}
%    \caption{morphism $\alpha\circ_2\sigma$}\label{fig_alphasigma}
%  \end{minipage}
%    \begin{minipage}[b]{0.45\linewidth}
%    \centering
%    \includegraphics[width=3.2cm]{sigma_2_alpha.jpg}
%    \caption{morphism $\alpha\circ_2\sigma$}\label{fig_sigmaalpha}
%  \end{minipage}
%\end{figure}

\subsection{Conformal block on $\HH$ and homotopy action of Swiss-Cheese operad}
\label{sec_SC}
Let $C,D$ be small $\C$-linear categories.
The definition of a proendomorphism operad can be generalized to the colored case;
\begin{align*}
\Endp^o(r,s)&=\Func(C \times (D^\op)^r \times (C^\op)^s, \Vect) \\
\Endp^c(r,s)&=\Func(D \times (D^\op)^r \times (C^\op)^s, \Vect),
\end{align*}
which is a 2-colored 2-operad similarly to \eqref{eq_proend} and \eqref{eq_proend_comp}.

We consider the case of $C=\Vmodf$ (resp. $D=\Vmodft$), which corresponds to ``open strings / boundary states''  (resp. closed strings/ bulk states)
and construct functors
\begin{align*}
\CB^c: &\CPaB(r) \rightarrow \Endp_D^c(r),\\
\CB^o: &\PaPB^o(r,s) \rightarrow \Endp_{C,D}^o(r,s),
\end{align*}
which are shown to be a lax 2-morphism of colored 2-operads.

\vspace{3mm}

\textbf{(Definition of functors)}

\vspace{2mm}

Let $\Phi:X_{r,s}(\HH) \rightarrow X_{2r+s}(\C)$ be a continuous embedding given by
\begin{align*}
\Phi:X_{r,s}(\HH) \rightarrow X_{2r+s}(\C),\quad
(z_1,\dots,z_r,z_{r+1},\dots,z_{r+s}) &\mapsto (z_1,\bar{z}_1,\dots,z_r,\z_r,z_{r+1},\dots,z_{r+s}).
\end{align*}
Correspondingly, we define the following map that embeds o-colored trees $\Tr^o(r,s)$ into $\Tr_{2r+s}$ as follows:
\begin{align*}
\tilde{\bullet}:\Tr^o(r,s) \rightarrow \Tr_{2r+s},\quad E \mapsto \tilde{E}.
\end{align*}
%対応して, o-colored tree $\Tr^o(r,s)$を$(2r+s)$個の樹に埋め込む次の写像を以下のように定義する:
\begin{minipage}[l]{.4\textwidth}
\centering
\begin{forest}
for tree={
  l sep=20pt,
  parent anchor=south,
  align=center
}
[o,
[[{\bf2},edge label={node[midway,left]{$\tau$}}][4]]
[[{},edge label={node[midway,left]{$\tau$}}[{\bf3}][{\bf1}]]
[5]]
]
\end{forest}
\captionof{figure}{}
\label{fig_double}
\end{minipage}
\begin{minipage}[l]{.6\textwidth}
\centering
\begin{forest}
for tree={
  l sep=20pt,
  parent anchor=south,
  align=center
}
[
[[[{\bf2}][$\bar{\bf 2}$]][4]]
[[{}[[{\bf3}][{\bf1}]][[$\bar{\bf3}$][$\bar{\bf1}$]] ]
[5]]
]
\end{forest}
%The corresponding tree $\tilde{E} \in \Tr_{8}$.
\end{minipage}
For a tree $E \in \Tr^o(r,s)$, 
consider double copies of the bottoms of all the edges labeled with $\tau$.
The label of the leaf of each double-copy is ${\bf n}$ on the left side and $\bar{\bf{n}}$ on the right side if the original label is ${\bf n}$ (see Fig. \ref{fig_double}).

%make a double copy of the tree directly below the edge labeled by $\tau$. Each double copy is assigned a conjugate label $\{{\bf n,\bar{n}}\}$. It is assumed that $\bar{\bf n}$ is always the leaf on the right side (see Fig. \ref{fig_double}).

%樹$E \in \Tr^o(r,s)$ に対して、$\tau$でlabel されている辺の下にある樹を二重にコピーして真下につける。二重のコピーにはそれぞれ共役な label $\{{\bf n,\bar{n}}\}$を割り振る。ただし$\bar{\bf n}$は常に右側の葉とする。

%An element in $\Tr^o(r,s)$ is a binary tree labeled by open and closed leaves
%such that closed part is duplicated (see Fig. \ref{fig_}).
%We can forget the colors and see $E \in \Tr^o(r,s)$ naturally as a binary tree with $2r+s$ leaves. 
%This identification is denoted by

For $A \in \Tr^c(r)$, define $\CB^c(A) \in \Endp_D^c(r)$ by
\begin{align*}
\CB^c(A): D \times (D^\op)^r &\rightarrow \Vect\\
(M_0,\bar{M}_0, M_1,\bar{M}_1,\dots,M_r,\bar{M}_r)
&\mapsto \CB_{\Mr}(U_A) \otimes \CB_{\bar{M}_{[0;r]}}(U_A)
\end{align*}
and for $E \in \Tr^o(r,s)$,
\begin{align*}
\CB^o(E): C \times  (D^\op)^r \times (C^\op)^s &\rightarrow \Vect\\
(N, M_1,\bar{M}_1,\dots,M_r,\overline{M}_r,M_{r+1},\dots,M_{r+s}) &\mapsto
\CB_{{M}_{[0;r,s]}}(U_{\tilde{E}}),
\end{align*}
where $\CB$ are the chiral conformal blocks in section \ref{sec_chiral_CB}
and $\CB_{{M}_{[0;r,s]}}(U_{\tilde{E}})$ is the chiral conformal block with $N$ covariantly inserted at the infinity and $M_i,\bar{M}_i, M_{r+j}$ contravariantly inserted at the corresponding $2r+s$ leaves of $\tilde{E} \in \Tr_{2r+s}$.
Note that elements of $\CB_{\Mr}(U_A) \otimes \CB_{\bar{M}_{[0;r]}}(U_A)$
are multivalued holomorphic functions on $X_r(\C) \times X_r(\C)$
and $\CB_{{M}_{[0;r,s]}}(U_{\tilde{E}})$ are multivalued holomorphic functions on $X_{2r+s}(\C)$.
\vspace{3mm}

%\begin{screen}
\fbox{
\begin{minipage}[l]{7cm}
\centering
\vspace{2mm}

$\CB^c: \PaPB^c(r) \rightarrow \Endp_D^c(r)$\\
\dotfill
\begin{align*}
\begin{array}{ccc}
%\PaPB^c(r) &\overset{\CB^c}{\longrightarrow}& \Endp_D^c(r)\\
A&& \CB_{\bullet}(U_A) \otimes \CB_{\bullet}(U_A)\\
\downarrow^{\ga} & \longmapsto &\downarrow_{A(\ga)\otimes A(\bar{\ga})} \\ 
A' &&\CB_{\bullet}(U_{A'}) \otimes \CB_{\bullet}(U_{A'}).
\end{array}
\end{align*}
\end{minipage}
}
%\end{screen}
\fbox{
\begin{minipage}[l]{7cm}
\centering
\vspace{2mm}

$\CB^o:\PaPB^o(r,s) \rightarrow \Endp_{C,D}^o(r,s)$\\
\dotfill
\begin{align*}
\begin{array}{ccc}
%\PaPB^c(r) &\overset{\CB^c}{\longrightarrow}& \Endp_D^c(r)\\
E&& \CB_{\bullet}(U_{\tilde{E}})\\
\downarrow^{\mu} & \longmapsto& \downarrow_{A(\Phi_*\mu)}\\ 
E' &&\CB_{\bullet}(U_{\tilde{E'}}).
\end{array}
\end{align*}
\end{minipage}
}
\vspace{3mm}

For $A,A' \in \Tr^c(r)$ and a path $\ga:A \rightarrow A'$ in $\PaPB^c(r) = \CPaB(r)$, let $\bar{\ga}$ be the complex conjugate of $\ga$ in $X_r(\C) \subset \C^r$.
Since we take the base points $Q(A)$ in $X_r(\R)$, $\bar{\ga}$ is again a path in $\CPaB(r)$. Define a map $\CB^c(A) \rightarrow \CB^c(A')$ by the analytic continuation along the path $\ga \times \bar{\ga}$ in $X_r(\C) \times X_r(\C)$.

Let $E,E' \in \Tr^o(r,s)$ and $\mu:E \rightarrow E'$ be a path in $\PaPB^o(r,s)$.
Then, $\Phi_*(\mu)$, the pushforward of the path by $\Phi:X_{r,s}(\HH) \rightarrow X_{2r+s}(\C)$, is a path in $X_{2r+s}(\C)$.
Define a map $\CB^o(E) \rightarrow \CB^o(E')$ by the analytic continuation along the path $\Phi_* \mu$ in $X_{2r+s}(\C)$.

Note that for $\iota \in \Om\Om(c;o) = \Tr^o(1,0)$,
\begin{align*}
\CB(\iota): C \times D^\op \rightarrow \Vect,\quad\quad
(M_0,M_1,\bar{M}_1) \mapsto I_{\log} \binom{M_0}{M_1\bar{M}_1},
\end{align*}
the space of intertwining operators by Proposition \ref{prop_int}.

\vspace{3mm}

\textbf{(Definition of compositions)}\\
We will define the compositions of 2-colored operad on $\CB^\bullet$.
There are three types of composites, closed-closed, open-open, and open-closed, which are defined as follows:
\begin{description}
\item[closed-closed)]
$\circ_p:\PaPB^c(r) \times \PaPB^c(t) \rightarrow \PaPB^c(r+t-1), (A,B)\mapsto A\circ_p B$ is given by
\begin{align}
\comp_p \times \comp_p:\CB(U_A)\otimes \CB(U_A) \times \CB(U_B)\otimes \CB(U_B)
{\rightarrow}
\CB(U_{A\circ_p B})\otimes \CB(U_{A\circ_p B}),
\end{align}
where $\comp_p$ is the glueing map of chiral conformal blocks given in \eqref{eq_glue},
and we glued the right and left conformal blocks independently.
\item[open-open)]
$\circ_p:\PaPB^o(r,s) \times \PaPB^o(t,u) \rightarrow \PaPB^o(r+t,s+u-1), (E,F)\mapsto E\circ_p F$
with an open leaf $p$ of $E \in \Tr^o(r,s)$ is given by
\begin{align}
\comp_p: \CB(U_{\tilde{E}}) \times \CB(U_{\tilde{F}}) \rightarrow \CB(U_{\tilde{E}\circ_p \tilde{F}})=\CB(U_{\widetilde{E \circ_p^o F}}).
\end{align}
Here, $\tilde{E}\circ_p \tilde{F}$ is a binary tree in $\Tr_{(2r+s)+(2t+u)-1}$
and $E \circ_p^o F$ is a colored tree in $\Tr^o(r+t,s+u-1)$.
It is clear that
\begin{align*}
\tilde{E}\circ_p \tilde{F} = \widetilde{E\circ_p F}.
\end{align*}
holds in $\Tr_{(2r+s)+(2t+u)-1}$.
\item[open-closed)]
$\circ_p:\PaPB^o(r,s) \times \PaPB^c(t) \rightarrow \PaPB^o(r+t-1,s), (E,A)\mapsto E\circ_p A$
with a closed leaf $p$ of $E \in \Tr^o(r,s)$ is given by
\begin{align*}
\comp_p\circ \comp_{\bar{p}}:\CB(U_{\tilde{E}}) \times\CB(U_A)\otimes \CB(U_A) \rightarrow \CB(U_{\tilde{E}\circ_{(p,\bar{p})}(A,A)})
=\CB(U_{\widetilde{E \circ_p^c A}}).
\end{align*}

\begin{minipage}[l]{.6\textwidth}
The closed leaves are always in pairs of $(p,\bar{p})$. We denote the closed leaf on the left side $p$. 
We insert the same tree $A$ at $p,\bar{p}$ of the tree $\tilde{E} \in \Tr_{2r+s}$,
which is denoted by $\tilde{E}\circ_{(p,\bar{p})}(A,A) \in \Tr_{2r+s+2t-2}$.
For example,
$\widetilde{(\tau({\bf2}) \circ_o 1)} \circ_{\bf2,\bar{\bf2}} ((2 \circ_c 1),(2 \circ_c 1))$ in $\Tr_{5}$ is describe in Fig. \ref{fig_open_closed_tree_insert}.
\end{minipage}
\begin{minipage}[r]{.27\textwidth}
\begin{forest}
for tree={
  l sep=20pt,
  parent anchor=south,
  align=center
}
[
[[,edge label={node[midway,left]{$\circ$}}[{\bf3}][{{\bf2}}]] [,edge label={node[midway,right]{$\bar{\circ}$}}[$\bar{{\bf3}}$][$\bar{{\bf2}}$]]]
[[[1]]]
]
\end{forest}
\captionof{figure}{}
\label{fig_open_closed_tree_insert}
\end{minipage}

It is clear that 
\begin{align*}
\tilde{E}\circ_{(p,\bar{p})}(A,A) =\widetilde{E \circ_p^c A}
\end{align*}
holds in $\Tr_{2r+s+2t-2}$.
\end{description}

Then, we have the following theorem (see \cite[Definition 2.12]{M6} for the definition of lax 2-morphism).
\begin{thm}
\label{thm_SC_action}
Let $V$ be a positive graded vertex operator algebra.
Then, 
\begin{align*}
\CB: \PaPB^\bullet(r,s) \rightarrow \Endp^\bullet(r,s), \quad\quad\bullet \in \{c,o\},r,s \geq 0
\end{align*}
is a lax 2-morphism of colored 2-operads.
%Moreover, the lax 2-morphism satisfies the assumption (1) in Proposition \ref{prop_representative}.
\end{thm}
\begin{proof}
It suffices to show that the above three compositions are compatible with analytic continuations along the paths, i.e., the diagrams in Theorem \ref{thm_glue} for $\PaPB$ commute.
In the closed-closed case, it follows immediately from Theorem \ref{thm_glue}.

In the open-open case, for paths $\mu:E \rightarrow E'$ in $\PaPB^o(r,s)$
 and $\nu:F \rightarrow F'$ in $\PaPB^o(t,u)$,
by Theorem \ref{thm_glue}, the following diagram commutes:
\begin{align*}
\begin{array}{ccc}
\CB(U_{\tilde{E}}) \times \CB(U_{\tilde{F}}) 
&\overset{\comp_p}{\rightarrow} &\CB(U_{\tilde{E}\circ_p \tilde{F}})\\
\downarrow_{A(\Phi_*\mu)\otimes A(\Phi_*\nu)} & &\downarrow_{
A(\Phi_*\mu \circ_p \Phi_*\nu)} \\ 
\CB(U_{\tilde{E'}}) \times \CB(U_{\tilde{F'}}) 
&\overset{\comp_p}{\rightarrow} &\CB(U_{\tilde{E'}\circ_p \tilde{F'}}).
\end{array}
\end{align*}
Since 
\begin{align*}
\Phi_*\mu \circ_p \Phi_*\nu = \Phi_*(\mu\circ_p^o \nu)
\end{align*}
up to homotopy in $X_{(2r+s)+(2t+u)-1}(\C)$,
where $\Phi_*\mu \circ_p \Phi_*\nu$ is the composition in $\CPaB$
and $\mu\circ_p^o \nu$ is the composition in $\PaPB$,
the assertion holds.

In the open-closed case, for paths $\mu:E \rightarrow E'$ in $\PaPB^o(r,s)$
 and $\ga:A \rightarrow A'$ in $\PaPB^c(t)$,
by Theorem \ref{thm_glue}, the following diagram commutes:
\begin{align*}
\begin{array}{ccc}
\CB(U_{\tilde{E}}) \times (\CB(U_{A})\otimes \CB(U_A))
&\overset{\comp_p}{\rightarrow} &\CB(U_{\tilde{E}\circ_{p,\bar{p}}(A,A)})\\
\downarrow_{A(\Phi_*\mu)\otimes A(\ga)\otimes A(\bar{\ga})} & &\downarrow_{
A(\Phi_*\mu \circ_{p,\bar{p}}(\ga,\bar{\ga})}\\ 
\CB(U_{\tilde{E'}}) \times (\CB(U_{A'})\otimes \CB(U_{A'}))
&\overset{\comp_p}{\rightarrow} &\CB(U_{\tilde{E'}\circ_{p,\bar{p}}(A',A')}).
\end{array}
\end{align*}
Since 
\begin{align*}
\Phi_*\mu \circ_{p,\bar{p}}(\ga,\bar{\ga}) = \Phi_*(\mu\circ_p^o \ga)
\end{align*}
up to homotopy in $X_{(2r+s)+2t-2}(\C)$,
and $\mu\circ_p^o \ga$ is the composition in $\PaPB$,
the assertion holds.
\end{proof}

We have given conformal blocks as a function on open regions $U_{\tilde{E}}$ in $X_{2r+s}(\C)$.
As we will see in the next section, the correlation functions of conformal field theory are single valued function on $X_{r,s}(\HH)$ and not $X_{2r+s}(\C)$.
We end this section by defining the open regions where correlation functions are actually defined.

%我々は共形ブロックを $X_{2r+s}(\C)$ 上のopen region $U_{\tilde{E}}$上の多価関数として与えた。
%次の章で見るように共形場理論の相関関数は、$X_{r,s}(\HH) \subset X_{2r+s}(\C)$上の一価関数である。
%最後に開領域$U_{\tilde{E}}$ の$X_{r,s}(\HH)$への制限などを定義する.

For $A \in \Tr^c(r)$, set $U_A^c = \overline{U}_A$, where $\overline{U}_{A}$ is the open region without branch cut at the origins defined in \eqref{eq_no_cut}. We regard $U_A^c$ as a subset of $\overline{U}_{A}\times \overline{U}_{A} \subset X_r(\C) \times X_r(\C)$ by
\begin{align}
U_A^c \hookrightarrow \overline{U}_{A}\times \overline{U}_{A},\quad z_{[r]} \mapsto (z_{[r]},\z_{[r]}). \label{eq_UAc}
\end{align}

For an open tree $E \in \Tr^o(r,s)$, set
\begin{align}
U_E^o &= \Phi^{-1}(\overline{U}_{\tilde{E}}) \cap X_{r,s}(\HH) \label{eq_UEo} \\ 
&= \{w=(z_1,\dots,z_r,z_{r+1},\dots,z_{r+s}) \in X_{r,s}(\HH)\mid \Phi(w) \in \overline{U}_{\tilde{E}}\},\nonumber
\end{align}
which is regarded as a domain in \(X_{r,s}(\HH)\).

\section{Algebra of bulk-boundary 2d CFT and consistency}
\label{sec_consistency}
The bulk OPE algebra of 2d conformal field theory was introduced by Huang and Kong \cite{HK1}.
Later, it is generalized into the bulk-boundary OPE algebra by Kong in \cite{Ko1,Ko3}.
The following is different from the data originally given by Kong, but is equivalent under the assumption of local $C_1$-cofiniteness as we will see in this paper.
In Kong's definition, the open-closed operator is replaced by a linear map $Y_{\text{bulk-bdy}}(\bullet,\uz):F^c \rightarrow \End F^o[[z,\z]]$.
%次はKongによる元々の定義データとは異なるが、本論文で見るようにlocal $C_1$-cofinite の仮定の下同値である(Kongの定義では open-closed operator が linear map $Y^c(\bullet,\uz):F^c \rightarrow \End F^o[[z,\z]]$として与えられている)。
In this paper, boundary conformal field theory is given by the following data:
\begin{description}
\item[state space)]
An $\R^2$-graded vector space $F^c=\bigoplus_{h,\h \in \R}F_{h,\h}^c$ and 
an $\R$-graded vector space $F^o=\bigoplus_{h \in \R}F_{h}^o$;
\item[closed-closed operator)] A linear map with formal variables $z,\z$
\begin{align*}
Y^c(\bullet,\uz):F^c \rightarrow \End F^c[[z,\z,|z|^\R]],\quad
a \mapsto Y(a,\uz)=\sum_{r,s\in \R}a(r,s)z^{-r-1}\z^{-s-1};
\end{align*}
\item[open-open operator)] A linear map with a formal variable $x$
\begin{align*}
Y^o(\bullet,z):F^o &\rightarrow \End F^o[[z^\R]],\quad v \mapsto Y(v,z)=\sum_{r\in \R}v(r)z^{-r-1};
\end{align*}
\item[{open-closed operator})] {A linear map} with a formal variable $z$
\begin{align*}
Y^b(\bullet,z): F^c \rightarrow F^o[[z^\R]],\quad
a \mapsto Y^b(a,z)=\sum_{r\in \R} B_r(a) z^{-r-1};
\end{align*}
\item[vacuum vector)]
$\va^c \in F_{0,0}^c$ and $\va^o \in F_0^o$;
\item[Virasoro elements)] $\om \in F_{2,0}^c$, $\omb \in F_{0,2}^c$ and $\nu \in F_2^o$.
\end{description}
Those three OPEs $Y^c,Y^o,Y^b$ classically correspond to the three generators of 2-colored magma in Section \ref{sec_2magma}.

Using the action of $\PaPB$ on conformal blocks, we show in this section that for any 2-colored tree $E\in \Tr^o(r,s)$, the three OPEs converge in the corresponding open region  in the upper half-plane $U_E^o \subset X_{r,s}(\HH)$ and define the correlation functions which are independent of orders and parentheses of OPEs.

In Section \ref{sec_full_VOA}, we will study the consistency of a bulk algebra $(F^c,Y^c,\va^c)$ for any binary tree $\Tr^c$.
In section \ref{sec_boundary} and section \ref{sec_bulk_boundary},the consistency of a boundary algebra $(F^o,Y^o,\va^o)$ and a bulk-boundary algebra $(F^c,F^o,Y^c,Y^o,Y^b)$ will be studied.

\subsection{Bulk OPE algebra}
\label{sec_full_VOA}
A mathematical formulation of the OPE algebra in a 2d conformal field theory without boundary was introduced in \cite{HK1}.
We focus on the consistency of OPEs from the bootstrap equation.
We will look back at a slightly different reformulation given in \cite{M1} based on the bootstrap equation
 and formulate and prove the consistency of a bulk conformal field theory based on Section \ref{sec_homSC}.

A full vertex operator algebra is a generalization of a vertex operator algebra as in the following table:
\begin{table}[htb]
\begin{tabular}{|l||c|c|}\hline
 & chiral vertex algebra & full vertex algebra \\ \hline \hline
symmetry & $\Vir_c$ & $\Vir_c \oplus \Vir_{\bar{c}}$ \\
vector space &  $V=\bigoplus_{n \in \Z} V_n$ & $F=\bigoplus_{h,\h \in \R^2} F_{h,\h}$ \\
vertex operator &$\End V[[z^\pm]]$ & $\End F[[z,\z,|z|^\R]]$ \\
pole & $\C((z))$ & $\C((z,\z,|z|^\R))$ \\ \hline
\end{tabular}
\vspace{2mm}
\caption{Comparison of chiral and full vertex algebras}
 \label{table_add}
\end{table}
Since the full conformal field theory has two mutually commuting Virasoro algebras acting on it, it has $\R^2$-grading by $L(0),\overline{L(0)}$.

%full共形場理論には互いに可換な二つのVirasoro代数が作用しているため、L(0),\overline{L(0)}固有値によりC^2-grading を持つ。

Let $F=\bigoplus_{h,\h\in \R}F_{h,\h}$ be an $\R^2$-graded vector space
and $L(0),\Ld(0):F\rightarrow F$ linear maps defined by
$L(0)|_{F_{h,\h}}=h \id_{F_{h,\h}}$ and 
$\Ld(0)|_{F_{h,\h}}=\h \id_{F_{h,\h}}$
for any $h,\h\in \R$.
We assume that:
\begin{enumerate}
\item[FO1)]
$F_{h,\h}=0$ unless $h-\h \in \Z$;
\item[FO2)]
There exists $N \in \R$ such that
$F_{h,\h}=0$ unless $h\geq N$ and $\h \geq N$;
\item[FO3)]
For any $H\in \R$,
$\bigoplus_{h+\h \leq H}F_{h,\h}$ is finite-dimensional.
\end{enumerate}
Set 
\begin{align*}
F^\vee =\bigoplus_{h,\h\in\R} F_{h,\h}^*,
\end{align*}
where $F_{h,\h}^*$ is the dual vector space.

We will use the notation $\uz$ for the pair $(z,\z)$ and $|z|^2$ for $z\z$.
For a vector space $V$,
we denote by $V[[z^\R,\z^\R]]$ the set of formal sums 
$$\sum_{r,s \in \R} a_{r,s}z^{r} \z^{s},$$
and
by $V[[z,\z,|z|^\R]]$
the subspace of $V[[z^\R,\z^\R]]$
such that
\begin{itemize}
\item
$a_{r,s}=0$ unless $r-s \in \Z$.
%\item
%$\{({s_1,\s_1,\dots,s_n,\s_n}) \in \R^{2n} \;|\; a_{s_1,\s_1,\dots,s_n,\s_n} \neq 0 \}$ is a countable set.
\end{itemize}
We also denote by
$V((z,\z,|z|^\R))$ the subspace of $V[[z,\z,|z|^\R]]$ spanned by the series satisfying
\begin{itemize}
\item
There exists $N \in \R$ such that
$a_{r,s}=0$ unless $r,s \geq N$;
\item
For any $H\in \R$,
$$\{(r,s) \;|\; a_{r,s}\neq 0 \text{ and }r+s \leq H \}$$
is a finite set.
\end{itemize}

The formal power series $\C((z,\z,|z|^\R))$ is a generalization of the Laurent series bounded below into two variables, which describes singularities of correlation functions in 2d compact conformal field theory.

A  {\it full vertex operator} on $F$ is a linear map
\begin{align*}
Y(\bullet, z,\z):F \rightarrow \mathrm{End}(F)[[z,\z,|z|^\R]],\; a\mapsto Y(a,z,\z)=\sum_{r,s \in \R}a(r,s)z^{-r-1}\z^{-s-1}
\end{align*}
such that:
\begin{align}
\begin{split}
[L(0),Y(a,\uz)]&= z\frac{d}{dz}Y(a,\uz) + Y(L(0)a,\uz),\\
[\Ld(0),Y(a,\uz)]&= \z\frac{d}{d\z}Y(a,\uz) + Y(\Ld(0)a,\uz).
\label{eq_L0_cov}
\end{split}
\end{align}
Then, by (FO1), (FO2) and (FO3), $Y(a,\uz)b \in F((z,\z,|z|^\R))$ (see \cite[Proposition 1.5]{M1}).

% $Y(a,\uz)$はEnd F-valuedなformal power series であるが、我々は$(\ze,\zee)$に複素数値を代入することにより、これを解析的関数と思いたい.
% $a_i \in F_{h_i,h_i}$ and $u\in F_{h_0}^\vee$に対して、
% $u(Y(a,\uz_1)Y(b,\uz_2)c)$は $\ze_1,\zee_1,\ze_2,\zee_2$に関する形式的級数である.
By \eqref{eq_L0_cov} (for more detail, see \cite[Lemma 1.6]{M1}), for $u \in F_{h_0,\h_0}^\vee$ and $a_i \in F_{h_i,\h_i}$ we have
\begin{align}
u(Y(a_1,\uz_1)Y(a_2,\uz_2)a_3) \in z_2^{h_0-h_1-h_2-h_3}\z_2^{\h_0-\h_1-\h_2-\h_3}\C\Bigl(\Bigl(\frac{z_2}{z_1},\frac{\z_2}{\z_1},\Bigl|\frac{z_2}{z_1}\Bigr|^\R\Bigr)\Bigr),\label{eq_conv_rad2}\\
u(Y(Y(a_1,\uz_0)a_2,\uz_2)a_3) \in z_2^{h_0-h_1-h_2-h_3}\z_2^{\h_0-\h_1-\h_2-\h_3}\C\Bigl(\Bigl(\frac{z_0}{z_2},\frac{\z_0}{\z_2},\Bigl|\frac{z_0}{z_2}\Bigr|^\R\Bigr)\Bigr),
\label{eq_conv_rad}
\end{align}
where the left-hand side of \eqref{eq_conv_rad} is a formal series in $z_0,\z_0,|z_0|^\R, z_2,\z_2,|z_2|^\R$
but contains only ratios $\left(\frac{z_0}{z_2}\right)^n, \left(\frac{\z_0}{\z_2}\right)^m, \left|\frac{z_0}{z_2}\right|^s$ with $n,m\in \Z$ and $s\in\R$ up to the factor in front.
% 我々はこれらの series が絶対収束する
% これらの series は

\begin{dfn}\label{def_full_VA}
A {\it full vertex algebra} is an $\R^2$-graded $\C$-vector space
$F=\bigoplus_{h,\h \in \R^2} F_{h,\h}$ equipped with a
full vertex operator $Y(\bullet,\uz):F \rightarrow \mathrm{End}(F)[[z,\z,|z|^\R]]$
and an element $\va \in F_{0,0}$ satisfying the following conditions:
\begin{enumerate}
%For any $a \in F$ and $u \in F^\vee$, there exists $N \in \R$ such that
%$u(a(-r,-s)-)=0$ for any $r \geq N$ or $s \geq N$.
\item[FV1)]
For any $a \in F$, $Y(a,\uz)\vac \in F[[z,\z]]$ and $\displaystyle{\lim_{\uz \to 0}Y(a,\uz)\vac = a(-1,-1)\vac=a}$.
\item[FV2)]
$Y(\vac,\uz)=\mathrm{id}_F \in \End F$;
\item[FV3)]
%convergence
For any $a_i \in F_{h_i,\h_i}$ and $u \in F_{h_0,\h_0}^*$, \eqref{eq_conv_rad2} and \eqref{eq_conv_rad} are absolutely convergent in $\{|z_1|>|z_2|\}$ and $\{|z_0|<|z_2|\}$, respectively, and there exists a real analytic function $\mu: Y_2(\C)\rightarrow \C$ such that:
\begin{align*}
u(Y(a_1,\uz_1)Y(a_2,\uz_2)a_3) &= \mu(z_1,z_2)|_{|z_1|>|z_2|}, \\
u(Y(Y(a_1,\uz_0)a_2,\uz_2)a_3) &= \mu(z_0+z_2,z_2)|_{|z_2|>|z_0|},\\
u(Y(a_2,\uz_2)Y(a_1,\uz_1)a_3)&=\mu(z_1,z_2)|_{|z_2|>|z_1|}
\end{align*}
where $Y_2(\C)=\{(z_1,z_2)\in \C^2\mid z_1\neq z_2,z_1\neq 0,z_2\neq 0\}$.
%\item[FV6)]
%$F_{h,\h}(r,s)F_{h',\h'} \subset F_{h+h'-r-1,\h+\h'-s-1}$ for any $r,s,h,h',\h,\h' \in \R$.
\end{enumerate}
\end{dfn}

\begin{rem}
\label{rem_stuchannel}
Left-hand-sides of (FV3) coincide under conformal transformations with what is called the {\it s,t,u-channels} in physics \cite{M2}. Thus, a full vertex algebra is a formulation of conformal field theory by the bootstrap equation.
\end{rem}

Let $F$ be a full vertex algebra
and $D$ and $\dD$ denote the endomorphism of $F$
defined by $Da=a(-2,-1)\bm{1}$ and $\dD a=a(-1,-2)\va$ for $a\in F$,
i.e., $$Y(a,z)\va=a+Daz+\dD a\z+\dots.$$
Then, similarly to the vertex algebra, we have (see \cite[Proposition 3.7, Lemma 3.11, Lemma 3.13]{M1}):
\begin{prop}\label{prop_skew}
Let $F$ be a full vertex algebra. Then, the following properties hold:
\begin{description}
\item[translation invariance)] For any $a \in F$,
\begin{align*}
[D,Y(a,\uz)] &= Y(Da,\uz)=\dz Y(a,\uz),\\
[\dD,Y(a,\uz)] &= Y(\dD a,\uz)=\frac{d}{d\z} Y(a,\uz).
\end{align*}
\item[skew-symmetry)] For any $a,b\in F$,
\begin{align*}
Y(a,\uz)b=\exp(zD+\z\dD)Y(b,-\uz)a.
\end{align*} 
\end{description}
Moreover, if $\dD a=0$, then for any $n\in \Z$ and $b\in F$,
\begin{align*}
[a(n,-1),Y(b,\uz)]&= \sum_{j \geq 0} \binom{n}{j} Y(a(j,-1)b,\uz)z^{n-j},\\
Y(a(n,-1)b,\uz)&= 
\sum_{j \geq 0} \binom{n}{j}(-1)^j a(n-j,-1)z^{j}Y(b,\uz) \\
&\qquad -Y(b,\uz)\sum_{j \geq 0} \binom{n}{j}(-1)^{j+n} a(j,-1)z^{n-j}.
\end{align*}
Furthermore, if $\dD a =0$ and $D b=0$, then $[Y(a,\uz),Y(b,\uz)]=0$.
\end{prop}

\begin{rem}
\label{rem_skew}
By the grading condition \(F_{h,\bar h}=0\) unless \(h-\bar h\in\mathbb Z\), $a(r,s)=0$ if $r-s \notin \Z$.
Thus, $Y(a,\uz)$ consists of $z^n\z^m|z|^r$ with $n,m\in \Z$ and $r\in \R$.
Hence, $Y(a,\uz)$ does not have the monodromy around $z=0$.
In particular, $Y(a,-\uz)$ is well-defined.
\end{rem}

An {\it energy-momentum tensor}
of a full vertex algebra is a pair of vectors
$\om \in F_{2,0}$ and $\omb\in F_{0,2}$ such that
\begin{enumerate}
\item
$\dD \om=0$ and $D \omb=0$;
\item
There exist scalars $c, \bar{c} \in \C$ such that
$\om(3,-1)\om=\frac{c}{2} \va$,
$\omb(-1,3)\omb=\frac{\bar{c}}{2} \va$ and
$\om(k,-1)\om=\omb(-1,k)\omb=0$
for any $k=2$ or $k\in \Z_{\geq 4}$.
\item
$\om(0,-1)=D$ and $\omb(-1,0)=\dD$;
\item
$\om(1,-1)|_{F_{h,\h}}=h$ and
$\omb(-1,1)|_{F_{h,\h}}=\h$ for any $h,\h\in \R$.
\item
There is $N \in \R$ such that $F_{h,\h}=0$ for any $h<N$ or $\h<N$;
\item
For any $H>0$,
$\sum_{h+\h<H} \dim F_{h,\h}<\infty$.
\end{enumerate}
Set 
\begin{align*}
L(n)=\om(n,-1) \quad \text{ and }\quad \Ld(n) = \omb(-1,n).
\end{align*}
We remark that $\{L(n)\}_{n\in \Z}$ and $\{ \Ld(n)\}_{n\in \Z}$
satisfy the commutation relation of the Virasoro algebra
and are mutually commute by Proposition \ref{prop_skew}.
A {\it full vertex operator algebra} is a pair of a full vertex algebra and its energy momentum tensor.

%and $a \in \ker \D$.
%Then, by Lemma \ref{hol_commute},
%$\omb(1)a =0$. Thus, $\ker \D \subset \bigoplus_{n \in \Z} F_{n,0}$.
%Since $\om \in \ker \D$, we have:
\begin{prop}\cite[Proposition 3.18]{M1}
Let $(F,\om,\omb)$ be a full vertex operator algebra.
Then, $\ker \Ld(-1)$ and $\ker L(-1)$ are subalgebra of $F$ and 
\begin{align*}
\ker \Ld(-1) \otimes \ker L(-1) \rightarrow F,\quad a\otimes b \mapsto a(-1,-1)b
\end{align*}
is a full vertex algebra homomorphism.
Moreover, $(\ker \Ld(-1),\om)$ and $(\ker L(-1),\omb)$ are vertex operator algebras.
\end{prop}

%\subsection{Consistency of bulk 2d CFT}
Let $V$ be a positive graded vertex operator algebra.
\begin{dfn}
\label{dfn_LC_bulk}
We call a full vertex operator algebra $F$ {\it locally $C_1$-cofinite} over $V$ if 
there are $M_i, \overline{M}_i \in \Vmodf$ indexed by some countable set $I_c$ such that:
\begin{enumerate}
\item[LC1)]
$V$ is a subalgebra of $\ker L(-1)$ and $\ker \Ld(-1)$;
%$\ker \D$ and $\ker D$ are of CFT type.
\item[LC2)]
$F$ is isomorphic to $\bigoplus_{i \in I_c} M_i\otimes \M_i$ as a $V\otimes V$-module;
\item[LC3)]
For any $i,j \in I_c$, there exists finite subset $I_c(i,j) \subset I_c$ such that:
\begin{align*}
Y(\bullet,\uz)\bullet \in \bigoplus_{i,j \in I_c} \bigoplus_{k \in I_c(i,j)} I\binom{M_k}{M_iM_j} \otimes I\binom{{\M}_k}{{\M}_i{\M}_j},
\end{align*}
where $I\binom{M_k}{M_iM_j}$ and $I\binom{{\M}_k}{{\M}_i{\M}_j}$ are the space of intertwining operators of $V$.
\end{enumerate}
\end{dfn}

Note that (LC2) implies that
\begin{align*}
Y(a,\uz)b \in \bigoplus_{k\in I_c(i,j)}M_k\otimes \M_k((z,\z,|z|^\R))
\end{align*}
for any $a\in M_i\otimes \overline{M}_i$ and $b \in M_j\otimes \overline{M}_j$.

Assume that $(F,Y)$ is a locally $C_1$-cofinite full vertex operator algebra.
For each $r \geq 2$ and $A \in \Tr_r^c$, we can define a parenthesized composition of the full vertex operator 
$Y_{A}(\bullet,\uz)$, similarly to Section \ref{sec_VOA}.
For any $a_{[r]}\in F^{\otimes r}$ and $u\in F^\vee$,
\begin{align}
\langle u,\exp(L(-1)z_{A}+\Ld(-1)\z_{A}) Y_A(\ar,\uzr)\rangle,\label{eq_cor_full}
\end{align}
 is a formal power series.
By local $C_1$-cofiniteness and Theorem \ref{thm_glue},  \eqref{eq_cor_full} is absolutely convergent to a holomorphic function on $U_A \times U_A \subset X_r(\C)$.
By (L2), $z$ and $\z$ in \eqref{eq_cor_full} are of the form $z^n\z^m (z\z)^r$ with $n,m\in \Z$ and $r\in \R$.
Hence, the restriction of \eqref{eq_cor_full} on \eqref{eq_UAc},
\begin{align*}
U_A^c \hookrightarrow \overline{U}_A \times \overline{U}_A,\quad z_{[r]} \mapsto (z_{[r]},\z_{[r]}),
\end{align*}
 is a single-valued real analytic function.

%we call it a {\it parenthesized correlation function}.

\begin{dfn}\label{def_pre_bulk}
A full vertex operator algebra $F$ is said to be {\it consistent} if the following conditions hold:
\begin{description}
\item[Convergence]
For any $u\in F^\vee$ and $\ar \in \Fr$,
%\begin{align*}
%\langle u,\exp(L(-1)z_{r_A}+\Ld(-1)\z_{r_A}) Y_A(\ar,\uzr)\rangle \in S_A^\conv,
%\end{align*}
%i.e., 
$\langle u,\exp(L(-1)z_{A}+\Ld(-1)\z_{A}) Y_A(\ar,\uzr)\rangle$ is absolutely locally uniformly convergent to a 
holomorphic function on $U_A \times U_A \subset X_r(\C) \times X_r(\C)$.
Denote the restriction of this real analytic function
on $U_A^c$ by $C_A(u,\ar;\uzr)$,
which is a real analytic function on $X_r(\C)$.
\item[Compatibility]
There exists a family of linear maps
\begin{align*}
C_r:F^\vee\otimes \Fr \rightarrow C^\om(\Xr),\quad\text{ for }r\geq 2
\end{align*}
where $C^\om(\Xr)$ is the vector space of real analytic functions on $X_r(\C)$, such that:
\begin{align}
C_r(u,\ar;\zr)\Bigl|_{U_A^c}=C_A(u,\ar;\zr)
\label{eq_cor_def_bulk}
\end{align}
for any $\ar$ as real analytic functions on $U_A^c$.
\end{description}
\end{dfn}
\begin{rem}
\label{rem_ext_cor}
It is natural to extend the definition of $C_r:F^\vee\otimes \Fr \rightarrow C^\om(X_r(\C))$ to $r=0,1$ by
\begin{align*}
C_0:F^\vee \rightarrow \C,\quad u \mapsto \langle u,\va \rangle
\end{align*}
and
\begin{align*}
C_1:F^\vee \otimes F \rightarrow C^\om(X_1(\C)), \quad u \mapsto \langle u, \exp(L(-1)z+\Ld(-1)\z)a\rangle,
\end{align*}
where we think $z$ as the standard coordinate of $\C = X_1(\C)$.
\end{rem}
The functions \(\{C_r\}_{r\geq 0}\) are called \emph{correlation functions} in physics and
are among the basic physical quantities in quantum field theory.

\begin{thm}
\label{thm_bulk}
A locally $C_1$-cofinite full vertex operator algebra is consistent.
Moreover, the sequence of linear maps $\{C_r:F^\vee\otimes F^{\otimes r}\rightarrow C^\omega(\Xr)\}_{r=0,1,\dots}$ in \eqref{eq_cor_def_bulk} and Remark \ref{rem_ext_cor} satisfy the following conditions:
\begin{description}
\item[(Symmetry)]
For any $\si \in S_r$, $u \in F^\vee$ and $a_1,\dots,a_r \in F$,
\begin{align*}
C_r(u,a_1,\dots,a_r;z_1,\dots,z_r)=C_r(u,a_{\si(1)},\dots,a_{\si(r)};z_{\si(1)},\dots,z_{\si(r)}),
\end{align*}
where $S_r$ is the symmetric group.
\item[(Conformal covariance)]
\begin{align*}
C_r(u,L(-1)_i a_{[r]};z_{[r]}) &= \frac{d}{dz_i} C_r(u, a_{[r]};z_{[r]})\\
C_r(u,\Ld(-1)_i a_{[r]};z_{[r]}) &= \frac{d}{d\z_i} C_r(u, a_{[r]};z_{[r]})\\
C_r(L(-1)^* u,a_{[r]};z_{[r]}) &= \sum_{i=1}^r C_r(u, L(-1)_i a_{[r]};z_{[r]})\\
C_r(\Ld(-1)^* u,a_{[r]};z_{[r]}) &= \sum_{i=1}^r C_r(u, \Ld(-1)_i a_{[r]};z_{[r]})\\
C_r(L(0)^* u,a_{[r]};z_{[r]}) &= \sum_{i=1}^r C_r(u, (z_i\frac{d}{dz_i}+L(0)_i) a_{[r]};z_{[r]})\\
C_r(\Ld(0)^* u,a_{[r]};z_{[r]}) &= \sum_{i=1}^r C_r(u, (\z_i\frac{d}{d\z_i}+\Ld(0)_i) a_{[r]};z_{[r]})\\
C_r(L(1)^* u,a_{[r]};z_{[r]}) &= \sum_{i=1}^r C_r(u, (z_i^2\frac{d}{dz_i}+2z_iL(0)_i+L(1)_i) a_{[r]};z_{[r]})\\
C_r(\Ld(1)^* u,a_{[r]};z_{[r]}) &= \sum_{i=1}^r C_r(u, (\z_i^2\frac{d}{d\z_i}+2\z_i\Ld(0)_i+\Ld(1)_i) a_{[r]};z_{[r]})
\end{align*}
%\item[(Strong cluster property)]
%The sum
%\begin{align*}
%\sum_{i \in I^{h,\h}} C_{n+1}(u,a_{[n]},e_i; z_{[n]},y)C_{m}(\exp(-L(-1)y-\Ld(-1)\bar{y}) e_i^*,a_{[m]};z_{[m]}).
%\end{align*}
%is independent of the position $y$, and absolutely convergent to $C_{n+m}(u,a_{[n+m]};z_{[n+m]})$
%in $U_{\min}$ (the meaning of the convergence is explained below).
%The sum
%\begin{align*}
%\sum_{h,\h \in \R^2} |C_{n,m}^{h,\h}(u,a_{[n+m]};z_{[n+m]})|
%\end{align*}
%is uniformly convergent in $U_{\min}$ (explained below), and 
%\begin{align*}
%C_{n+m}(u,a_{[n+m]};z_{[n+m]})=\sum_{h,\h \in \R^2}C_{n,m}^{h,\h}(u,a_{[n+m]};z_{[n+m]}).
%\end{align*}
\item[(Vacuum property)]
For any $u \in F^\vee$ and $a_1,\dots,a_{r-1} \in F$,
\begin{align*}
C_{r}(u,a_1,\dots,a_{r-1},\va;z_1,\dots,z_r)=C_{r-1}(u,a_1,\dots,a_{r-1};z_1,\dots,z_{r-1}).
\end{align*}
\end{description}
\end{thm}
\begin{proof}
We will show this in the case that the index set $I$ in the definition of local $C_1$-cofiniteness is finite.
The general case can be shown in a similar way by (LC3).
Then, $F \in D= \Vmodf \boxtimes \Vmodf$.
By (LC3) and Proposition \ref{prop_int}, we can identify the full vertex operator as a full conformal block:
\begin{align*}
s_{12}(\langle \bullet, \exp( L(-1)z_2+ \Ld(-1)\z_2)Y(\bullet,\uz)\rangle) \in \CB_{F,F,F}^c(12).
\end{align*}

By Theorem \ref{thm_glue}, for any $A\in \Tr_r^c$, the composite full vertex operator defines a section
\begin{align*}
s_A(\langle \bullet,\exp(L(-1)z_{A}+\Ld(-1)\z_{A})Y_A(\bullet,\uzr)\rangle \in \CB_{F,F^{\otimes r}}^c(U_A^c).
\end{align*}
In particular, for any $u\in F^\vee$, $\langle u,\exp(L(-1)z_{A}+\Ld(-1)\z_{A}) Y_A(\ar,\uzr)\rangle$ absolutely convergent on $U_A \times U_A$, and thus $U_A^c$, and has analytic continuation to the possibly multi-valued real analytic function on $\Xr$.
To prove the single-valuedness, it suffices to show that for any path $\ga:A \to B$ in $\CPaB(r)$,
\begin{align}
A(\ga)\left(C_A(u,\ar,\uzr) \right) = C_B(u,\ar,\uzr).\label{eq_consistent_closed}
\end{align}
In the case of $r=2$, by the definition of a full vertex algebra, $\langle u, \exp(L(-1)z_2+\Ld(-1)\z_2)Y(a_1,\uz_{12})a_2\rangle$ is in $\C[z_1,z_2,\z_1,\z_2,|z_1-z_2|^\R]$.
In particular, it is a single-valued real analytic function.
Hence, it suffices to show that \eqref{eq_consistent_closed} holds for the path $\si:(12) \rightarrow (21)$ in $\CPaB(2)$ (see Fig \ref{fig_sigma}).

By Proposition \ref{prop_skew}, we have:
\begin{align*}
\langle u, Y(a_1,\uz_{12})a_2 \rangle= \langle u,\exp(L(-1)z_{12}+\Ld(-1)\z_{12}) Y(a_2, -\uz_{12})a_1 \rangle.
\end{align*}
Since the conformal weights of $F$ is bounded below, $\exp(L(-1)^*z_{12}+\Ld(-1)^*\z_{12}) \bra{u}$ is a finite sum.
By setting $u= \exp(L(-1)^*z_{2}+\Ld(-1)^*\z_{2})u'$, we have:
\begin{align*}
\langle u', \exp(L(-1)z_2+\Ld(-1)\z_2)Y(a_1,\uz_{12})a_2\rangle = \langle u', \exp(L(-1)z_1+\Ld(-1)\z_1)Y(a_2, -\uz_{12})a_1\rangle.
\end{align*}
Hence, we have:
\begin{align}
A(\si)\left(C_{12}(u,a_{[2]},\uz_{[2]}) \right) = C_{21}(u,a_{[2]},\uz_{[2]})
\label{eq_skew_closed}
\end{align}
and thus \eqref{eq_consistent_closed} hold for all the paths in $\CPaB(2)$ (see also Remark \ref{rem_skew}).

In the case of $r=3$, by the definition of a full vertex algebra, all of the following functions have analytic continuation to $X_3(\C)$
and coincide with each others;
\begin{align*}
s_{1(23)}&\left(\langle u, \exp(L(-1)z_3+\Ld(-1)\z_3)Y(a_1,\uz_{13})Y(a_2,\uz_{23})a_3\right)\\
s_{2(13)}&\left(\langle u, \exp(L(-1)z_3+\Ld(-1)\z_3)Y(a_2,\uz_{23})Y(a_1,\uz_{13})a_3\right)\\
s_{(12)3}&\left(\langle u, \exp(L(-1)z_3+\Ld(-1)\z_3)Y(Y(a_1,\uz_{12})a_2,\uz_{23})a_3\right).
\end{align*}
In particular, for $\al: (12)3 \rightarrow 1(23)$ in $\CPaB(3)$ in Fig. \ref{fig_alpha},
\begin{align}
A(\al)\left(C_{(12)3}(u,a_{[3]};z_{[3]})\right)=C_{1(23)}(u,a_{[3]};z_{[3]}).
\label{eq_assoc_closed}
\end{align}

Let us consider the general case:
By definition of $C_A$ and the glueing of conformal blocks, 
\begin{align*}
\comp_{p}(C_A, C_B) = C_{A \circ_p B}.
\end{align*}
Since $\CPaB$ is generated by $\al:(12)3 \rightarrow 1(23)$ and $\si:12 \rightarrow 21$ as an operad,
by Theorem \ref{thm_SC_action} and \eqref{eq_skew_closed} and \eqref{eq_assoc_closed},
\eqref{eq_consistent_closed} hold for all paths.

(Symmetry) is obvious by the construction of $C_r$'s.
For $i=1,\dots,r-1$,
\begin{align*}
C_r(u,L(-1)_i a_{[r]};z_{[r]}) &= \frac{d}{dz_i} C_r(u, a_{[r]};z_{[r]})\\
C_r(u,\Ld(-1)_i a_{[r]};z_{[r]}) &= \frac{d}{d\z_i} C_r(u, a_{[r]};z_{[r]})
\end{align*}
follows from Proposition \ref{prop_skew}.
Since
\begin{align*}
&\frac{d}{dz_r} C_r(u, a_{[r]};z_{[r]})\\
 &= \frac{d}{dz_r}
\langle u, \exp(L(-1)z_r+\Ld(-1)\z_r)Y(a_1,\uz_{1r})Y(a_2,\uz_{2r})\dots Y(a_{r-1}z_{r-1,r}) a_r  \rangle\\
&=\langle u, \exp(L(-1)z_r+\Ld(-1)\z_r)L(-1)Y(a_1,\uz_{1r})Y(a_2,\uz_{2r})\dots Y(a_{r-1}z_{r-1,r}) a_r  \rangle\\
&-\sum_{i=1}^{r-1}\frac{d}{dz_{i}} \langle u, \exp(L(-1)z_r+\Ld(-1)\z_r)Y(a_1,\uz_{1r})Y(a_2,\uz_{2r})\dots Y(a_{r-1}z_{r-1,r}) a_r  \rangle\\
&=\langle u, \exp(L(-1)z_r+\Ld(-1)\z_r)L(-1)Y(a_1,\uz_{1r})Y(a_2,\uz_{2r})\dots Y(a_{r-1}z_{r-1,r}) a_r  \rangle\\
&-\sum_{i=1}^{r-1}\langle u, \exp(L(-1)z_r+\Ld(-1)\z_r)Y(a_1,\uz_{1r})\dots [L(-1),Y(a_i,\uz_{ir})] \dots Y(a_{r-1}z_{r-1,r}) a_r  \rangle\\
&= \langle u, \exp(L(-1)z_r+\Ld(-1)\z_r)Y(a_1,\uz_{1r})Y(a_2,\uz_{2r})\dots Y(a_{r-1}z_{r-1,r}) L(-1)a_r  \rangle\\
&=C_r(u, L(-1)_r a_{[r]};z_{[r]}).
\end{align*}

Similarly, we have
\begin{align*}
C_r&(L(-1)^* u,a_{[r]};z_{[r]})\\
&=\langle u, L(-1)\exp(L(-1)z_r+\Ld(-1)\z_r)Y(a_1,\uz_{1r})Y(a_2,\uz_{2r})\dots Y(a_{r-1}z_{r-1,r}) a_r  \rangle\\
&=\sum_{k=1}^{r-1}\langle u, \exp(L(-1)z_r+\Ld(-1)\z_r)Y(a_1,\uz_{1r})\dots[L(-1), Y(a_k,\uz_{kr})]\dots Y(a_{r-1}z_{r-1,r}) a_r  \rangle\\
&+ \langle u, \exp(L(-1)z_r+\Ld(-1)\z_r)Y(a_1,\uz_{1r})Y(a_2,\uz_{2r})\dots Y(a_{r-1}z_{r-1,r}) L(-1)a_r  \rangle\\
&= \sum_{i=1}^r C_r(u, L(-1)_i a_{[r]};z_{[r]}).
\end{align*}

Before proving the remaining covariance identities, we prove the vacuum property.
Since $Y(\va,\uz) =\mathrm{id}_F$, we have:
\begin{align*}
&C_{r+1}(u,\va,a_1,\dots,a_{r}; z_0,z_1,\dots,z_r) |_{U_{0(1(\cdots (r-1 r)))}^c}\\
&=\langle u, \exp(L(-1)z_r+\Ld(-1)\z_r) Y(\va,\uz_{0r}) Y(a_1,\uz_{1r}) \dots Y(a_{r-1},\uz_{r-1r}) a_r\rangle\\
&=\langle u, \exp(L(-1)z_r+\Ld(-1)\z_r)Y(a_1,\uz_{1r}) \dots Y(a_{r-1},\uz_{r-1r}) a_r\rangle\\
&=C_{r}(u,a_1,\dots,a_{r};z_1,\dots,z_r)|_{U_{1(\cdots (r-1 r)))}^c}
\end{align*}
%and similarly for $\Ld(-1)$. Hence, $C_{r}(u,a_1,\dots,a_{r-1},\va;z_1,\dots,z_r)$ is independent of $z_r$.
Hence, (Vacuum) holds.
We remark that, by (Vacuum),
\begin{align*}
C_{r+1}(a_1,\dots,a_r,\va;z_{[r+1]})|_{U_{1(2(\cdots (rr+1)))}}= \langle u, \exp(L(-1)z_{r+1}+\Ld(-1)z_{r+1}) Y(a_1,\uz_1)Y(a_2,\uz_2) \dots Y(a_r,\uz_{r})\va\rangle,
\end{align*}
is independent of $z_{r+1}$. Therefore, we can set $z_{r+1}=0$ and obtain
\begin{align}
\langle u, Y(a_1,\uz_1)Y(a_2,\uz_2) \dots Y(a_r,\uz_{r})\va\rangle
= C_{r}(a_1,\dots,a_r;z_{[r]})|_{|z_1|>\cdots >|z_r|} \label{eq_conv_region}
\end{align}
by Proposition \ref{rem_tree_region}.
\eqref{eq_conv_region} is more symmetric with respect to the indexes from $1$ to $r$ and is often computationally convenient.
By using \eqref{eq_conv_region}, we can similarly obtain the conformal covariance by
\begin{align*}
[L(1),Y(a,\uz)]&=Y((z^2L(-1)+2L(0)z+L(1) )a,\uz)\\
[\Ld(1),Y(a,\uz)]&=Y((\z^2\Ld(-1)+2\Ld(0)\z+\Ld(1))a,\uz)
\end{align*}
and 
\begin{align*}
[L(0),Y(a,\uz)]&=Y((zL(-1)+L(0) )a,\uz)\\
[L(0),Y(a,\uz)]&=Y((\z\Ld(-1)+\Ld(0))a,\uz)
\end{align*}
and $L(0)\va=\Ld(0)\va=L(1)\va=\Ld(1)\va =0$
(see \cite[Lemma 3.11]{M1}).
\end{proof}

\begin{rem}
Part of this theorem is first obtained by Huang and Kong \cite{HK1} based on the representation theory of a rational $C_2$-cofinite vertex operator algebra developed by Huang and Lepowsky (see \cite{HL}). They showed the above theorem for special trees in $\Tr^c$ when $V$ is a rational $C_2$-cofinite VOA.
\end{rem}

The following corollary follows from the argument in the above proof \eqref{eq_conv_region}:
\begin{cor}\label{cor_conv_region}
under the assumption of Theorem \ref{thm_bulk}, for any $u\in F^\vee$ and $a_i \in F$,
\begin{align*}
\langle u, Y(a_1,\uz_1)Y(a_2,\uz_2) \dots Y(a_r,\uz_{r})\va\rangle
\end{align*}
is absolutely convergent in $|z_1|>|z_2|>\cdots >|z_r|$ to $C_{r}(a_1,\dots,a_r;z_{[r]})$.
\end{cor}

\begin{rem}
\label{rem_conv_exp}
Since $Y(a,\uz)\va=\exp(L(-1)z+\Ld(-1)\z)a$,
\begin{align*}
\langle u, &Y(a_1,\uz_1)Y(a_2,\uz_2) \dots Y(a_{r-1},\uz_{r-1})Y(a_r,\uz_r)\va \rangle \\
&=\langle u, Y(a_1,\uz_1)Y(a_2,\uz_2) \dots Y(a_{r-1},\uz_{r-1}) \exp(L(-1)z_r+\Ld(-1)\z_r)a_r\rangle\\
&=\langle u, \exp(L(-1)z_r+\Ld(-1)\z_r) Y(a_1,\uz_1-\uz_r)Y(a_2,\uz_2-\uz_r) \dots Y(a_{r-1},\uz_{r-1}-\uz_{r}) a_r\rangle,
\end{align*}
which is equal to $\langle u, \exp(L(-1)z_r+\Ld(-1)\z_r) Y(a_1,\uz_{1r})Y(a_2,\uz_{2r}) \dots Y(a_{r-1},\uz_{r-1,r}) a_r\rangle$ after the change of variables.
However, formal calculus alone cannot yield the convergence region in Corollary \ref{cor_conv_region}.
\end{rem}

\begin{rem}
\label{rem_OS}
The Osterwalder-Schrader axioms (OS axioms) are axioms that characterize quantum field theory using correlation functions \cite{OS1,OS2}.
%full vertex operator algebra $F$が
%\begin{itemize}
%\item
%$F_{h,\h} = 0$ if $h<0$ or $\h <0$;
%\item
%$F_{0,0}=\C\va$;
%\item
%$L(1)F_{1,0}=0$ and $L(1)F_{0,1}=0$
%\end{itemize}
%を満たすとき、一意的な$\bra{\va} \in F_{0,0}^*$ で
%\begin{align}
%L(k)^* \bra{\va}=\Ld(k)^*\bra{\va}=0\quad\quad \text{ for }k=-1,0,1
%\label{eq_vac_inv}
%\end{align}
%and $\bra{\va}\ket{\va}=1$ を満たすものが存在する。
%Set
%\begin{align*}
%C_r^{a_{[r]}}(\zr) = C_r(\bra{\va},\ar,\zr).
%\end{align*}
%Then, $C_r^{a_{[r]}}(\zr)$ is called {\it an $r$-point correlation function} in physics.
%By \eqref{eq_vac_inv}, correlation functions satisfies:
%\begin{align*}
%(\sum_{i=1}^r \frac{d}{dz_i}) C_r^{a_{[r]}}(z_{[r]})=0 \quad&\text{ and }\quad (\sum_{i=1}^r \frac{d}{d\z_i}) C_r^{a_{[r]}}(z_{[r]})=0\\
% (\sum_{i=1}^r z_i\frac{d}{dz_i}+h_i) C_r^{a_{[r]}}(z_{[r]})=0
%  \quad&\text{ and }\quad 
% (\sum_{i=1}^r \z_i\frac{d}{d\z_i}+\h_i)C_r^{a_{[r]}}(z_{[r]})=0\\
% (\sum_{i=1}^r z_i^2\frac{d}{dz_i}+h_iz_i) C_r^{a_{[r]}}(z_{[r]})=0  \quad&\text{ and }\quad 
% (\sum_{i=1}^r \z_i^2\frac{d}{d\z_i}+\h_i\z_i)C_r^{a_{[r]}}(z_{[r]})=0
%\end{align*}
%for any quasi-primary vectors $\ar$, i.e., $L(1)a_i=\Ld(1)a_i=0$ for $i=1,\dots,r$.
Since $z\frac{d}{dz}-\z\frac{d}{d\z}$ generates the rotation of the complex plane $\C$,
the first and the second equations in (Conformal covariance) imply the Poincare invariance $\R^2 \rtimes \mathrm{SO}(2)$. The orthogonal group $\mathrm{SO}(2)$ appears here since $h_i-\h_i \in \Z$.
If we consider a full vertex operator superalgebra, i.e., $h_i-\h_i \in \frac{1}{2}\Z$, then $\mathrm{SO}(2)$ should be changed by $\mathrm{Spin}(2)$.
These invariance is nothing but the part of the OS axioms.
The commutativity of a full vertex operator algebra corresponds to (Symmetry) which is the locality condition of the  OS axioms.
%Those are the important part of the OS axioms.
It is noteworthy that we do not assume any unitarity condition, while the OS axioms does,
the Reflection positivity.
Thus, to prove that $\{C_r\}_{r\geq 0}$ satisfy the OS axioms,
it is obvious that we need to assume unitarity on $F$.

Another point to note is that in the OS axioms $\{C_r\}_{r\geq 0}$ are tempered distributions (with the linear growth condition), not real analytic functions.
Once these three, (Temperedness), (Reflection positivity) and (Linear growth) are shown, the Wightman field (quantum field on the Minkowski space $\R^{1,1}$) can be obtained.
Those are discussed in a joint work with M.S. Adamo and Y. Tanimoto in \cite{AMT}.
\end{rem}

We note that Theorem \ref{thm_bulk} is easily extended to the case where the VOAs appearing in the chiral and anti-chiral parts do not coincide. 
We will state the precise statement here. The proof is completely parallel.
\begin{thm}
\label{thm_bulk2}
Let $F$ be a full vertex operator algebra and $V,W$ positive graded vertex operator algebras.
Assume that there are $C_1$-cofinite $V$-modules $M_i$
and $C_1$-cofinite $W$-modules $\overline{M}_i$ indexed by some countable set $I$ such that:
\begin{enumerate}
\item
$V$ is a subalgebra of $\ker L(-1)$ and $W$ is a subalgebra of $\ker \Ld(-1)$;
%$\ker \D$ and $\ker D$ are of CFT type.
\item
$F$ is isomorphic to $\bigoplus_{i \in I} M_i\otimes \M_i$ as a $V\otimes W$-module;
\item
For any $i,j \in I$, there exists finite subset $I(i,j) \subset I$ such that:
\begin{align*}
Y(\bullet,\uz)\bullet \in \bigoplus_{i,j \in I} \bigoplus_{k \in I(i,j)} I\binom{M_k}{M_iM_j} \otimes I\binom{{\M}_k}{{\M}_i{\M}_j},
\end{align*}
where $I\binom{M_k}{M_iM_j}$ and $I\binom{{\M}_k}{{\M}_i{\M}_j}$ are the space of intertwining operators of $V$ and $W$, respectively.
\end{enumerate}
Then, $F$ is consistent and $C_r$'s satisfy the properties in Theorem \ref{thm_bulk}.
\end{thm}

\subsection{Boundary OPE algebra}
%\subsection{Full vertex operator prealgebra modeled on $\HH$}
\label{sec_boundary}
The algebra appearing on the boundary of 2d CFT is mathematically formulated in \cite{HK2}, which is called an {\it open-string vertex algebra}.
The necessary and sufficient conditions for constructing an open-string vertex algebra are described in \cite [Proposition 2.7]{HK2} under assumptions similar to (but slightly stronger) the local $C_1$-cofiniteness in the previous section.

In this section, we will introduce the boundary OPE algebra based on the bootstrap equation (see also \cite[Proposition 2.7]{HK2}). Then, we state and prove the consistency of the OPEs with respect to all trees $\Tr^o(0,s)$.

Let $V$ be a positively graded vertex operator algebra with the conformal vector $\om^o \in V_2$.
Let $F^o=\bigoplus_{h\in \R}F_{h}^o$ be a $V$-module such that 
$L^o(0)|_{F_{h}^o}=h \id_{F_{h}^o}$
for any $h\in \R$,
where 
\begin{align*}
Y(\om^o,z) = \sum_{n\in \Z}L^o(n)z^{-n-1}.
\end{align*}
We assume that:
\begin{itemize}
\item
For any $H\in \R$,
$\bigoplus_{h \leq H}F_{h}^o$ is finite-dimensional.
\end{itemize}
Set 
\begin{align*}
(F^o)^\vee =\bigoplus_{h \in\R} (F_{h}^o)^*,
\end{align*}
where $(F_{h}^o)^*$ is the dual vector space.

For a vector space $W$,
we denote by 
$W((z^\R))$ the subspace of $W[[z^\R]]$ spanned by the series $\sum_{r\in \R}a_r z^r$ satisfying
\begin{itemize}
\item
For any $H\in \R$,
$$\{r \in \R \;|\; a_{r}\neq 0 \text{ and }r \leq H \}$$
is a finite set.
\end{itemize}

A {\it boundary vertex operator} on $F^o$ with the chiral symmetry $V$ is a linear map
\begin{align*}
Y^o(\bullet, z):F^o \rightarrow \mathrm{End}(F^o)[[z^\R]],\; a\mapsto Y^o(a,z)=\sum_{r \in \R}a(r)z^{-r-1}
\end{align*}
such that:
\begin{align}
\begin{split}
[L^o(-1),Y^o(v,z)]&=\frac{d}{dz}Y^o(v,z)\\
[a(n),Y^o(v,z)]&= \sum_{k\geq 0}\binom{n}{k} Y^o( a(k)v,z)z^{n-k}\\
Y^o(a(n)v,z) &= \sum_{k\geq 0} \binom{n}{k}\left( a(n-k)Y^o(v,z)(-z)^{k} - Y^o(v,z)a(k)(-z)^{n-k}\right)
\label{eq_L0_open}
\end{split}
\end{align}
for any $a\in V$ and $v \in F^o$, i.e., $Y^o(\bullet,z)$ is an intertwining operator of type $\binom{V}{VV}$.
Applying $a =\om^o$ and $n=1$, we have:
\begin{align*}
[L^o(0),Y^o(v,z)]&= z\frac{d}{dz}Y^o(v,z)+ Y^o(L^o(0)a,z).
\end{align*}
Hence, similarly to the bulk case, 
$Y(a,z)b \in F^o((z^\R))$ and for $u \in (F_{h}^o)^*$ and $a_i \in F_{h_i}^o$ we have
\begin{align}
u(Y^o(a_1,z_1)Y^o(a_2,z_2)a_3) &\in z_2^{h_0-h_1-h_2-h_3} \C\Bigl(\Bigl(\left(\frac{z_2}{z_1}\right)^\R\Bigr)\Bigr),\label{eq_conv_rad2_open}\\
u(Y^o(Y^o(a_1,z_0)a_2,z_2)a_3) &\in z_2^{h_0-h_1-h_2-h_3}\C\Bigl(\Bigl(\left(\frac{z_0}{z_2}\right)^\R\Bigr)\Bigr).
\label{eq_conv_rad_open}
\end{align}

We consider the following assumptions (see \cite{HK2}):
\begin{dfn}\label{def_bfullVA}
Assume that the boundary vertex operator $(F^o,Y^o)$ with the chiral symmetry $V$ together with a distinguished vector $\va^o \in F_0^o$ satisfying the following conditions:
\begin{enumerate}
\item[bFV1)]
For any $a \in F^o$, $Y^o(a,z)\va^o \in F^o[[z]]$ and $\lim_{z \to 0}Y^o(a,z)\va = a(-1)\va=a$, $Y^o(\va,z)=\mathrm{id}_{F^o} \in \End F^o$;
\item[bFV2)]
%convergence
For any $a_1,a_2,a_3 \in F^o$ and $u \in (F^o)^\vee$, \eqref{eq_conv_rad2_open} and \eqref{eq_conv_rad_open}
are absolutely convergent in $\{|z_1|>|z_2|\}$ and $\{|z_0|<|z_2|\}$, respectively, and 
there exists a real analytic function $\mu(z_1,z_2)$ on $\{(z_1,z_2)\in \R^2 \mid z_1> z_2>0\}$ such that
\begin{align}
u(Y^o(a_1,z_1)Y^o(a_2,z_2)a_3) &= \mu(z_1,z_2)|_{z_1>z_2>0},\nonumber \\
u(Y^o(Y^o(a_1,z_0)a_2,z_2)a_3) &= \mu(z_1,z_2)|_{z_2>z_1-z_2>0}, \label{eq_boundary_borcherds}
\end{align}
where we set $z_0=z_1-z_2$.
%\item[bFV4')]
%there exists a possibly multivalued holomorphic function $\mu(z_1,z_2)$ on $\{(z_1,z_2) \in \C^2 \mid z_i\neq z_j\}$ such that
\end{enumerate}
\end{dfn}

\begin{prop}\label{prop_open_sub}
Let $(F^o,Y^o,\va^o)$ satisfy Definition \ref{def_bfullVA}. Then, for any $a \in V$, 
the vertex operator $Y^o(a(-1)\va^o,z)$ coincides with the vertex operator which gives the $V$-module structure on $F^o$. In particular, $Y^o(a(-1)\va^o,z) = \sum_{n\in \Z}a(n)z^{-n-1}$ and 
\begin{align*}
i_o: V \rightarrow F^o,\quad a\mapsto a(-1)\va^o
\end{align*}
satisfies $i_o(a(n)b)=i_o(a)(n)i_o(b)$ for any $a,b\in V$ and $n\in \Z$.
\end{prop}
\begin{proof}
By \eqref{eq_L0_open} and (bFV1), we have
\begin{align*}
Y^o(a(-1)\va^o,z) &= \sum_{k\geq 0} (-1)^{k}\left( a(-1-k)Y^o(\va^o,z)(-z)^{k} - Y^o(\va^o,z)a(k)(-z)^{-1-k}\right)\\
&=\sum_{n \in \Z} a(n)z^{-n-1}.
\end{align*}
Hence, for any $a,b\in V$ and $n\in \Z$,
$i_o(a)(n)i_o(b)= (a(-1)\va^o)(n) (b(-1)\va^o) = a(n)b(-1)\va^o$.
By the Borcherds identity as $V$-module and (bFV1),
\begin{align*}
(a(n)b)(-1)\va^o &=\sum_{k\geq 0}(-1)^k \binom{n}{k}\left( a(n-k)b(-1+k) \va^o - (-1)^n b(-1-k)a(k)\va^o \right)\\
&=a(n)b(-1)\va^o.
\end{align*}
\end{proof}
By the above proposition, we can identify $V$ as a subalgebra of $F^o$.
The following proposition is analogous to \cite[Proposition 3.1.19]{LL}: 
\begin{prop}
Let $(F^o,Y^o,\va^o)$ satisfy Definition \ref{def_bfullVA}. Then, 
\begin{align*}
Y^o(v,z)a = \exp(L(-1)^o z) Y^o(a,-z)v
\end{align*}
for any $v \in F^o$ and $a\in V \subset F^o$.
\end{prop}
\begin{proof}
We first show the case of $a =\va^o$. 
By (bFV1), $Y(\om^o,z)\va^o \in F^o[[z]]$, which implies $L^o(n-1)\va^o=0$ for any $n \geq 0$.
By \eqref{eq_L0_open},
\begin{align*}
L(-1)^o Y^o(v,z)\va^o = \frac{d}{dz} Y^o(v,z)\va^o.
\end{align*}
Since $\lim_{z\to 0}Y^o(v,z)\va^o =v$, we can solve this differential equation inductively and get
\begin{align*}
Y^o(v,z)\va^o = \exp(L^o(-1)z)v.
\end{align*}
By \eqref{eq_L0_open}, we have:
\begin{align*}
Y^o(v,z)a(-1)\va^o &=-[a(-1), Y^o(v,z)]\va^o + a(-1)Y^o(v,z)\va^o\\
&=-[a(-1), Y^o(v,z)]\va^o + a(-1)\exp(L(-1)^o z)v\\
&=-\sum_{k \geq 0}\binom{-1}{k} Y^o(a(k)v,z)\va^o z^{-1-k} + \exp(L^o(-1)z) \sum_{k\geq 0}a(-k-1)v (-z)^{k}\\
&=\sum_{k \geq 0} \exp(L^o(-1)z)a(k)v (-z)^{-1-k} + \exp(L^o(-1)z) \sum_{k\geq 0}a(-k-1)v (-z)^{k}\\
&= \exp(L^o(-1)z)Y(a,-z)v
\end{align*}
where we use $[L^o(-1),a(n)] = -na(n-1)$ for any $n \in\Z$.
\end{proof}
 
\begin{rem}
\label{Rem_center}
Set
\begin{align*}
Z(F^o) = \{v\in &\bigoplus_{n\in\Z} F_n^o \mid Y^o(v,z) \in \End F^o[[z^\pm]],\\
&Y^o(v',z)v= \exp(L(-1)^o z)Y^o(v,-z)v' \text{ for any }v' \in F^o\},
\end{align*}
which is called a {\it meromorphic center} of $F^o$ in \cite{HK2}.
Without a priori assuming the existence of VOA, Huang and Kong showed from (bFV2) together with some assumptions on $Y^o(\bullet,z):F^o \rightarrow \End F^o[[z^\R]]$ that $Z(F^o)$ is a vertex algebra \cite[Theorem 2.3]{HK2}.
\end{rem}
 
%
%Set
%\begin{align*}
%Z(F^o) = \{v\in &\bigoplus_{n\in\Z} F_n^o \mid Y^o(v,x) \in \End F^o[[x^\pm]],\\
%&Y^o(v',x)v= \exp(L(-1)^o x)Y^o(v,-x)v' \text{ for any }v' \in F^o\},
%\end{align*}
%which is called a {\it meromorphic center} of $F^o$ introduced in \cite{HK2}.
%Our definition is weaker than \cite{HK2}'s, but in the proof of the following proposition they only use the convergence of \eqref{eq_boundary_borcherds} in $x_1>x_2>x_3$. Hence, we have:
%\begin{prop}\cite[Theorem 2.3]{HK2}
%$Z(F^o)$ is a vertex algebra.
%\end{prop}
%\begin{itemize}
%\item
%$Y(a,x)Y(b,x)Y(c,x)\va$の収束性が必要かもしれない。
%\item
%\cite{HK2} のTheorem 2.3 and 
%\end{itemize}

\begin{dfn}
\label{def_LC_boundary}
We call $(F^o,Y^o,\va^o)$ with the chiral symmetry $V$ {\it locally $C_1$-cofinite} if there are $N_i \in \Vmodf$
 indexed by some countable set $I_o$ such that:
\begin{enumerate}
\item[bLC1)]
$F^o$ is isomorphic to $\bigoplus_{i \in I_o} N_i$ as a $V$-module;
\item[bLC2)]
For any $i,j \in I_o$, there exists finite subset $I_o(i,j) \subset I_o$ such that:
\begin{align*}
Y(\bullet,\uz)\bullet \in \bigoplus_{i,j \in I_o} \bigoplus_{k \in I_o(i,j)} I\binom{N_k}{N_iN_j}.
\end{align*}
\end{enumerate}
\end{dfn}

Recall that open leaves of a tree $E \in \Tr^o(0,s)$ are assumed to be ordered and
\begin{align*}
X_{0,s}(\HH) = \{z_1,\dots,z_s \in \R^s\mid z_1>z_2>\cdots >z_s\}.
\end{align*}
We can define the iterated vertex operator $Y_E^o(a_{[0,s]},z_{[0,s]})$ for $a_{[0,s]}\in (F^o)^{\otimes s}$ and each $E\in \Tr^o(0,s)$ exactly as in Section \ref{sec_full_VOA}.
%a formal power series 
%\begin{align*}
%S_E^o(u,v_{[r]};x_{[r]})= 
%u,\exp(L(-1)x_r)Y_E^o(v_{[r]},x_{[r]}).
%\end{align*}

\begin{dfn}\label{def_pre_boundary}
Let $(F^o,Y^o,\va^o)$ be the triple in Definition \ref{def_bfullVA} with a chiral symmetry $V$.
We call it consistent if the following properties are satisfied:
\begin{description}
\item[Convergence]
For any $u\in (F^o)^\vee$ and $a_{[0,s]} \in (F^o)^{\otimes s}$ and $E\in \Tr^o(0,s)$,
\begin{align*}
\langle u,\exp(L(-1)^oz_r)Y_E^o(a_{[0,s]},z_{[0,s]}\rangle,
\end{align*}
is absolutely locally uniformly convergent to a holomorphic function on $U_{\tilde{E}} \subset X_{s}(\C)$.
Denote the restriction of this analytic function on $U_E^o \subset X_{0,s}(\HH)$ by $C_E(u,a_{[0,s]};z_{[0,s]})$.
%i.e., $\langle u,\exp(L(-1)z_s)Y_E^o(v_{[s]},z_{[s]}\rangle$ is absolutely locally uniformly convergent to a complex analytic function on $U_{\tilde{E}} \subset X_s(\C)$.
%Denote the restriction of this analytic function
%on the real points 
%\begin{align*}
%X_{0,s}(\HH) \subset X_s(\C)
%\end{align*}
%by $C_E(u,v_{[s]};z_{[s]})$,
%which is a real analytic function on $X_{0,s}(\HH)$.
\item[Compatibility]
There exists a family of linear maps
\begin{align*}
C_{0,s}:(F^o)^\vee\otimes (F^o)^{\otimes s} \rightarrow C^\om(X_{0,s}(\HH))\quad \text{ for }s\geq 2
\end{align*}
such that:
\begin{align}
C_s(u,a_{[0,s]};z_{[0,s]})\Bigl|_{U_E^o} =C_E(u,a_{[0,s]};z_{[0,s]})
\label{eq_cor_def_open}
\end{align}
for any $a_{[0,s]} \in (F^o)^{\otimes s}$, $u\in (F^o)^\vee$ and $E \in \Tr^o(0,s)$ as real analytic functions.
\end{description}
\end{dfn}

The following theorem follows from the same argument as in Theorem \ref{thm_bulk}:
\begin{thm}
Let $(F^o,Y^o,\va^o)$ with a chiral symmetry $V$ be locally $C_1$-cofinite.
Then, $F^o$ is consistent. 
\end{thm}
Note that $C_s$ does not have the symmetry of the permutation group $S_s$.

\subsection{Bulk-boundary OPE algebra}
\label{sec_bulk_boundary}

The algebra appearing on the bulk-boundary of 2d CFT is mathematically formulated in \cite{Ko1}, which is called an {\it open-closed field algebra}.
The necessary and sufficient conditions for constructing an open-closed algebra are described in \cite[Theorem 1.28]{Ko1} under assumptions similar to (but slightly stronger) the local $C_1$-cofiniteness in the previous section.

We reformulate Kong's approach in the following two respects:
Kong considered the vertex operator $Y_{\mathrm{cl-op}}(\bullet,z,\z):F^c \otimes F^o \rightarrow F^o[[z^\R,\z^\R]]$ as the basic building block of the theory, but it is actually sufficient to give a bulk-boundary operator $\tau_y: F^c \rightarrow F^o[[y^\R]]$.
In Theorem \cite[Theorem 1.28]{Ko1}, six conditions were assumed as sufficient conditions to construct an open-closed field algebra, but in this section we will see that five are sufficient, and that they correspond exactly to the generator of $\PaPB$ as a 2-operad. This five conditions are known as genus 0 boundary bootstrap equations in physics \footnote{Fig. (9.a) in \cite{Le} corresponds to both the bulk commutativity and associativity, which we count as two.}{\cite[Conditions (a), (c), (d), (e) in Fig. 9]{Le}}.

Then, we state and prove the consistency of the OPEs with respect to all trees $\Tr^o(r,s)$ by using the result in Section \ref{sec_homSC}.

Let $V$ be a positive graded vertex operator algebra.
Let $(F^o,Y^o,\va^o)$ be the triple in Definition \ref{def_bfullVA} with the $V$-chiral symmetry
and $(F^c,Y^c,\va^c,\om,\omb)$ a full vertex operator algebra.
A $V$-chiral symmetry on $F^c$ is a vertex algebra homomorphism preserving the conformal vectors $i_r:V \hookrightarrow \ker L^c(-1)$ and $i_l:V \hookrightarrow\ker \Ld^c(-1)$.
Note that for the chiral symmetry we specify the three embedings of $V$ into $F^o,F^c$.
Hence, for a vertex operator algebra automophism $g \in \Aut V$, we think the twisted $V$-chiral symmetry,
e.g., $i_r \circ g:V \hookrightarrow \ker L(-1)$, as different ones, which is important when one considers D branes \cite{M8}.

To distinguish the three actions of $V$, set
\begin{align*}
a^l(n) = i_l(a)(n,-1),\quad a^r(n) = i_r(a)(-1,n),\quad a^o(n) = i_o(a)(n)
\end{align*}
for $a\in V$ and $n\in \Z$ (see also Proposition \ref{prop_skew} and Proposition \ref{prop_open_sub}).
We also denote $\om(n+1,-1)_l, \omb(-1,n+1)_r$ by $L^l(n), L^r(n)$, respectively.
\begin{dfn}\label{def_bb_operator}
A {\it bulk-boundary vertex operator} on $(F^c,F^o)$ with the chiral symmetry $V$ is a linear map
\begin{align*}
Y^b(\bullet,z) :F^c \rightarrow F^o[[z^\R]],\; v\mapsto Y^b(v,z)=\sum_{r \in \R}B_r(v)z^{-r-1}
\end{align*}
such that:
\begin{enumerate}
\item[BBC1)]
For any $v\in F^c$, $Y^b(L^l(-1)v,z) =\frac{d}{dz}Y^b(v,z)$;
\item[BBC2)]
For any $a\in V$, $v \in F^c$ and $n\in \Z$,
\begin{align}
\begin{split}
a^o(n)Y^b(v,z) &= Y^b(a^r(n)v,z) + \sum_{k\geq 0}\binom{n}{k} Y^b( a^l(k)v, z )z^{n-k}\\
Y^b(a^l(n)v,z) &= \sum_{k\geq 0} \binom{n}{k}\left( a^o(n-k)Y^b(v,z)(-z)^{k} - Y^b(a^r(k)v,z)(-z)^{n-k}\right)
\label{eq_L0_bb}
\end{split}
\end{align}
\end{enumerate}
%
%\begin{enumerate}
%\item
%For any $v\in F^c$, $\tau_z(L^l(-1)v) =\frac{d}{dz}\tau_z(v)$;
%\item
%For any $a\in V$, $v \in F^c$ and $n\in \Z$,
%\begin{align}
%\begin{split}
%a^o(n)\tau_z(v) &= \tau_z(a^r(n)v) + \sum_{k\geq 0}\binom{n}{k} \tau_z( a^l(k)v)z^{n-k}\\
%\tau_z(a^l(n)v) &= \sum_{k\geq 0} \binom{n}{k}\left( a^o(n-k)\tau_z(v)(-z)^{k} - \tau_z(a^r(k)v)(-z)^{n-k}\right)
%\label{eq_L0_bb}
%\end{split}
%\end{align}
%\end{enumerate}
\end{dfn}
Let $Y(\bullet,z):F^c \rightarrow F^o[[z^\R]]$ be a bulk-boundary vertex operator on $(F^c,F^o)$ with the chiral symmetry $V$.
Applying $a =\om^o$ and $n=1$, we have:
\begin{align*}
L^o(0)Y^b(v,z)&= z\frac{d}{dz}Y^b(v,z)+ Y^b((L^l(0)+L^r(0))v,z),
\end{align*}
which implies that
\begin{align*}
B_r(v) \in (F^o)_{h+\h-r-1}
\end{align*}
for $v \in F_{h,\h}$ and $r\in \R$.
Hence, $Y^b(v,z) \in F^o((z^\R))$.
Note that if $F^c$ is a direct sum of tensor products of $V$-modules then the above condition is equivalent to saying that $Y^b(\bullet,z)$ is a $V$-module intertwining operator.

In the above definition, the holomorphic and anti-holomorphic parts of $F^c$ are treated asymmetrically. However, making them asymmetric is inherently unnatural.
Therefore, we introduce a bulk-boundary operator
\begin{align*}
\tau_y:F^c \rightarrow F^o[[y^\R]]
\end{align*}
to treat them symmetrically. As we will see in the next proposition, these concepts are equivalent, and from the standpoint of dealing with vertex operators, $Y^b(\bullet,z)$ is more convenient.

Let $\tau_y:F^c \rightarrow F^o[[y^\R]]$ be a linear map defined by
\begin{align}
\tau_y(v) = \exp(-iyL(-1)^o) Y^b(v,2iy),\quad\quad\quad (v\in F^c)
\label{eq_def_tau}
\end{align}
where $y$ is a formal variable and $Y^b(v,2iy)$ is defined as
\begin{align*}
Y^b(v,2iy) = \sum_{r\in \R} 2^r \exp\left(\frac{\pi i r}{2}\right) B_r(v)y^r,
\end{align*}
i.e., we choose the branch of $\log(i)$ as $\frac{\pi i}{2}$.
Since $Y^b(v,z) \in F^o((z^\R))$, \eqref{eq_def_tau} is well-define, that is, each coefficient of $y$ is a finite sum.
\begin{prop}\label{prop_equiv_bb}
Let $Y^b(\bullet,z):F^c \rightarrow F^o[[z^\R]]$ be a linear map.
Then, $Y^b(\bullet,z)$ satisfies Definition \ref{def_bb_operator} if and only if 
the associated map $\tau_y$ in \eqref{eq_def_tau} satisfies the following conditions:
\begin{enumerate}
\item
For any $v\in F^c$,
$\frac{d}{dy}\tau_y(v)= i\tau_y((L^l(-1)-L^r(-1))v)$
\item
For any $a\in V$, $v\in F^c$ and $n\in \Z$, $m \in \Z_{\geq 0}$
\begin{align*}
a^o(m) \tau_y(v) &= \sum_{k \geq 0}
\binom{m}{k} \tau_y\left (a^l(k)(iy)^{m-k}+a^r(k)(-iy)^{m-k})v\right)\\
\tau_y(a^l(n)v) &= 
\sum_{k \geq 0}\binom{n}{k} \left(a^o(n-k) \tau_y(a) (-iy)^{n-k}
- \tau_y(a^r(k)a) (-2iy)^{n-k}\right)\\
\tau_y(a^r(n)v) &= 
\sum_{k \geq 0}\binom{n}{k}\left( a^o(n-k) \tau_y(v) (+iy)^{n-k}
- \tau_y(a^l(k)v) (+2iy)^{n-k}\right).
\end{align*}
\end{enumerate}
\end{prop}

\begin{proof}
Let $Y^b(\bullet,z):F^c \rightarrow F^o[[z^\R]]$ be a bulk-boundary vertex operator on $(F^c,F^o)$ with the chiral symmetry $V$. By (BBC1) and (BBC2), we have 
\begin{align*}
\frac{d}{dy}\tau_y(v) &= \frac{d}{dy}\exp(-iyL(-1)^o) Y^b(v,2iy)\\
&= -i\exp(-iyL(-1)^o) L(-1)^o Y^b(v,2iy) + \exp(-iyL(-1)^o) \frac{d}{dy}Y^b(v,2iy)\\
&= -i \exp(-iyL(-1)^o) Y^b((L(-1)^l+L(-1)^r) v,2iy) + \exp(-iyL(-1)^o) 2i Y^b(L(-1)^l v,2iy)\\
&=i \exp(-iyL(-1)^o) Y^b((L(-1)^l-L(-1)^r) v,2iy)\\
&=i\tau_y((L(-1)^l-L(-1)^r)v).
\end{align*}
Recall that 
\begin{align}
\exp(L(-1)z)a^o(n)\exp(-L(-1)z)&=\sum_{k\geq 0}^n \binom{n}{k} a^o(k)(-z)^{n-k}\label{eq_P_trans}
\\
\exp(L(-1)z)a^o(-n-1)\exp(-L(-1)z)&= \sum_{k\geq 0} \binom{k}{n} a^o(-k-1)z^{k-n},\label{eq_Y_trans}
\end{align}
for any $a \in V$ and $n \geq 0$ \cite[Lemma 1.11]{M6}.
Hence, for $a\in V$ and $n\geq 0$, by (BBC2) and \eqref{eq_P_trans}, we have
\begin{align*}
\tau_y(a^l(n)v) &=\exp(-iyL(-1)^o) Y^b(a^l(n)v,2iy)\\
&=\exp(-iyL(-1)^o)\sum_{k \geq 0}\binom{n}{k} \left(a^o(n-k) Y^b(v,2iy) (-2iy)^{k}- Y(a^r(k)v, 2iy) (-2iy)^{n-k}\right)\\
&=\exp(-iyL(-1)^o)\sum_{k \geq 0}\binom{n}{k} \left(a^o(n-k) Y^b(v,2iy) (-2iy)^{k}- Y(a^r(k)v, 2iy) (-2iy)^{n-k}\right)\\
&=\exp(iyL^o(-1)) a^o(n)\exp(-2iyL^o(-1))  Y^b(v,2iy)
-\sum_{k \geq 0}\binom{n}{k} \tau_y (a^r(k)v) (-2iy)^{n-k}\\
&=\sum_{k \geq 0}\binom{n}{k} a^o(k)\tau_y (v) (-iy)^{n-k} - \tau_y (a^r(k)v) (-2iy)^{n-k}.
\end{align*}
Similarly, by \eqref{eq_Y_trans} and $\binom{k+n}{n}=(-1)^k \binom{-n-1}{k}$, we have
\begin{align*}
&\tau_y(a^l(-n-1)v) \\
&=\exp(-iyL(-1)^o) Y^b(a^l(-n-1)v,2iy)\\
&=\exp(-iyL(-1)^o)\sum_{k \geq 0}\binom{-n-1}{k} \left(a^o(-n-1-k) Y^b(v,2iy) (-2iy)^{k}- Y(a^r(k)v, 2iy) (-2iy)^{-n-1-k}\right)\\
&=\exp(iyL(-1)^o)a^o(-n-1) \exp(-2iyL^o(-1)) Y^b(v,2iy)\\
&-\sum_{k \geq 0}\binom{-n-1}{k} \exp(-iyL(-1)^o)Y(a^r(k)v, 2iy) (-2iy)^{-n-1-k}\\
&=\sum_{k \geq 0} \binom{-n-1}{k} \left(a^o(-n-1-k)\tau_y(v) (-iy)^{k} - \tau_y(a^r(k)v)(-2iy)^{-n-1-k}\right).
\end{align*}
Hence, we have the second equality in (2). The last equality follows similarly.
Finally, for any $m \geq 0$, we have
\begin{align*}
a^o(m) \tau_y(v) &= a^o(m)\exp(-iyL(-1)^o) Y^b(v,2iy)\\
&=\exp(-iyL(-1)^o) \sum_{k\geq 0}^n \binom{m}{k} a^o(k)(-iy)^{m-k} Y^b(v,2iy)\\
&=\exp(-iyL(-1)^o) \sum_{k\geq 0}^n \binom{m}{k} (-iy)^{m-k} \left(Y^b(a^r(k)v,2iy)+ \sum_{l \geq 0}\binom{k}{l}Y^b(a^l(l)v,2iy)(2iy)^{k-l}\right).
\end{align*}
Since for formal variable $x,y$
\begin{align*}
\sum_{k,l\geq 0}\binom{m}{k}\binom{k}{l} x^{m-k}y^{k-l} =(1+x+y)^m,
\end{align*}
we have
\begin{align*}
\sum_{k\geq 0}^m &\sum_{l \geq 0}\binom{m}{k}\binom{k}{l}Y^b(a^l(l)v,2iy) (-iy)^{m-k}(2iy)^{k-l}\\
&=\sum_{l\geq 0} \binom{m}{l}Y^b(a^l(l)v,2iy) (iy)^{m-l}.
\end{align*}
Hence, $a^o(m) \tau_y(v) = \sum_{k \geq 0}\binom{m}{k} \tau_y\left (a^l(k)(iy)^{m-k}+a^r(k)(-iy)^{m-k})v\right)$.
The same can be equally verified for the opposite direction.
\end{proof}
\begin{rem}
\label{rem_bb_symmetric}
%上記の recursive relation は、共形ブロックにおいて正則部分にあたる加群を上半平面の点$iy$に反正則部分を$-iy$においたときの、定義式\eqref{eq_ker1}に他ならない。
The above recursive relation is nothing but the defining formula of conformal block \eqref{eq_ker1} when the chiral module is inserted at the point $iy$ in the upper half-plane and the anti-chiral is at $-iy$ in the lower half-plane.
\end{rem}

Let $u \in (F^o)^*$, $a\in F^c$ and $b \in F^o$.
We will consider a correlation function with $a$ inserted at a point $z_1\in \uH$ and $b$ inserted at a point $z_2\in \R$ (see Fig \ref{fig_BB3_12} and \ref{fig_BB3_21}).
\begin{figure}[h]
\begin{minipage}[l]{.2\textwidth}
\centering
\begin{forest}
for tree={
  l sep=20pt,
  parent anchor=south,
  align=center
}
[
[[\bf{1}][$\bar{\bf{1}}$]]
[2]
]
\end{forest}
%\captionof{figure}{}
%\label{fig_2tree}
\end{minipage}
\begin{minipage}[l]{.25\textwidth}
%\begin{figure}[t]
%    \centering
    \includegraphics[width=2.5cm]{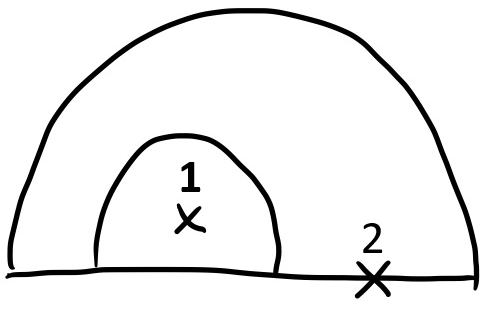}
%    \caption{$\alpha_c$}\label{fig_al_c}
%\end{figure}
    \caption{$\tau({\bf1})2$}\label{fig_BB3_12}
\end{minipage}
\begin{minipage}[r]{.2\textwidth}
\centering
\begin{forest}
for tree={
  l sep=20pt,
  parent anchor=south,
  align=center
}
[
[2]
[[\bf{1}][$\bar{\bf{1}}$]]
]
\end{forest}
%\captionof{figure}{}
%\label{fig_2tree}
\end{minipage}
\begin{minipage}[r]{.25\textwidth}
%\begin{figure}[t]
%    \centering
    \includegraphics[width=2.5cm]{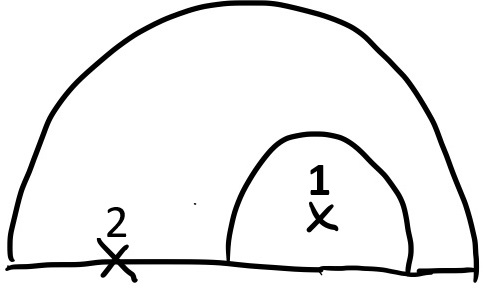}
    \caption{$2\tau({\bf1})$}\label{fig_BB3_21}
%\end{figure}
\end{minipage}
\end{figure}

 This correlation function is a single-valued real analytic function on $(z_1,z_2) \in \uH \times \R$.
Roughly speaking, the expansion of this correlation function in the domain $\{(z_1,z_2) \in \uH \times \R\mid \mathrm{Re}z_1 >z_2\text{ and }|z_1-\z_1|<|\z_1-z_2|\}$ (Fig \ref{fig_BB3_12}) is given by
\begin{align}
\langle u, \exp(z_2 L^o(-1))Y^o(Y^b(a,z_{1\bar{1}}),z_{\bar{1}2})b \rangle \in  \C\Bigl(\Bigl(\left(\frac{z_{1\bar{1}}}{z_{\bar{1}2}}\right)^\R\Bigr)\Bigr)[z_2,z_{\bar{1}2}^\R], \label{eq_BB_31}
\end{align}
and in the domain $\{(z_1,z_2) \in \uH \times \R\mid z_2 > \mathrm{Re}z_1 \text{ and } |z_1-\z_1|<|\z_1-z_2|\}$ (Fig \ref{fig_BB3_21}) is 
\begin{align}
\langle u, \exp(\z_{1}L^o(-1))Y^o(b,z_{2\bar{1}}) Y^b(a,z_{1\bar{1}})\rangle \in \C\Bigl(\Bigl(\left(\frac{z_{1\bar{1}}}{z_{2\bar{1}}}\right)^\R\Bigr)\Bigr)[\z_1,z_{2\bar{1}}^\R]. \label{eq_BB_32}
\end{align}
In order to consider them as analytic functions, the branch of $z^r=\exp( r\Log z)$ must be determined.
Recall that the branch of
\begin{align*}
\Log :\C^\cut = \C \setminus \R_{-} \rightarrow \C
\end{align*}
is taken so that $-\pi < \mathrm{Arg}(z) < \pi$. 

\eqref{eq_BB_31} and \eqref{eq_BB_32} are just formal series. We assume that they converge absolutely in $|z_{1\bar{1}}|<|z_{\bar{1}2}|$. Then, we obtain holomorphic functions on $U_{\tau({\bf1})2}, U_{2\tau({\bf1})} \subset X_3(\C)$, respectively.
Following \eqref{eq_UEo}, we substitute the complex number $z_1-z_{\bar{1}}$ to the variable $z_{1\bar{1}}$ and think of the vertex operator $Y^b(\bullet,z_{1\bar{1}})$ as follows:
\begin{align}
Y^b(a,z_{1\bar{1}}) &= \sum_{r\in \R} B_r(b) \exp(-(r+1)\Log(z_1-\z_1))\\
& = \sum_{r\in \R} B_r(b) (2\mathrm{Im}\,z_1)^{-r-1} \exp(-\frac{(r+1)\pi i}{2}) \label{eq_BB_branch}
\end{align}
and similar to $Y^o(\bullet,z_{\bar{1}2})$,
which uniquely determine the branches of \eqref{eq_BB_31} and \eqref{eq_BB_32}.
Hence, we obtain real analytic functions on $U_{\tau(1)2}^o, U_{2\tau(1)}^o \subset X_{1,1}(\HH)$.

Note that when taking the product of the boundary states, it is possible to get into the region $\R_- \subset \C$ which we cut for taking the branch, but this does not happen because we are taking the product and variables in the correct order when we expand the correlation function with respect to trees.

To state the definition of boundary bootstrap equation, we need to consider one more correlation function.
\begin{figure}[h]
\begin{minipage}[l]{.2\textwidth}
\centering
\begin{forest}
for tree={
  l sep=20pt,
  parent anchor=south,
  align=center
}
[
[[\bf{1}][$\bar{\bf{1}}$]]
[[\bf{2}][$\bar{\bf{2}}$]]
]
\end{forest}
%\captionof{figure}{}
%\label{fig_2tree}
\end{minipage}
\begin{minipage}[l]{.25\textwidth}
%\begin{figure}[t]
%    \centering
    \includegraphics[width=2.5cm]{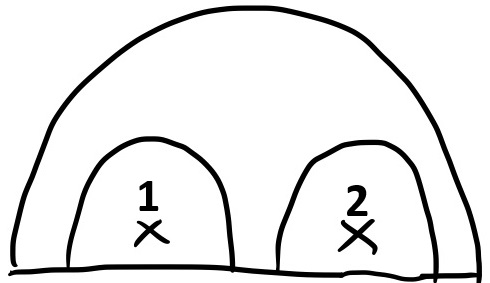}
%    \caption{$\alpha_c$}\label{fig_al_c}
%\end{figure}
    \caption{$\tau({\bf1})\tau({\bf2})$}\label{fig_BB2_a}
\end{minipage}
\begin{minipage}[r]{.2\textwidth}
\centering
\begin{forest}
for tree={
  l sep=20pt,
  parent anchor=south,
  align=center
}
[
[[\bf{1}][{\bf{2}}]]
[[$\bar{\bf{1}}$][$\bar{\bf{2}}$]]
]
\end{forest}
%\captionof{figure}{}
%\label{fig_2tree}
\end{minipage}
\begin{minipage}[r]{.25\textwidth}
%\begin{figure}[t]
%    \centering
    \includegraphics[width=2.5cm]{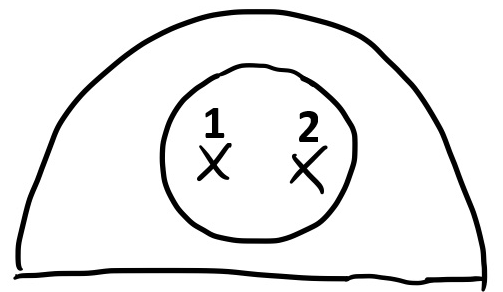}
    \caption{$\tau({\bf1 \bf 2})$}\label{fig_BB2_b}
%\end{figure}
\end{minipage}
\end{figure}

Let $u \in (F^o)^*$, $a_1,a_2 \in F^c$.
Fig \ref{fig_BB2_a} corresponds to the domain 
$\left|\frac{z_{1\bar{1}}}{z_{\bar{1}\bar{2}}}\right|+\left|\frac{z_{2\bar{2}}}{z_{\bar{1}\bar{2}}}\right|<1$ and $\mathrm{Re}\,z_1 > \mathrm{Re}\,z_2$, and the expansion is 
\begin{align}
\langle u, \exp(\z_2L^o(-1))Y^o(Y^b(a_1,z_{1\bar{1}}),z_{\bar{1}\bar{2}})Y^b(a_2,z_{2\bar{2}})\rangle
 \in  \C\Bigl(\Bigl(\left(\frac{z_{1\bar{1}}}{z_{\bar{1}\bar{2}}}\right)^\R, \left(\frac{z_{2\bar{2}}}{z_{\bar{1}\bar{2}}}\right)^\R\Bigr)\Bigr)[\z_2,z_{\bar{1}\bar{2}}^\R],
\label{eq_BB2_a}
\end{align}
and Fig \ref{fig_BB2_b} corresponds to the domain
$\left|\frac{z_{12}}{z_{2\bar{2}}}\right|+\left|\frac{\z_{12}}{z_{2\bar{2}}}\right|<1$
 and $\mathrm{Re}\,z_1 > \mathrm{Re}\,z_2$, and the expansion is 
\begin{align}
\langle u, \exp(\z_2L^o(-1)) Y^b(Y^c(a_1,z_{12},\z_{12})a_2,z_{2\bar{2}})\rangle
 \in  \C\Bigl(\Bigl(\left(\frac{z_{12}}{z_{2\bar{2}}}\right)^\R, \left(\frac{\z_{12}}{z_{2\bar{2}}}\right)^\R\Bigr)\Bigr)[\z_2,z_{2\bar{2}}^\R].
\label{eq_BB2_b}
\end{align}
We assume that \eqref{eq_BB2_a} converges absolutely in 
$\left|\frac{z_{1\bar{1}}}{z_{\bar{1}\bar{2}}}\right|+\left|\frac{z_{2\bar{2}}}{z_{\bar{1}\bar{2}}}\right|<1$
and \eqref{eq_BB2_b} converges absolutely in 
$\left|\frac{z_{12}}{z_{2\bar{2}}}\right|+\left|\frac{\z_{12}}{z_{2\bar{2}}}\right|<1$.
Thus, we obtain holomorphic functions on $U_{\tau(1)\tau(2)}, U_{\tau(12)} \subset X_4(\C)$, respectively.
By the restriction \eqref{eq_UEo}, we also obtain real analytic functions on $U_{\tau(1)\tau(2)}^o, U_{\tau(12)}^o \subset X_{2,0}(\HH)$.
By Fig \ref{fig_BB2_a} and \ref{fig_BB2_b}, it is clear that $U_{\tau(1)\tau(2)}^o \cap U_{\tau(12)}^o\neq \emptyset$.

The following assumption is called a {\it (boundary) bootstrap equation} in physics
(see also \cite[Theorem 1.28]{Ko1} and Remark 3.23):
\begin{dfn}\label{def_bulk_boundary}
We say that the bulk-boundary vertex operator $Y^b(\bullet,z)$ on $(F^c,F^o)$ with the chiral symmetry $V$ 
 satisfies the {\it bootstrap equation} when it satisfies the following conditions:
\begin{enumerate}
\item[BB1)]
$Y^b(\va^c,z)=\va^o$.
\item[BB2)]
For any $u \in (F^o)^*$, $a_1,a_2 \in F^c$,
\eqref{eq_BB2_a} converges absolutely in $\left|\frac{z_{1\bar{1}}}{z_{\bar{1}\bar{2}}}\right|+\left|\frac{z_{2\bar{2}}}{z_{\bar{1}\bar{2}}}\right|<1$ and \eqref{eq_BB2_b} converges absolutely in $\left|\frac{z_{12}}{z_{2\bar{2}}}\right|+\left|\frac{\z_{12}}{z_{2\bar{2}}}\right|<1$, and thus, define holomorphic functions on $U_{\widetilde{\tau(1)\tau(2)}}, U_{\widetilde{\tau(12)}} \subset X_{4}(\C)$, respectively.
Moreover, there is a single valued real analytic function $C_{2,0}(u,a_1,a_2)$ on $X_{2,0}(\HH)$ such that
the restriction of $C_{2,0}(u,a_1,a_2)$ on $U_{\tau(1)\tau(2)}^o \subset X_{2,0}(\HH)$ (resp. $U_{\tau(12)}^o \subset X_{2,0}(\HH)$) coincides with \eqref{eq_BB2_a} (resp. \eqref{eq_BB2_b}) for the branch specified above.
\item[BB3)]
For any $u \in (F^o)^*$, $a\in F^c$ and $b \in F^o$,
\eqref{eq_BB_31} and \eqref{eq_BB_32} converge absolutely in $|z_{1\bar{1}}|<|z_{\bar{1}2}|$,
and thus, define holomorphic functions on $U_{\widetilde{\tau(1)2}}, U_{\widetilde{2\tau(1)}} \subset X_3(\C)$, respectively.
Moreover, there is a single valued real analytic function $C_{1,1}(u,a,b)$ on $X_{1,1}(\HH)$ such that
the restriction of $C_{1,1}(u,a,b)$ on $U_{\tau(1)2}^o \subset X_{1,1}(\HH)$ (resp. $U_{2\tau(1)}^o \subset X_{1,1}(\HH)$) coincides with \eqref{eq_BB_31} (resp. \eqref{eq_BB_32}).
\end{enumerate}
\end{dfn}

%
%\begin{dfn}
%A bulk-boundary OPE prealgebra is a pair of a full vertex operator prealgebra $(F^c,Y(\bullet,\uz),\va^c,\om,\omb)$ and
%a boundary vertex operator prealgebra $(F^o,Y^o,\va^o,\nu)$ 
%equipped with a linear map
%\begin{align*}
%\iota_y: F^c \rightarrow F^o[[y^\R]],\quad
%a \mapsto \iota_y(a)=\sum_{r\in \R} \iota_r(a) y^{-r-1}
%\end{align*}
%such that:
%\begin{description}
%\item[vacuum)]
%$\iota_y(\va^c)=\va^o$.
%\item[associativity)]
%For any $a,b \in F^c$,
%\begin{align*}
%\iota_y(Y^c(a,\uz)b) = Y(\iota(a),x)\iota(b).
%\end{align*}
%\item[central)]
%For any $r\in \R$ and $a\in F^c$,
%\begin{align*}
%\iota_r(a) \in Z(F^o).
%\end{align*}
%\item[Virasoro symmetry)]
%For any $a\in F^c$ and $n\in \Z$,
%\begin{align*}
%L^o(n) \iota(a) &= \sum_{k \geq 0}
%\binom{n}{k} \iota\left (L^c(k)(iy)^k+\Ld^c(k)(-iy)^k)a\right)\\
%\iota(L^c(n)a) &= 
%\sum_{k \geq 0}\binom{n}{k} \left(L^o(n-k) \iota(a) (-iy)^{n-k}
%+ \iota(\Ld^c(k)a) (+2iy)^{n-k}\right)\\
%\iota(\Ld^c(n) a) &= 
%\sum_{k \geq 0}\binom{n}{k}\left(L^o(n-k) \iota(a) (+iy)^{n-k} + \iota(L^c(k)a) (-2iy)^{n-k}\right)
%\end{align*}
%\end{description}
%\end{dfn}

\begin{rem}
\label{rem_kong}
(BB2) corresponds to Associativity II in \cite[(1.56)]{Ko1}, and (BB3) corresponds to Commutativity I in \cite[Proposition 1.18]{Ko1}. Associativity I in \cite[(1.52)]{Ko1} follows from (BB2) and (BB3) under the assumption that the chiral symmetry is locally $C_1$ cofinite, as we will see below.
\end{rem}

\begin{rem}
\label{rem_non_com}
It is important to note that condition (BB3) does not necessarily mean that the image of $Y^b(\bullet,z)$ in $F^o$ is commutative.
The state of the image is commutative only when it goes around the upper half-plane, otherwise it would not be interchangeable. In a categorical-theoretic setting, this corresponds to considering the left / full center instead of the usual center of the algebra in the modular tensor category (see \cite{FFRS, KR}).
We can also see it explicitly in our next paper, which constructs the bulk-boundary operator of the Narain CFTs \cite{M8}.
\end{rem}

%
%\begin{dfn}\label{def_symmetry}
%Let $V$ be a vertex operator algebra. 
%We say a bulk-boundary full vertex operator prealgebra possesses $V$-chiral symmetry if there is a vertex operator algebra homomorphisms, $\phi_l:V \rightarrow \ker L^c(-1)$, $\phi_r:V \rightarrow \ker \Ld^c(-1)$
%and $\phi_o:V \rightarrow \ker Z(F^o)$ such that:
%\begin{align*}
%\phi_o(s)(n) \iota(a) &= \sum_{k \geq 0}
%\binom{n}{k} \iota\left (\phi_l(s)(k,-1)(iy)^k+\phi_r(s)(-1,k)(-iy)^k)a\right)\\
%\iota(\phi_l(s)(n,-1)a) &= 
%\sum_{k \geq 0}\binom{n}{k} \left(\phi_o(v)(n-k) \iota(a) (-iy)^{n-k}
%+ \iota(\phi_r(v)(-1,k)a) (+2iy)^{n-k}\right)\\
%\iota(\phi_r(s)(-1,n)a) &= 
%\sum_{k \geq 0}\binom{n}{k}\left( \phi_o(v)(n-k) \iota(a) (+iy)^{n-k}
%+ \iota(\phi_l(v)(k,-1)a) (-2iy)^{n-k}\right)
%\end{align*}
%for any $s \in V$ and $n \in \Z$.
%\end{dfn}
%Note that the Virasoro symmetry is a special case of the chiral symmetry.

Let $(F^o,F^c)$ possesses a $V$-chiral symmetry.
Let $M,\bar{M}$ be $V$-modules
and $M\otimes \bar{M} \hookrightarrow F^c$ a $V\otimes V$-module homomorphism.
Then, it is clear that the restriction of $Y^b(\bullet,z)$ on $M\otimes \bar{M}$ is an intertwining operator of type 
$\binom{F^o}{M \bar{M}}$.
%the following lemma follows from Definition \ref{def_symmetry}:
%\begin{lem}
%\label{lem_iota_int}
%$Y_\iota(\bullet,x)$ is an intertwining operator
%among $V$-modules of type 
%\end{lem}

\begin{dfn}\label{def_LC_BB}
We call a bulk-boundary vertex operator on $(F^c,F^o)$ with the chiral symmetry $V$ is
 {\it locally $C_1$-cofinite} if both $(F^c,Y^c,\va^c)$ and $(F^o,Y^o,\va^o)$ are locally $C_1$-cofinite over $V$ (Definition \ref{dfn_LC_bulk} and \ref{def_LC_boundary}) such that:
\begin{enumerate}
\item[bbLC)]
For each $i \in I_c$ in \eqref{dfn_LC_bulk}, there is a finite subset $B(i) \subset I_o$  such that:
\begin{align*}
Y^b(m_i\otimes \bar{m}_i, z) \in \bigoplus_{j \in B(i)}N_j((z^\R))
\end{align*}
with $I_o$ and $N_j$ in \eqref{def_LC_boundary}.
\end{enumerate}
%\begin{itemize}
%\item
%$\iota$ について何かを課すべき。
%\end{itemize}
\end{dfn}

Let $(F^o,F^c,Y^o,Y^c,Y^b)$ be a bulk-boundary operator which is $C_1$-cofinite over $V$
and satisfies the boundary bootstrap equation (Definition \ref{def_bulk_boundary}).
We will explain how to assign the iterated vertex operators $Y_E$ for a tree $E \in \Tr^o(r,s)$.
Let $a_{[r,s]} \in (F^c)^{\otimes r} \otimes (F^o)^{\otimes s}$.
As an example, we will consider the case where the tree is Fig \ref{fig_open_closed_tree}:

\begin{minipage}[l]{.6\textwidth}
\begin{forest}
for tree={
  l sep=20pt,
  parent anchor=south,
  align=center
}
[$z_{\bar{{\bf3}}\bar{{\bf2}}}$
[$z_{{\bf3}\bar{{\bf3}}}$[$z_{{\bf43}}$,edge label={node[midway,left]{$\tau$}}[[{\bf4}]][$z_{{\bf13}}$[{{\bf1}}][{{\bf3}}]]] [$z_{\bar{{\bf4}}\bar{{\bf3}}}$,edge label={node[midway,right]{$\tau$}}[[$\bar{{\bf4}}$]][$z_{\bar{{\bf1}}\bar{{\bf3}}}$[$\bar{{\bf1}}$][$\bar{{\bf3}}$]]]]
[$z_{{5}\bar{{\bf2}}}$[[{ 5}]]
[$z_{{\bf2}\bar{{\bf2}}}$[{{\bf2}},edge label={node[midway,left]{$\tau$}}][$\bar{ {\bf2}}$,edge label={node[midway,right]{$\tau$}}]]]
]
\end{forest}
\captionof{figure}{}
\label{fig_open_closed_tree}
\end{minipage}
\begin{minipage}[r]{.3\textwidth}
\begin{align*}
\tau\left({\bf4} \cdot_c ({\bf1} \cdot_c {\bf3})\right) \cdot_o \left(5 \cdot_o \tau({\bf 2})\right)\in \Tr_{4,1}^o
\end{align*}
\end{minipage}

We may assume that $a_i = a_i \otimes \bar{a}_i \in F^c$, where $a_i \otimes \bar{a}_i$ is taken from a direct summand of (LC2).
By assumption, 
we can locally decompose $Y^c = Y^l \otimes Y^r$ by (LC3), 
and we can denote $Y^b(a\otimes \bar{a},z)$ by $Y^b(a,z)\bar{a}$.
Then, for any $u\in (F^o)^\vee$,
\begin{align}
&\langle u,e^{L_o(-1)\z_{{2}}}Y^o\left(Y^b\left(Y^c(a_4\otimes \bar{a}_4,\uz_{43})Y^c(a_1\otimes \bar{a}_1,\uz_{13})a_3\otimes \bar{a}_{3}, z_{3\bar{3}}\right),z_{\bar{3}\bar{2}}\right) Y^o(a_5,z_{5\bar{2}})Y^b(a_2\otimes \bar{a}_2,z_{2\bar{2}})\rangle \label{eq_cor_bulk_boundary1}\\
&=\langle u,e^{L_o(-1)\z_{{2}}} 
Y^o\left( 
Y^b\left(Y^l(a_4,z_{43})Y^l(a_1,z_{13})a_3, z_{3\bar{3}}\right)Y^r(\bar{a}_4,z_{\bar{4}\bar{3}})Y^r(\bar{a}_1,z_{\bar{1}\bar{3}})\bar{a}_3,z_{3\bar{2}}\right)\label{eq_cor_bulk_boundary2}\\
& Y^o(a_5,z_{5\bar{2}})Y^b(a_2, z_{2\bar{2}} )\bar{a}_2)\rangle. \nonumber
\end{align}
The right-hand-side \eqref{eq_cor_bulk_boundary2} is a formal power series in $T_{\tilde{E}}$ and is absolutely convergent to a holomorphic function in $U_{\tilde{E}}$ by the assumption of local $C_1$-cofiniteness.
Then, the restriction of \eqref{eq_cor_bulk_boundary2} onto
\begin{align*}
U_E^o = \Phi^{-1}(\overline{U}_{\tilde{E}}) \cap X_{r,s}(\HH)
\end{align*}
is a well-defined real analytic function.
Denote the compositions of vertex operators \eqref{eq_cor_bulk_boundary1} by $Y_E$,
which in itself makes sense even if the vertex operator is not locally $C_1$ cofinite and does not decompose into chiral / anti-chiral parts by \eqref{eq_cor_bulk_boundary1}.
\begin{dfn}\label{def_pre_boundary}
Let $(F^c,F^o,Y^c,Y^o,Y^b,\va^c,\va^o)$ be in Definition \ref{def_bulk_boundary} with a chiral symmetry $V$.
We call it consistent if the following properties are satisfied:
\begin{description}
\item[Convergence]
For any $u\in (F^o)^\vee$ and $a_{[r,s]} \in (F^c)^{\otimes r} \otimes (F^o)^{\otimes s}$ and $E\in \Tr^o(r,s)$,
\begin{align*}
\langle u,\exp(L^o(-1)z_{E})Y_E(a_{[r,s]}; z_{[r,s]}\rangle,
\end{align*}
is absolutely locally uniformly convergent to a holomorphic function on $U_{\tilde{E}} \subset X_{2r+s}(\C)$.
Denote the restriction of this analytic function on $U_E^o$ by $C_E(u,a_{[r,s]};z_{[r,s]})$.
\item[Compatibility]
There exists a family of linear maps
\begin{align*}
C_{r}^c:(F_c)^\vee\otimes (F_c)^{\otimes r} \rightarrow C^\om(X_{r}(\C))
\end{align*}
and
\begin{align*}
C_{r,s}^o:(F_o)^\vee\otimes (F_c)^{\otimes r}\otimes (F_o)^{\otimes s} \rightarrow C^\om(X_{r,s}(\HH))
\end{align*}
such that:
\begin{align*}
C_r^c(u,a_{[r]};z_{[r]})\Bigl|_{U_A^c} =C_A(u,a_{[r]};z_{[r]})
\end{align*}
and
\begin{align}
C_{r,s}^o(u,a_{[r,s]};z_{[r,s]})\Bigl|_{U_E^o} =C_E(u,a_{[r,s]};z_{[r,s]})
\label{eq_cor_def_bb}
\end{align}
for any $A \in \Tr^c(r)$ and $E \in \Tr^o(r,s)$ as real analytic functions.
\end{description}
\end{dfn}

\begin{thm}\label{thm_bulk_boundary}
Let $(F^c,F^o,Y^c,Y^o,Y^b,\va^c,\va^o)$ be in Definition \ref{def_bulk_boundary} which is locally $C_1$-cofinite over a positive graded vertex operator algebra. Then, it is consistent.
\end{thm}
\begin{proof}
By Theorem \ref{thm_SC_action}, similarly to the proof of Theorem \ref{thm_bulk}, it suffices to show that (Compatibility) holds on the generator of $\PaPB$ as a 2-colored operad.

\begin{figure}[t]
  \begin{minipage}[c]{0.3\linewidth}
    \centering
    \includegraphics[width=1.5cm]{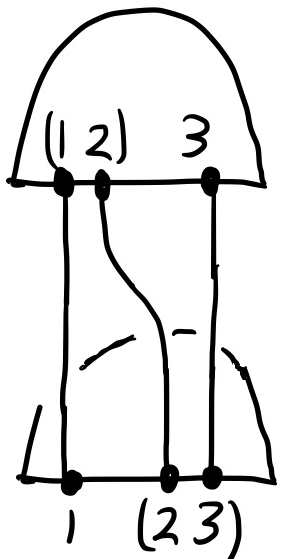}
    \caption{Path $\al_o$}\label{fig_al_o}
  \end{minipage}
  \begin{minipage}[c]{0.3\linewidth}
    \centering
    \includegraphics[width=1.5cm]{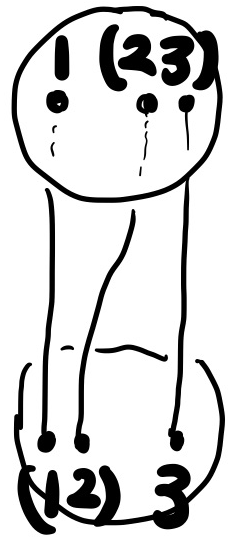}
    \caption{Path $\alpha_c$}\label{fig_al_c}
  \end{minipage}
      \begin{minipage}[c]{0.3\linewidth}
    \centering
    \includegraphics[width=1.3cm]{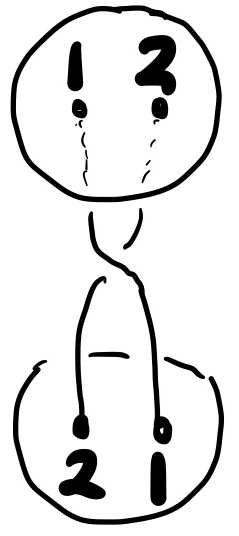}
    \caption{Path $\sigma$}\label{fig_sigma_l}
  \end{minipage}
    \begin{minipage}[c]{0.3\linewidth}
    \centering
    \includegraphics[width=1.5cm]{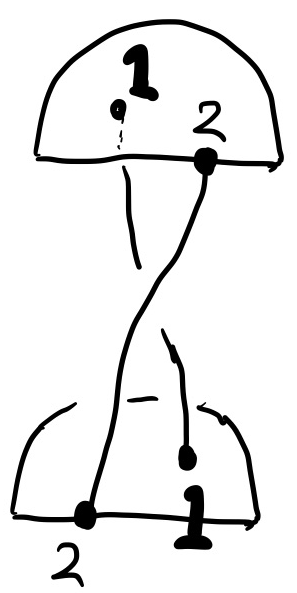}
    \caption{Path $p$}\label{fig_p}
  \end{minipage}
      \begin{minipage}[c]{0.3\linewidth}
    \centering
    \includegraphics[width=1.5cm]{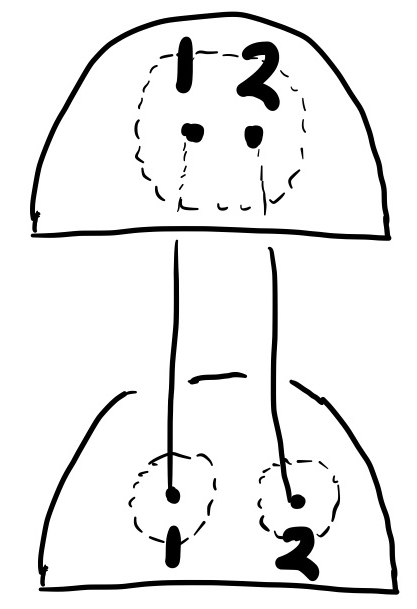}
    \caption{Path $q$}\label{fig_q}
  \end{minipage}
\end{figure}

In the proof of Theorem 4.2 in \cite{Id}, it was shown that $\PaPB$ is generated by the following five paths as a 2-colored operad:
\begin{align*}
\al_o&:(1 \cdot_o 2)\cdot_o 3 \rightarrow 1 \cdot_o (2\cdot_c3)\\
\al_c&:({\bf 1} \cdot_c {\bf 2})\cdot_c {\bf3} \rightarrow {\bf1} \cdot_c ({\bf2}\cdot_c {\bf3})\\
\si&: {\bf1}\cdot_c {\bf2} \rightarrow {\bf2}\cdot_c{\bf1}\\
p&: \tau({\bf1})\cdot_o 2 \rightarrow 2\cdot_o \tau({\bf1})\\
q&: \tau({\bf1}\cdot_c {\bf2}) \rightarrow \tau({\bf1}) \cdot_o \tau({\bf2}),
\end{align*}
which are given in Fig \ref{fig_al_o} - \ref{fig_q}.
By the bootstrap equations,
 the vertex operators $(Y^c,Y^o,Y^b)$ are invariant with respect to the analytic continuation along these paths.
Hence, by Theorem \ref{thm_SC_action}, the assertion holds.
\end{proof}

\appendix
\section{}
\label{sec_example_current}
Exactly marginal deformations of rational conformal field theories are generally not rational at all. However, in  Appendix, we will show that the current-current deformations of rational (bulk) conformal field theories always satisfy the local $C_1$-cofiniteness condition, and thus, are consistent.

%\subsection{Current-current deformations}\label{sec_example_current}
We call a full VOA $F$ {\it strongly rational} if $\ker \Ld(-1)$ and $\ker L(-1)$ are simple self-dual rational $C_2$-cofinite VOAs.
Assume in this section that $F$ is a strongly rational full VOA.
Then, by \cite{DM}, the degree one subspaces of $\ker \Ld(-1)$ and $\ker L(-1)$ are reductive Lie algebras. 
Let 
\begin{align*}
H_l \subset \ker \Ld(-1) \text{ and } H_r \subset \ker L(-1)
\end{align*}
be Cartan subalgebras and set $n_l =\dim H_l$ and $n_r =\dim H_r$, the ranks of the Lie algebras.
We think $H_l$ (resp. $H_r$) inherits a bilinear form $(-,-)$ by $h(1,-1)h' = (h,h')\va$ (resp. $h(-1,1)h' = (h,h')\va$).

Then, in \cite{M1}, we construct a deformation family of a full vertex operator algebra parametrized by a quotient of the orthogonal Grassmannian:
\begin{align}
D_F \backslash \mathrm{O}(n_l,n_r;\R)/  \mathrm{O}(n_l;\R)\times  \mathrm{O}(n_r;\R),
\label{eq_double_coset}
\end{align}
where $\mathrm{O}(n_l,n_r;\R)$ is the real orthogonal group with signature $(n_l,n_r)$.
The subgroup $D_F \subset \mathrm{O}(n_L,n_R;\R)$ is defined in \cite{M1} as an automorphism group of a generalized full VOA, which corresponds to the T-dulaity of string theory in the case of Narain CFTs  \cite{Polc}.
In this section, we will briefly review this result and show that they are locally $C_1$-cofinite 
while at a general point in \eqref{eq_double_coset}, the full VOA is not a rational CFT, i.e., $\ker \Ld(-1)$ and $\ker L(-1)$ are not rational VOAs.

Recall that for a VOA $V$ and a subset $S \subset V$, set
\begin{align*}
\mathrm{Com}_V(S) = \{a\in V\mid s(n)a =0 \text{ for any }s \in S,n \geq 0 \},
\end{align*}
which is called {\it a commutant} vertex algebra.

Let $M_{n_l}(0)$ and $M_{n_r}(0)$ be the subVOAs of $\ker \Ld(-1)$ and $\ker L(-1)$ generated by the Cartan subalgebras.
Set
\begin{align}
W &= \mathrm{Com}_{\ker \Ld(-1)}(M_{n_l}(0)),\quad W'= \mathrm{Com}_{\ker \Ld(-1)}(W)\\
\overline{W} &= \mathrm{Com}_{\ker L(-1)}(M_{n_r}(0)),\quad \overline{W}'= \mathrm{Com}_{\ker L(-1)}(\overline{W}).
\end{align}
Then, combining theorems on the structure of strongly rational vertex operator algebras \cite[Theorem 1]{Mas} and results from representation theory \cite{CKLR,CKM22}, the following proposition is obtained in \cite{HM23}:
\begin{prop}[Proposition 4.3 in \cite{HM23} and Theorem 1 in \cite{Mas}]
\label{prop_structure_regular}
Suppose $F$ is strongly rational full VOA. Then, the following properties are hold:
\begin{enumerate}
\item
There are positive-definite even lattices $L_l$ of rank $n_l$ (resp. $L_r$ of rank $n_r$) such that $W' \cong V_{L_l}$ and $\overline{W}' \cong V_{L_r}$ as VOAs.
\item
$W$ and $\overline{W}$ are also strongly rational and $\ker \Ld(-1)$ and $\ker L(-1)$ are
simple-current extensions of the strongly rational VOAs $W \otimes V_{L_l}$ and $\overline{W} \otimes \overline{V_{L_r}}$:
\begin{align*}
\ker \Ld(-1) &=\bigoplus_{\al \in A_l} W^\al \otimes V_{\al +L_l} \\
\ker L(-1) &=\bigoplus_{\be \in A_r} \overline{W}^\be \otimes \overline{V_{\be +L_r}},
\end{align*}
for some subgroup $A_l \subset L_l^\vee/L_l$ and $A_r \subset L_r^\vee/L_r$.
\end{enumerate}
\end{prop}

For $\ga \in H_l \oplus H_r$, set
\begin{align*}
\Om_F^\ga =\{v \in F\mid &h_l(n,-1)v =0, h_r(-1,n)v =0, \\
&h_l(0,-1)v =(h_l,\ga)v, h_r(-1,0)v =(h_r,\ga)v \text{ for any }h_l \in H_l,h_r\in H_r, n \geq 1\},
\end{align*}
the lowest weight space of the affine Heisenberg Lie algebra.
Then, by \cite[Theorem 5.3]{M1}, $\Om_F=\bigoplus_{\ga \in H_l \oplus H_r}\Om^\ga$ inherits a structure of a generalized full VOA,
Let $D_F \subset \mathrm{O}(H_l\oplus -H_r)\cong \mathrm{O}(n_l,n_r;\R)$ denote the automorphism group of the generalized full VOA $\Om$,
where $\mathrm{O}(n_l,n_r;\R)$.
Then, we construct a family of full VOAs that continuously deforms the eigenvalues $\ga$ (called charges)
parametrized by \eqref{eq_double_coset}  (for more precise statement, see \cite[Section 6.2]{M1}).
This family is called a {\it current-current deformation} of conformal field theory in physics.
\begin{thm}
Let $F$ be a strongly rational full VOA. Then, at any point in the current-current deformations of $F$ parametrized by
\begin{align*}
D_F \backslash \mathrm{O}(n_l,n_r;\R)/  \mathrm{O}(n_l;\R)\times  \mathrm{O}(n_r;\R),
\end{align*}
the full VOA is locally $C_1$-cofinite. In particular, it is consistent.
\end{thm}
\begin{proof}
Let $\tau \in D_F \backslash \mathrm{O}(n_l,n_r;\R)/  \mathrm{O}(n_l;\R)\times  \mathrm{O}(n_r;\R)$ and $F_\tau$ be the corresponding full VOA. In general, $\ker \Ld(-1)|_{F_\tau}$ and $\ker L(-1)|_{F_\tau}$ are no longer strongly rational VOA, however, they always contain subVOAs
\begin{align*}
V&= M_{n_l}(0) \otimes W \subset \ker \Ld(-1)|_{F_\tau}\\
\overline{V}&= \overline{M_{n_r}(0)} \otimes \overline{W} \subset \ker L(-1)|_{F_\tau}.
\end{align*}
By construction, $F_\tau$ is the direct sum of the $C_1$-cofinite modules of $V \otimes \overline{V}$ since $W$ and $\overline{W}$ are strongly rational by Proposition \ref{prop_structure_regular}.
Since the deformation modifies the intertwining operators of Heisenberg VOAs, (LC3) holds.
\end{proof}

\noindent
\begin{center}
{\bf Acknowledgements}
\end{center}

The author would like to thank Yoh Tanimoto, Maria Stella Adamo, Masahito Yamazaki, Yuji Tachikawa, Tomoyuki Arakawa, Mayuko Yamashita, Tomohiro Asano, Thomas Creutzig, Shuhei Ohyama and Kan Kitamura for valuable discussions.
%for letting me know about the conjectures in \cite{CG} and valuable discussions
%and to Yuki Arano for discussions on quantum groups,
%and to Makoto Yamashita
%and Hironori Oya for giving me the references.
%I would also like to thank Tomoyuki Arakawa and Thomas Creutzig
%for valuable comments.
This work is supported by Grant-in Aid for Early-Career Scientists (24K16911) and FY2023 Incentive Research Projects (Riken).

\end{document}